\pdfoutput=1
\documentclass[reqno]{amsart}

\usepackage{hyperref, bbm, amssymb, mathtools, url, enumerate}
\usepackage{amsmath}
\usepackage{tikz}
\usepackage{tikz-cd}
\usepackage[scr]{rsfso}
\usepackage{bm}
\usepackage{amssymb}
\usepackage{bbm}
\usetikzlibrary{matrix,arrows,decorations.pathmorphing, decorations.markings, cd}
\usepackage[nameinlink, capitalise, compress]{cleveref}
\usepackage{macros,stmaryrd}
\crefname{diag}{Diagram}{Diagrams}
\creflabelformat{diag}{(#2\textup{#1})#3}

\usepackage{rotate}
\usepackage[alphabetic]{amsrefs}
\usepackage{comment}
\usepackage{microtype}

\DeclareMathAlphabet{\mymathbb}{U}{BOONDOX-ds}{m}{n}
\def\twocell[#1]{\arrow[#1, dash, phantom, "\Rightarrow"{scale=1.125, yshift=-.4pt, description, allow upside down, sloped, inner sep=0pt}]}

\tikzset{curve/.style={settings={#1},to path={(\tikztostart)
			.. controls ($(\tikztostart)!\pv{pos}!(\tikztotarget)!\pv{height}!270:(\tikztotarget)$)
			and ($(\tikztostart)!1-\pv{pos}!(\tikztotarget)!\pv{height}!270:(\tikztotarget)$)
			.. (\tikztotarget)\tikztonodes}},
	settings/.code={\tikzset{quiver/.cd,#1}
		\def\pv##1{\pgfkeysvalueof{/tikz/quiver/##1}}},
	quiver/.cd,pos/.initial=0.35,height/.initial=0}

\pgfdeclarelayer{bg}
\pgfsetlayers{bg,main}
\usetikzlibrary{calc}

\theoremstyle{definition}

\newtheorem{introthm}{Theorem}

\newtheorem{theorem}{Theorem}[section]

\newtheorem{corollary}[theorem]{Corollary}
\newtheorem{lemma}[theorem]{Lemma}
\newtheorem{proposition}[theorem]{Proposition}

\newtheorem*{claim*}{Claim}
\newtheorem{construction}[theorem]{Construction}

\newtheorem{definition}[theorem]{Definition}

\newtheorem{example}[theorem]{Example}

\newtheorem{notation}[theorem]{Notation}

\newtheorem{remark}[theorem]{Remark}
\newtheorem*{remark*}{Remark}
\newtheorem{warning}[theorem]{Warning}

\newcommand{\PreCat}[1]{\Cat^\ast_{#1}}
\newcommand{\PreOp}[2]{\Op^{#2,\ast}_{#1}}
\newcommand{\PreNmCat}[2]{\Mack_{#1}^{#2,\ast}(\Cat)}
\newcommand{\NmCat}[2]{\Mack_{#1}^{#2}(\Cat)}

\newcommand{\Nstr}{\cN\textup{-}\otimes}

\newcommand{\PCat}[1]{\Cat_{#1}}
\newcommand{\POp}[2]{\Op_{#1}^{#2}}

\title{Global Picard Spectra and Borel Parametrized Algebra}
\author{Phil Pützstück}
\address{Phil Pützstück, FB Mathematik und Informatik, Universität Münster, Einsteinstraße 62, 48149 Münster, Germany}

\begin{document}

\begin{abstract}
    We answer a question of Schwede
    on the existence of global Picard spectra associated to his ultra-commutative global ring spectra; given an ultra-commutative global ring spectrum $R$, we show there exists a global spectrum $\pic_\eq(R)$ assembling the Picard spectra of all underlying
    $G$-equivariant ring spectra $\res_G R$ of $R$ into one object,
    in that for all finite groups $G$, the genuine fixed points are given by $\pic_\eq(R)^G \simeq \pic(\Mod_{\res_G R}(\Sp_G))$.

    Along the way, we develop a generalization of Borel-equivariant objects
    in the setting of parametrized higher algebra.
    We use this to assemble the symmetric monoidal categories of $G$-spectra
    for all finite groups $G$ together with all restrictions and norms into a single `normed global category',
    and build a comparison functor which allows us to import
    ultra-commutative $G$-equivariant or global ring
    spectra into the setting of parametrized higher algebra.
\end{abstract}

\begingroup\parskip=0pt
\maketitle
\tableofcontents
\endgroup

\section{Introduction}

\subsection*{Picard Spectra}

The Picard group of a symmetric monoidal category is an interesting invariant capturing information
about objects which are invertible with respect to the monoidal structure.
Classically, this is applied to categories of vector bundles on geometric spaces,
giving a group of line bundles.

In the context of higher category theory,
a symmetric monoidal ($\infty$-)category\footnote{In this paper, we will refer to $(\infty,1)$-categories simply as \emph{categories}.} $\cC$ has an underlying
$\E_\infty$-monoid $\cC^\simeq$ in spaces,
and taking units we obtain an $\E_\infty$-group $(\cC^\simeq)^\times$
which is often called the Picard space of $\cC$.
By May's recognition principle, we may equivalently
consider this as a connective spectrum,
the so-called Picard spectrum $\pic(\cC)$ of $\cC$.
Most of the interesting information lies in the Picard group $\pi_0\pic(\cC)$,
however it is often convenient to consider the entire Picard spectrum, as it has better categorical properties;
the induced functor $\pic \colon \CMon(\Cat) \to \Sp_{\geq 0}$ preserves limits,
and is amenable to descent-theoretic methods \cite{Mathew-Stojanoska}.

In the realm of stable homotopy theory,
the first example to consider is that of the category of spectra $\Sp$;
one checks that the only invertible spectra are shifts of the sphere,
and the Picard group is thus isomorphic to $\Z$.
More generally, one often considers $\E_\infty$-rings $R \in \CAlg(\Sp)$,
and the Picard groups/spectra of their symmetric monoidal
module categories $\pic(R) \coloneqq \pic(\Mod_R(\Sp))$
to learn about $\tensor_R$-invertible $R$-modules.
Interesting examples reflecting Bott-Periodicity are $\pi_0\pic\KU \cong \Z/2$ and $\pi_0\pic \KO \cong \Z/8$,
see \cite{Mathew-Stojanoska}.
Using the functoriality of $\Mod_R(\Sp)$ in $R$, we can build the functor\footnote{Since the categories $\Mod_R(\Sp)$ are presentable, it follows that $\pic(\Mod_R(\Sp))$ is always small and we can ignore size issues here.}
\begin{equation}\label{eq:pic-composite}
    \pic \colon \CAlg(\Sp) \xto{\Mod_{(-)}(\Sp)} \CMon(\Cat) \xto{\pic} \Sp.
\end{equation}
This project started with a question of Schwede
of whether the above can be generalized to the setting of equivariant and global stable homotopy theory.
To state our positive answer to this, let us introduce some notation.

For a finite group $G$, we write $\Sp_G$ for the (symmetric monoidal)
category of (genuine) $G$-equivariant spectra.
Given a commutative algebra $R \in \CAlg(\Sp_G)$, we have a Picard spectrum
$\pic(\Mod_{\res^G_H R}(\Sp_H))$ for each subgroup $H \leq G$,
and we can ask in what sense this system of spectra can itself
by assembled into the genuine fixed points of some $G$-spectrum.
In \cite[Remark 5.1.18]{Global} Schwede gives evidence that this should be possible
for ultra-commutative ring spectra,
which are simply strictly
commutative algebras in some suitable model category of $G$-equivariant spectra,
such as $G$-orthogonal or $G$-symmetric spectra.
We define $\UComm_G$ as the ($\infty$-)category of such, by Dwyer--Kan
localizing at the underlying $G$-equivariant weak equivalences.

Similarly, we will consider the category of global spectra $\Sp_\gl$
(for finite groups) as introduced by Schwede in \cite{Global},
and define the category ultra-commutative global ring spectra $\UCom_\gl$
as the Dwyer--Kan localization of the category of strictly commutative algebras
in a suitable model of spectra at the underlying global weak equivalences.
A global spectrum $X$ has an underlying $G$-spectrum $\res_G X$
and a genuine fixed points spectrum $X^G$ for every finite group $G$.

The following theorem then gives a positive answer to Schwede's question:

\begin{introthm}[{{Constructions \ref{con:pic-g} and \ref{con:pic-eqv}, and \cref{lem:pic-g-vs-eq}}}]\label{introthm:pic}
    ~\begin{enumerate}
        \item For a fixed finite group $G$ there is a functor $\pic_G \colon \UCom_G \to \Sp_G$ such that
            \begin{equation}\label{eq:g-pic}
                \pic_G(R)^H \simeq \pic(\Mod_{\res^G_H R}(\Sp_H)),\quad H \leq G.
            \end{equation}
            Moreover, there are natural equivalences $\res^G_H \pic_G \simeq \pic_H \res^G_H$.

        \item There is a functor $\pic_\eq \colon \UCom_\gl \to \Sp_{\gl}$ such that
            \[
                \pic_\eq(R)^G \simeq \pic(\Mod_{\res_G R}(\Sp_G))
            \]
            for all finite groups $G$.
            Also $\res_G \pic_\eq(R) \simeq \pic_G(\res_G R)$ naturally in $R$.
    \end{enumerate}
\end{introthm}

In fact, the categories $\Sp_G$ and $\Sp_\gl$ admit the common
generalization of $G$-global spectra $\Sp_{G\dgl}$ as introduced by Lenz \cite{Lenz-phd},
and we also construct an analogue $\UCom_{G\dgl} \to \Sp_{G\dgl}$ of the above functors in this case, which is moreover natural in $G$.
For $G=1$, this specializes to the global case above,
and in general there is a symmetric monoidal left and right Bousfield localization $i^* \colon \Sp_{G\dgl} \to \Sp_G$ with left adjoint $i_!$,
which allows us to precisely formulate the fact that `$G$-equivariant
and $G$-global stable homotopy theory have the same invertible objects'
as natural equivalences (c.f.~\cref{ex:pic-eq-gl}):
\[
    i_! \colon \pic(\Mod_R(\Sp_G)) \simeq \pic(\Mod_{i_!R}(\Sp_{G\dgl})),
\]
\[
    i^* \colon \pic(\Mod_{S}(\Sp_{G\dgl})) \simeq \pic(\Mod_{i^*S}(\Sp_{G})).
\]
for $R \in \CAlg(\Sp_G)$ and $S \in \CAlg(\Sp_{G\dgl})$ (no ultra-commutativity required).

A natural way to approach \cref{introthm:pic} for a finite group $G$
is by remembering the following result of Guillou--May \cite{Guillou-May}
(see also \cite[Appendix A]{CMNN} for the modern version we present here):
there is an equivalence $\Sp_G \simeq \Mack_G(\Sp) \coloneqq \Fun^\times(\Span(G),\Sp)$
between $G$-spectra and spectral $G$-Mackey functors, i.e.~spectrum-valued finite product preserving functors on the category $\Span(G) \coloneqq \Span(\F_G)$ of spans \cite{BGS-I} of finite $G$-sets.
This equivalence identifies the genuine $H$-fixed points $(-)^H \colon \Sp_G \to \Sp$
with evaluating $G$-Mackey functors at the transitive finite $G$-set $G/H$.
This perspective suggests that to generalize (\ref{eq:pic-composite})
to a functor $\pic_G$ satisfying (\ref{eq:g-pic}),
we want to assemble the symmetric monoidal categories $(\Mod_{\res^G_HR}(\Sp_H))_{H \leq G}$ into a Mackey functor $\Span(G) \to \CMon(\Cat)$ and then postcompose with $\pic$.
By semiadditivity of $\Span(G)$, such a functor is equivalently a categorical $G$-Mackey functor $\Span(G) \to \Cat$.
The idea is that the transfers of such a Mackey functor should come from composites
\[
    \Mod_{\res^G_KR}(\Sp_K)
    \xto{N_K^H} \Mod_{N_K^H\res^G_KR}(\Sp_H)
    \to \Mod_{\res^G_H R}(\Sp_H)
\]
where $N_K^H \colon \Sp_K \to \Sp_H$ denotes the symmetric monoidal
Hill--Hopkins--Ravenel norm \cite{HHR} and the second functor
basechanges along `normed multiplication maps' $\Nm_K^H\res^G_KR \to \res^G_HR$
that are present on ultra-commutative $G$-ring spectra $R$,
lifting equivariant power operations present on $\und{\pi}_0R$.

One perspective on ultra-commutativity which is amenable to building such categorical Mackey functors of module categories comes from parametrized category theory; with this Ansatz, we are left with the following two tasks:
\begin{enumerate}
    \item[($\dagger$)] build categorical Mackey functors of module categories for the `normed algebras' studied in parametrized higher algebra.

    \item[($\star$)] build a comparison functor which lets us view ultra-commutative ring spectra as normed algebras in certain `normed categories' of equivariant or global spectra.
\end{enumerate}

We elaborate on these points below.

\subsection*{Parametrized Higher Category Theory \& Parametrized Higher Algebra}

The field of parametrized higher category theory \cite{BDGNS}
was inspired by the perspective on equivariant stable homotopy theory
developed by Hill--Hopkins--Ravenel in their landmark solution of the Kervaire Invariant One Problem \cite{HHR},
which centers around the study of indexed (co)products and indexed tensor products
(incorporating the symmetric monoidal norms).
In $G$-equivariant category theory, the basic objects of study are functors $\Orb_G^\op \to \Cat_\infty$
known as $G$-categories, where $\Orb_G$ denotes the category of transitive $G$-sets for a finite group $G$.
The main philosophy is that when studying categories of $G$-equivariant objects such as $\Sp_G$,
we should consider them as coming together with the corresponding categories of $H$-equivariant
objects $\Sp_H$ for all $H \leq G$, as well as all their restriction functors $\res^H_K \colon \Sp_H \to \Sp_K$.
This defines the so-called $G$-category of $G$-spectra
$\und{\Sp}_G \colon \Orb_G^\op \to \Cat$.
For the interplay with parametrized higher algebra, it will often be convenient
to limit-extend and view $G$-categories as finite product preserving functors $\F_G^\op \to \Cat$.

A normed structure on a $G$-category $\cC \colon \F_G^\op \to \Cat$ is then an extension to a $G$-Mackey functor
(also known as a $G$-symmetric monoidal $G$-$\infty$-category \cite{Nardin-Shah}):
\[
    \cC^\tensor \colon \Span(G) \to \Cat.
\]
This encodes a symmetric monoidal structure on $\cC(H)$ and each restriction functor
$\res^H_K \colon \cC(H) \to \cC(K)$, as well as symmetric monoidal `norm' functors
$\cC^\tensor(K = K \to H) \colon \cC(H) \to \cC(K)$
which lift the $|H/K|$-fold tensor products, for each $K \leq H \leq G$ (as well as higher coherences).
For example, there is a normed $G$-category of $G$-spectra $\und{\Sp}_G^\tensor$
extending $\und{\Sp}_G$, which encodes the usual symmetric monoidal structure on each $\Sp_H$,
and sends forwards maps $K = K \to H$ in $\Span(G)$ to the symmetric monoidal Hill--Hopkins--Ravenel norms $N_K^H \colon \Sp_K \to \Sp_H$.
That an equivariant symmetric monoidal structure
should incorporate these norms
was already observed by Hill--Hopkins \cite{Hill-Hopkins-Closure,Hill-Hopkins-Gsym}.

A normed algebra in a normed $G$-category $\cC^\tensor$ is defined as a section of the cocartesian unstraightening $\int \cC^\tensor \to \Span(G)$ which is cocartesian on backwards maps.
This definition is originally due to Bachmann--Hoyois \cite{Bachmann-Hoyois},
and can be unraveled to encode a version of the `normed multiplications'
we mentioned above.

All of the above also works in the setting of global homotopy theory.
Here, a normed global category $\cC^\tensor$ consists of a finite product preserving functor
\[
    \cC^\tensor \colon \Span_\Forb(\Fglo) \to \Cat
\]
where the domain denotes the category of spans in the category $\Fglo$ of finite 1-groupoids
of the form $K \leftarrow H \to G$, where the forwards map $H \to G$ lies in the wide subcategory $\Forb \subset \Fglo$
on the faithful functors. Equivalently, $\Fglo \simeq \F[\Glo]$ is the free finite coproduct completion
of the $(2,1)$-category $\Glo$ of finite groups, group homomorphisms and conjugations,
and likewise $\FOrb \simeq \F[\Orb]$ where $\Orb \subset \Glo$ is the wide subcategory on injective group homomorphisms.
One can show that each $\und{\Sp}_G^\tensor$ is actually restricted from a single normed global category
$\und{\Sp}^\tensor$ of equivariant spectra sending $G \mapsto \Sp_G$,
which we construct in this paper.
We will also build a normed global category of global spectra $\und{\Sp}_\gl^\tensor$,
which sends $G$ to the category of $G$-global spectra $\Sp_{G\dgl}$ mentioned above.

Even more generally, consider a category $\cF$ equipped with a wide subcategory $\cN \subset \cF$.
Under suitable hypotheses on the pair $(\cF,\cN)$ (it is an extensive span pair in the sense of \cite{Tambara})
which are satisfied by $(\F_G,\F_G)$ and $(\FGlo,\Forb)$,
one has a good theory of $\cN$-normed $\cF$-categories and $\cN$-normed algebras
specializing to the above.
In this setting, we give a highly coherent construction of normed module categories,
taking care of the task $(\dagger)$ mentioned above.

\pagebreak
\begin{introthm}[{{\cref{thm:mod,rem:mod-functoriality}}}]
    For $\cC$ an $\cN$-normed $\cF$-category
    compatible with geometric realizations
    and $R \in \CAlg_{\cF}^\cN(\cC)$ a normed algebra.
    Then there exists an $\cN$-normed $\cF$-category
    of modules $\und{\Mod}_R(\cC)$ which sends a span $A \xleftarrow{f} B \xto{n} C$ to the composite
    \[
        \Mod_{RA}(\cC(A))
        \xto{f^*} \Mod_{RB}(\cC(B))
        \xto{n_\tensor} \Mod_{n_\tensor RB}(\cC(C))
        \xto{RC \tensor_{(n_\tensor RB)}(-)} \Mod_{RC}(\cC(C)).
    \]
    where $f^* = \cC(A \xleftarrow{f} B = B)$ and $n_\tensor = \cC(B = B \xto{n} C)$.
    Moreover, this is functorial in both $R$ and $\cC$ in the expected way, c.f.~\cref{rem:mod-functoriality}.
\end{introthm}

We are therefore left with solving the second task ($\star$) mentioned above.

\subsection*{Borel Parametrized Algebra}

A common occurrence in equivariant homotopy theory
is to embed the `naive' theory of spaces or spectra with $G$-action,
i.e.~the categories $\Fun(BG,\An)$ or $\Fun(BG,\Sp)$,
into their `genuine' counterparts $\An_G \simeq \Fun(\Orb_G^\op,\An)$
and $\Sp_G \simeq \Fun^\times(\Span(G),\Sp)$
by means of right Kan extending along the fully faithful
inclusions $BG \subseteq \Orb_G^\op$ or $BG \to \Span(G)$.
Objects in the image of the resulting fully faithful inclusions
$\An^{BG} \subseteq \An_G$ and $\Sp^{BG} \subseteq \Sp_G$
are called Borel, and we will refer to this process of turning
an object with $G$-action into a genuinely $G$-equivariant one
as Borelification.

As it turns out, an analogous strategy also works for normed categories.
In complete generality, we introduce the notion of a Borel-inclusion $(\cE,\cM) \subseteq (\cF,\cN)$
of extensive span pairs, which induces a fully faithful inclusion $\Span_\cM(\cE) \subseteq \Span_\cN(\cF)$
along which we can right Kan extend $\cM$-normed $\cE$-categories to $\cN$-normed $\cF$-categories.

Examples include the $G$-equivariant case $(\F[BG],\F[BG]) \subseteq (\F_G,\F_G)$
as well as the global case $(\F,\F) \subseteq (\FGlo,\FOrb)$.
Here $\F$ is the category of finite sets, and $\F[BG] \subseteq \F_G$
is the free finite coproduct completion of $BG$, or equivalently the full subcategory of $\F_G$ on free finite $G$-sets.
Global Borelification defines functors
\[
    (-)^\flat \colon \Cat \hookrightarrow \Cat_\gl = \Fun^\times(\Fglo^\op,\Cat)
    \quad\text{and}\quad
    (-)^\flat \colon \CMon(\Cat) \hookrightarrow \Mack_\gl(\Cat)
\]
where $\Mack_\gl(\Cat)$ denotes the category of normed global categories,
i.e.~of global Mackey functors $\Span_\Forb(\Fglo) \to \Cat$.
Concretely (c.f.~\cref{ex:borel-global}), if $\cC^\tensor$ is a symmetric monoidal category, then $\cC^{\tensor,\flat}$
is a normed global category $\cC^{\tensor,\flat}$ which sends $G$ to $\cC^{BG}$ equipped with the pointwise
symmetric monoidal structure, and forwards subgroup inclusions $H = H \hookrightarrow G$
to certain twisted tensor products $\cC^{BH} \to \cC^{BG}$ lifting the $|G/H|$-fold tensor product.

Our main theorem in this setting is actually proven in the generality of arbitrary Borel-inclusions as detailed above,
and slightly extends a theorem of Hilman \cite[Theorem A]{Hilman-NMot} in the equivariant case, but we will only present the global version here.
To state it, let us define normed $G$-global algebras in $\cC^\tensor$
as maps $\Span_\Forb(\Fglo_{/G}) \to \int\cC^\tensor$ over $\Span_\Forb(\Fglo)$,
and use the functoriality in $G$ to make this into a global category
$\und{\CAlg}_\gl(\cC^\tensor)$ which for $G=1$ recovers $\CAlg_\gl(\cC^\tensor)$
from above.

\pagebreak
\begin{introthm}[{{\cref{prop:borel-ran-img,prop:borel-sections}}}]\label{introthm:borel}
    ~
    \begin{enumerate}
        \item There is a pullback square
            \[\begin{tikzcd}[cramped]
                {\CMon(\Cat)} & \Cat \\
                {\Mack_{\gl}(\Cat)} & {\Cat_\gl}
                \arrow["\fgt"', from=1-1, to=1-2]
                \arrow["{(-)^\flat}"', hook, from=1-1, to=2-1]
                \arrow["\lrcorner"{anchor=center, pos=0.125}, draw=none, from=1-1, to=2-2]
                \arrow["{(-)^\flat}", hook', from=1-2, to=2-2]
                \arrow["\fgt", from=2-1, to=2-2]
            \end{tikzcd}\]
            In particular, a normed global category is Borel precisely if its underlying global category
            is Borel, and global Borelification induces an equivalence between the spaces
            of symmetric monoidal structures on a category $\cC$ and of normed structures on $\cC^\flat$.

        \item There is a natural equivalence of global categories compatible with the forgetful functors
            \[\begin{tikzcd}[cramped, row sep = small]
                {\und{\CAlg}_{\gl}(\cC^{\tensor,\flat})} && {\CAlg(\cC^\tensor)^\flat} \\
                & {\cC^\flat}
                \arrow["\simeq", from=1-1, to=1-3]
                \arrow["\bbU"', from=1-1, to=2-2]
                \arrow["\bbU", from=1-3, to=2-2]
            \end{tikzcd}\]
            In particular, the global category of normed algebras in a Borel normed category
            is itself Borel, with underlying category $\CAlg_\gl(\cC^{\tensor,\flat}) \simeq \CAlg(\cC^\tensor)$.
    \end{enumerate}
\end{introthm}

Moreover, global Borelification provides an extremely convenient way for constructing normed global categories from 1-categorical data. We illustrate this approach
by constructing the normed global categories $\und{\Sp}^\tensor$ and $\und{\Sp}_\gl^\tensor$
by suitably deriving the global Borelification of a 1-categorical model of the category of spectra.
Furthermore, in combination with the above results, this allows us to build comparison functors\footnote{In \cite{LLP} we show that all these comparison functors are actually equivalences.}
\begin{equation}\label{eq:intro-comparison}
    \UCom_G \xto{\Phi_G} \CAlg_G(\und{\Sp}_G^\tensor)
    \quad\text{and}\quad
    \UCom_\gl \xto{\Phi_\gl} \CAlg_\gl(\und{\Sp}_\gl^\tensor)
    \xto{i^*} \CAlg_\eq(\und{\Sp}^\tensor)
\end{equation}
all of which are compatible with forgetful functors to $\Sp_G$ respectively $\Sp_\gl$,
where $\CAlg_\eq(\und{\Sp}^\tensor)$ is again a certain category
of sections of $\int \und{\Sp}^\tensor \to \Span_\Forb(\Fglo)$.
We will also give a version for ultra-commutative $G$-global ring spectra,
and make all the comparison functors natural in $G$,
see Constructions \ref{con:ucomm-comp},\ref{con:ucomm-G} and \ref{con:ucomm-eq}.
This takes care of the remaining task $(\star)$ mentioned above,
so we can construct $\pic_G$ from \cref{introthm:pic} as the composite
\[
    \pic_G \colon \UCom_G
    \xto{\Phi_G} \CAlg_{G}(\und{\Sp}_G^\tensor)
    \xto{\und{\Mod}_{(-)}(\und{\Sp}_G^\tensor)} \Mack_G(\Cat)
    \xto{\pic_*} \Sp_G
\]
and use an analogous construction in the global case.

\subsection*{Outline}
We begin in \cref{sec:algebra}
by introducing the framework of parametrized higher algebra
we will be working in for most of this article.
This consists of a short introduction to span categories
followed by a recollection of the notions of normed categories
and their (parametrized) categories of normed algebras as developed in \cite{LLP}
and originally introduced in \cite{Bachmann-Hoyois}.

In \cref{sec:borel} we introduce the theory of (monoidal) Borelification of normed categories
and prove \cref{introthm:borel}.
Contrary to what one might expect, the specialization to the equivariant
case turns out to be more work than the global one,
as it involves some careful record keeping of various $G$-actions.

In \cref{sec:ex} we construct the normed global categories $\und{\Sp}^\tensor,\und{\Sp}_\gl^\tensor$
of equivariant and global spectra.
This is achieved by suitably `deriving' the global monoidal Borelification of a 1-categorical
model for the category of spectra.
We begin with a general account on how to derive normed categories,
and then specialize to construct $\und{\Sp}^\tensor,\und{\Sp}_\gl^\tensor$
and a normed global adjunction between them.
Moreover, we construct the comparison functors from ultra-commutative global ring spectra
to normed algebras in $\und{\Sp}^\tensor$ and $\und{\Sp}_\gl^\tensor$ mentioned above.

\cref{sec:modules} is spent on the construction of parametrized module categories.
Along the way, we recognize the domain of Lurie's functor $\Theta \colon \Cat^{\Alg} \to \Cat^{\Mod},\ (\cC,R) \mapsto \Mod_R(\cC)$ from \cite[Sections 4.8.3-4.8.5]{HA}
as the cocartesian unstraightening of $\Alg(-) \colon \Mon(\Cat) \to \Cat$
using a convenient description of it from the appendix.

In \cref{sec:pic} we begin with some
motivation and background on the classical notion
of Picard groups and Picard spectra,
and then prove \cref{introthm:pic} using the strategy outlined above.

We conclude with two short appendices:
In \cref{appendix:cocart} we show that taking monoids in the total space of the universal cocartesian fibration gives a model for the cocartesian unstraightening of the algebra functor $\Alg \colon \Mon(\Cat)\to \Cat$, and also give a version for commutative monoids.
The purpose of \cref{appendix:lax} is, as the name suggests,
to recall some background on 2-natural transformations and lax natural transformations.

\subsection*{Relation to other work}

This article is an updated version of the author's master's thesis \cite{puetzstueck}
supervised by Stefan Schwede and Kaif Hilman,
whose main goal was to answer Schwede's question on the existence of global Picard spectra
by proving a version of \cref{introthm:pic}.
Weaker versions of most results in this paper are already contained in the thesis,
and have already found use in e.g.~\cite{Tambara}.
The thesis also contained some arduous comparison results for various foundations of parametrized higher algebra
as well as a spectral Mackey functor description of $G$-global spectra.
Both of these have in the meantime independently been shown in other articles
\cite{Tambara,CLL-Semiadd}, so we decided not to include them here.

The paper \cite{LLP} was written simultaneously to this revised version
of the master's thesis. One of its main results is that the comparison functors
(\ref{eq:intro-comparison}) constructed here are in fact equivalences.
The parametrized higher algebra developed for this
is of course well suited to the results in the present article;
originally, the author's master's thesis was entirely built on top of
the theory of algebraic patterns \cite{BHS-Env},
but we have now replaced this with the theory of normed categories from \cite{LLP}.

\subsection*{Acknowledgements} I would like to thank Bastiaan Cnossen, Kaif Hilman, Tobias Lenz, Sil Linskens and Maxime Ramzi for many helpful discussion during the writing of my master's thesis.
I would also like to thank my master's advisor Stefan Schwede for
introducing me to (stable, equivariant, global) homotopy theory,
for countless insightful conversations during my master's,
and for suggesting his question from \cite[Remark 5.1.18]{Global} as a master's thesis project.

Finally, during the writing of the present updated version of my master's thesis,
I have had many useful discussions with Tobias Lenz and Sil Linskens
on content related to this article during our work on the paper \cite{LLP}.
I would also like to thank them both for feedback on a draft of this article.

This project was funded by the Deutsche Forschungsgemeinschaft (DFG, German Research Foundation) -- Project-ID 427320536 -- SFB 1442, as well as under Germany's Excellence Strategy EXC 2044 390685587, Mathematics Münster: Dynamics--Geometry--Structure.

\section{Recollections on Parametrized Higher Algebra}\label{sec:algebra}

\subsection{Span Categories}

Span categories allow us to encode algebraic data in a convenient way,
and will be the base for all that follows.
We will give only the results necessary for this paper,
and refer the reader to \cite{HHLN-Two} for a thorough treatment.

\begin{definition}
    The data of an \emph{adequate triple} $(\cC,\cC_B,\cC_F)$ consists of a category $\cC$
    together with wide subcategories $\cC_B \subset \cC$
    of `backwards' and $\cC_F \subset \cC$ of `forwards' morphisms.\footnote{\cite{HHLN-Two} uses `ingressive' respectively `egressive' for what we call forwards respectively backwards maps. Moreover, they write adequate triples in the order $(\cC,\cC_F,\cC_B)$ instead of our $(\cC,\cC_B,\cC_F)$.}
    We require that the basechange of a backwards morphism along a forwards morphism
    exists in $\cC$ and is again a backwards morphism, and vice-versa.
    In the special case $\cC_B = \cC$, we say $(\cC,\cC_F)$ is a \emph{span pair}.
\end{definition}

Given an adequate triple $(\cC,\cC_B,\cC_F)$,
there is an associated category of spans $\Span_{B,F}(\cC) = \Span(\cC,\cC_B,\cC_F)$ whose objects are those of $\cC$, and whose morphisms are spans
$X \xleftarrow{B} Y \xto{F} Z$, where the labels indicate that $Y \to X$ is a backwards
morphism (i.e.~lies in $\cC_B$) and $Y \to Z$ is a forwards morphism (i.e.~lies in $\cC_F$).
Composition is given by pullback, and the conditions on an adequate triple are exactly the minimum necessary ones to make this well-defined:
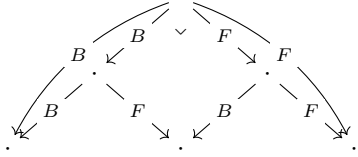
\[\begin{tikzcd}[cramped]
	&& \cdot \\
	& \cdot && \cdot \\
	\cdot && \cdot && \cdot
	\arrow["B"{description}, from=1-3, to=2-2]
	\arrow["F"{description}, from=1-3, to=2-4]
	\arrow["B"{description}, curve={height=12pt}, from=1-3, to=3-1]
	\arrow["\lrcorner"{anchor=center, pos=0.125, rotate=-45}, draw=none, from=1-3, to=3-3]
	\arrow["F"{description}, curve={height=-12pt}, from=1-3, to=3-5]
	\arrow["B"{description}, from=2-2, to=3-1]
	\arrow["F"{description}, from=2-2, to=3-3]
	\arrow["B"{description}, from=2-4, to=3-3]
	\arrow["F"{description}, from=2-4, to=3-5]
\end{tikzcd}\]

\begin{notation}
    Let $(\cC,\cC_B,\cC_F)$ be an adequate triple.
    In the special case where $\cC_B = \cC$ (i.e.~$(\cC,\cC_F)$ is a span pair),
    we will write $\Span_{F}(\cC) \coloneqq \Span_{\all,F}(\cC) \coloneqq \Span(\cC,\cC,\cC_F)$.
    Similarly, if $\cC_F = \cC$, then we write $\Span_{B,\all}(\cC) \coloneqq \Span(\cC,\cC_B,\cC)$.
    Finally, note that if $\cC$ admits pullbacks, then $(\cC,\cC,\cC)$
    is an adequate triple, and we write $\Span(\cC) \coloneqq \Span(\cC,\cC,\cC)$.
\end{notation}

A morphism of adequate triples is a functor which preserves
backwards and forwards maps as well as their basechanges along each other.
This defines a subcategory $\AdTrip \subset \Fun(\Lambda^2_2,\Cat)$,
and the above span construction upgrades to a functor
\[
    \Span \colon \AdTrip \to \Cat,\ (\cC,\cC_B,\cC_F) \mapsto \Span_{B,F}(\cC)
\]
by \cite[Definition 2.12]{HHLN-Two}. We list some important properties of span categories.

\begin{theorem}\label{thm:spans}
    Let $(\cC,\cC_B,\cC_F)$ be an adequate triple.
    \begin{enumerate}
        \item $\AdTrip$ admits all limits, and they are computed in
            $\Fun(\Lambda^2_2,\Cat)$.

        \item The functor $\Span \colon \AdTrip \to \Cat$ preserves all limits.

        \item There is a natural equivalence
            $\Span_{B,F}(\cC)^\op \simeq \Span_{F,B}(\cC)$.

        \item The map of adequate triples $(\cC,\cC^\simeq,\cC_F) \to (\cC,\cC_B,\cC_F)$
            induces an inclusion
            \[
                \cC_F \simeq \Span_{\simeq,F}(\cC) \subset \Span_{B,F}(\cC)
            \]
            sending a morphism $X \to Y$ in $\cC_F$ to the forwards map
            $X \xleftarrow{=} X \to Y$.
            Dually (e.g.~using (3)) we have a subcategory inclusion
            \[
                \cC_B^\op \simeq \Span_{B,\simeq}(\cC) \subset \Span_{B,F}(\cC).
            \]
    \end{enumerate}
\end{theorem}
\begin{proof}
    This is \cite[Lemma 2.4, Theorem 2.18, Lemma 2.14, Prop.~2.15]{HHLN-Two}.
\end{proof}

In the introduction we mentioned that span categories are useful for encoding
algebraic data. Let us elaborate on this.
Consider the category of finite sets $\F$.
It admits pullbacks, so we can form $\Span(\F)$\footnote{Even though $\F$ is a 1-category, $\Span(\F)$ is a $(2,1)$-category!},
and one can show (c.f.~\cite[Lemma C.3]{Bachmann-Hoyois})
that $\Span(\F)$ is semiadditive, with biproducts given by coproducts in $\F$,
and projections respectively inclusions
\[
    \pr_X = (X \sqcup Y \hookleftarrow X = X)
    \qquad\text{respectively}\qquad
    i_X = (X = X \hookrightarrow X \sqcup Y).
\]

\begin{definition}
    Let $\cE$ be a category with finite products.
    We define $\Mack(\cE) \coloneqq \Fun^\times(\Span(\F),\cE)$
    and the forgetful functor $\fgt \coloneqq \ev_1 \colon \Mack(\cE) \to \cE$.
\end{definition}

Let $\Cat^\times$ denote the category of small categories admitting finite products
and functors preserving them, and let $\Cat^\sadd \subseteq \Cat^\times$
be spanned by semiadditive categories.

\begin{proposition}\label{prop:mack}
    The functor $\Mack \colon \Cat^\times \to \Cat^\sadd$ is a right Bousfield
    localization with counit $\fgt$.
    Moreover, there is a homotopy of natural equivalences
    \[
        \fgt_* \simeq \fgt \colon \Mack(\Mack(-)) \xRightarrow{\simeq} \Mack(-).
    \]
\end{proposition}
\begin{proof}
    By \cite[Theorem 1.1]{Harpaz} $\Span(\F)$
    is in fact the free semiadditive category on one generator,
    i.e.~$\fgt_\cE \colon \Mack(\cE) \to \cE$ is an equivalence whenever $\cE$ is semiadditive.
    Since $\Span(\F)$ is semiadditive, so is $\Mack(\cE)$ by \cite[Corollary 2.4]{GGN},
    so $\Mack$ really does take values in $\Cat^\sadd$.
    In view of \cite{HTT}*{Proposition 5.2.7.4${}^\op$}
    it remains to show that $\fgt_* \simeq \fgt$.
    To this end, note that the cartesian product $\times \colon \F \times \F \to \F$
    preserves coproducts in each variable and also preserves pullbacks,
    hence induces a functor
    $\tensor \colon \Span(\F) \times \Span(\F) \simeq \Span(\F \times \F) \to \Span(\F)$
    such that $\tensor$ preserves biproducts in each variable and $1 \tensor - \simeq - \tensor 1 \simeq \id$.
    So precomposing with $\tensor$ gives
    \[
        \tensor^* \colon \Mack(\cE) \to \Fun^{\times,\times}(\Span(\F) \times \Span(\F), \cE) \simeq \Mack(\Mack(\cE)).
    \]
    Now note that both $\fgt$ and $\fgt_*$ are left inverse to $\tensor^*$.
    Since $\fgt$ is an equivalence by semiadditivity of $\Mack(\cE)$,
    we conclude that $\fgt \simeq \fgt_*$ are inverse to $\tensor^*$.
\end{proof}

\begin{remark}\label{rem:mack-vs-cmon}
    We will sometimes refer to $\Mack(\cE)$ as the category of commutative monoids in $\cE$.
    Often, this name is instead reserved for a certain full subcategory $\CMon(\cE) \subseteq \Fun(\F_*,\cE)$,
    c.f.~\cite[Definition 2.4.2.1]{HA}.
    By \cite[Proposition C.1]{Bachmann-Hoyois} there is a functor $\F_* \to \Span(\F)$
    restriction along which induces a natural equivalence
    $\Mack(\cE) \xto{\simeq} \CMon(\cE)$ compatible with the forgetful functors to $\cE$,
    so our notion of commutative monoid agrees with the one of Lurie.
\end{remark}

\begin{remark}
    Consider a category $\cE$ with finite products and $\Phi \in \Mack(\cE)$.
    Let us explain how $\Phi$ encodes the structure of a commutative monoid
    on $M \coloneqq \Phi(1) \in \cE$.
    \begin{itemize}
        \item We have $\Phi(X) = M^{|X|}$, and in particular $\Phi(\empty) = * \in \cE$
            is terminal.

        \item Backwards maps encode auxiliary data;
            a map $f \colon Y \to X$ in $\F$ induces
            \[
                f^* \coloneqq \Phi(X \leftarrow Y = Y) \colon M^{|X|} \to M^{|Y|}
            \]
            where $\pr_yf^* \simeq \pr_{f(y)}$ for $y\in Y$.
            In particular, if $f$ is an injection respectively surjection,
            then $f^*$ is a projection respectively product of diagonal functors.

        \item The unique map $X \to *$ defines an $|X|$-fold multiplication
            \[
                \mu_{X} \colon \Phi(X = X \to *) \colon M^{|X|} \to M.
            \]
            In general $\mu_f \coloneqq \Phi(X = X \xto{f} Y)$ sums over the fibers,
            i.e.~$\pr_y\mu_f = \mu_{f^{-1}(y)}$.
            The special case $X = \empty$ defines a unit $\mu_{\empty} \colon * \to M$.

        \item Functoriality of $\Phi$ allows $M$ to inherit
            unitality, associativity, commutativity and the higher coherence data
            all from the corresponding data of the canonical commutative monoid
            structure on $1 \in \F$ in the cocartesian monoidal structure on finite sets.
    \end{itemize}
\end{remark}

We would like to have analogous semiadditivity properties
in more general span categories.
To this end, we consider the following definition.

\begin{definition}
    A span pair $(\cF,\cN)$ is \emph{extensive} if
    \begin{enumerate}
        \item $\cF$ is extensive, i.e.~it has finite coproducts and the coproduct functor
            \[
                \amalg \colon \prod_{i=1}^n \cF_{/A_i} \to \cF_{/\coprod_{i=1}^n A_i}
            \]
            is an equivalence for all $A_1,\dots,A_n \in \cF$;

        \item the morphisms in $\cN$ are closed under finite coproducts;

        \item $\cN$ contains the maps $\empty \to A$ and $\nabla \colon A \amalg A \to A$
            for each $A \in \cF$.
    \end{enumerate}
\end{definition}

\begin{example}\label{ex:extensive}
    For a category $T$, denote by $\F[T]$ the free finite coproduct completion of $T$,
    defined as the full subcategory $\F[T] \subseteq \PSh(T)$ on finite coproducts of representables.
    A wide subcategory $P \subset T$ is called \emph{orbital}
    if $(\F[T],\F[P])$ is a span pair,
    which is then automatically extensive, see \cite[Example 3.1.9]{Tambara}.
\end{example}

In this paper we will mostly be interested in the following two examples,
which let us reason about (globally) equivariant homotopy theory in our framework.

\begin{example}
    Let $G$ be a finite group. Then the category $P = T = \Orb_G$ of transitive
    $G$-sets is orbital in itself. The category of finite $G$-sets
    $\F_G \coloneqq \F[\Orb_G]$ thus gives an extensive span pair $(\F_G,\F_G)$,
    and we will write $\Span(G) \coloneqq \Span(\F_G)$.
\end{example}

\begin{example}
    Let $\Glo$ denote the (2,1) category of finite groups,
    group homomorphisms and conjugations, and $\Orb \subset \Glo$
    the wide subcategory on injective group homomorphisms.
    Equivalently, we can think of $\Glo$ as the category of finite connected
    1-groupoids, and $\Orb$ as the subcategory on faithful maps.
    It was shown in \cite[Example 4.2.5]{CLL-Global} that $\Orb$ is an orbital
    subcategory of $\Glo$.
    We obtain an extensive span pair $(\Fglo,\Forb) \coloneqq (\F[\Glo],\F[\Orb])$.
\end{example}

\begin{example}
    Let $(\cF,\cN)$ be an extensive span pair and $A \in \cF$.
    Then $(\cF_{/A}, \cF_{/A} \times_\cF \cN)$ is again an extensive span pair.
    This is evident from the fact that coproducts and pullbacks in $\cF_{/A}$ are computed in $\cF$, and the natural equivalence $(\cF_{/A})_{/B} \simeq \cF_{/B}$.
    We will denote the associated span category by $\Span_\cN(\cF_{/A})$.
\end{example}

\begin{proposition}[{{\cite[Proposition 2.2.5]{Tambara}}}]\label{prop:span-semiadd}
    Let $(\cF,\cN)$ be an extensive span pair.
    Then $\Span_\cN(\cF)$ is semiadditive with projections and inclusions given by
    \[
        \pr_A = (A \amalg B \leftarrow A = A)
        \qquad\text{and}\qquad
        \inc_A = (A = A \to A \amalg B).
    \]
    Hence a functor $\cC \colon \Span_\cN(\cF) \to \cE$ preserves finite products
    precisely if $\cC|_{\cF^\op}$ does.
\end{proposition}

\begin{definition}
    If $(\cF,\cN)$ is an extensive span pair and $\cE \in \Cat^\times$, we define
    \[
        \Mack_\cF^\cN(\cE) \coloneqq \Fun^\times(\Span_\cN(\cF),\cE)
    \]
    and refer to it as the category of $\cN$-normed $\cF$-monoids in $\cE$.
\end{definition}

\begin{notation}\label{not:g-gl}
    We will abbreviate $\Mack_G \coloneqq \Mack_{\F_G}^{\F_G}$ and $\Mack_\gl \coloneqq \Mack_{\Fglo}^\Forb$.
\end{notation}

A famous use of such $\F_G$-normed $\F_G$-monoids,
often simply called $G$-commutative or $G$-$\E_\infty$-monoids,
is in a theorem of Guillou--May \cite{Guillou-May} (see also \cite[Appendix A]{CMNN})
which gives an equivalence $\Mack_G(\Sp) \simeq \Sp_G$
between $G$-commutative monoids in $\Sp$ and genuine $G$-spectra.
This equivalence identifies evaluation at $G/H$
with the genuine $H$-fixed points functor $(-)^H \colon \Sp_G \to \Sp$.

Analogously to how one can consider symmetric monoidal categories
as commutative monoids in $\Cat$, we can therefore also consider
the case $\cE = \Cat$ above, which leads us to the study
of normed categories and parametrized higher algebra.

\subsection{Normed Categories}

In this subsection we recall some constructions and results on parametrized
higher algebra as developed in \cite[Section 2]{LLP}.

\begin{definition}\label{def:f-cat}
    Let $\cF$ be any category.
    \begin{itemize}
        \item We refer to $\PreCat{\cF} \coloneqq \Fun(\cF^\op,\Cat)$
            as the category of $\cF$-precategories and $\cF$-functors.

        \item If $\cF$ admits finite coproducts,
            we also define the full subcategory
            $\PCat{\cF} \coloneqq \Fun^\times(\cF^\op,\Cat)$
            and refer to the objects as $\cF$-categories.

        \item Given an $\cF$-precategory,
            we call $\lim_{\cF^\op}\cC$ the underlying category of $\cC$.
            Note that if $\cF$ admits a final object $1 \in \cF$,
            then this is simply $\cC(1)$.
    \end{itemize}
\end{definition}

As mentioned above, we will use span categories to define normed categories:

\begin{definition}\label{def:n-f-cat}
    Let $(\cF,\cN)$ be a span pair.
    \begin{itemize}
        \item We refer to $\PreNmCat{\cF}{\cN} \coloneqq \Fun(\Span_\cN(\cF),\Cat)$
            as the category of $\cN$-normed $\cF$-precategories
            and ($\cN$-)normed $\cF$-functors.

        \item If $(\cF,\cN)$ is extensive, we define the full subcategory of $\cN$-normed
            $\cF$-categories as $\Mack_\cF^\cN(\Cat) = \Fun^\times(\Span_\cN(\cF),\Cat)$.

        \item The underlying $\cF$-(pre)category of an $\cN$-normed $\cF$-(pre)category $\cC$ is defined to be the restriction $\cC|_{\cF^\op}$.
    \end{itemize}
\end{definition}

\begin{remark}
    In \cite{LLP} the notations $\Cat^*(\cF),\Cat(\cF),\text{NCat}_\cN^*(\cF),\text{NCat}_\cN(\cF)$ were used for what we write as
    $\Cat_\cF^*,\Cat_\cF,\Mack_{\cF}^{\cN,*}(\Cat),\Mack_\cF^\cN(\Cat)$.
\end{remark}

Note that an $\cN$-normed $\cF$-(pre)category is nothing but a $\Span_\cN(\cF)^\op$-(pre)category in the sense of \cref{def:f-cat}.
For $(\cF,\cN) = (\F,\F)$, we see that $\NmCat{\F}{\F} = \Mack(\Cat)$
is the category of symmetric monoidal categories.

\begin{notation}
    In the case $(\cF,\cN) = (\F_G,\F_G)$ respectively $(\cF,\cN) = (\Fglo,\Forb)$
    we refer to ($\cN$-normed) $\cF$-(pre)categories as (normed) $G$-(pre)categories respectively (normed) global (pre)categories.
    We will write $\PCat{G} \coloneqq \PCat{\F_G}$ and $\PCat{\gl} \coloneqq \PCat{\Fglo}$
    in accordance with \cref{not:g-gl},
    and use analogous notations for the precategory versions.
\end{notation}

\begin{remark}
    In the literature on parametrized higher category theory,
    the name $\cF$-category refers to what we call an $\cF$-precategory,
    see e.g.~\cite{BDGNS,CLL-Global,MW-Colimits}.
    The reason for this is that we are more interested in parametrized higher algebra,
    and while for example one can define $G$-categories as functors
    $\Orb_G^\op \to \Cat$, we cannot build a reasonable span category on $\Orb_G$,
    and need to instead limit extend to $\F_G$,
    where then the data of a normed structure becomes an extension to $\Span(\F_G)$.
\end{remark}

The study of (higher) algebra naturally involves some 2-categorical elements;
after all, we are interested in \emph{categories} of algebras,
and these are defined most naturally as certain mapping categories
in the category of operads (or symmetric monoidal categories,
using the theory of symmetric monoidal envelopes).
This is even more convenient in the theory of parametrized higher algebra,
so let us recall some basic 2-category theory we will use throughout this article. For a detailed overview, we refer the reader to \cite[Appendices C and D]{Heyer-Mann}.

\begin{remark}\label{rem:2cat}
    Most commonly, 2-category theory is synonymous with $\Cat$-enriched category theory.
    There are various ways on how to formally define this in a higher categorical setting
    \cite[Definition 4.2.1.28]{HA}, \cite{Gepner-Haugseng,Hinich}
    which were shown to be equivalent for all practical purposes in \cite{Heine}.
    For us it will be enough to consider the following special case.

    Viewing $\Cat$ as a commutative algebra in $\CAT$ via the
    cartesian symmetric monoidal structures,
    we can consider $\Cat$-modules in $\CAT$.
    If $\cC$ is such a $\Cat$-module, and each tensoring $- \times c \colon \Cat \to \cC$
    admits a right adjoint $\hom_\cC(c,-) \colon \cC \to \Cat$,
    then $\cC$ canonically inherits a $\Cat$-enrichment, i.e.~a 2-categorical
    structure with hom-categories $\hom_\cC(c,d)$, c.f.~\cite[Example C.1.11]{Heyer-Mann}.
    Moreover, the category of 2-functors between two such $\Cat$-modules
    can be defined as the category of lax $\Cat$-linear functors
    \cite[Example C.2.2]{Heyer-Mann}.
    Recall from \cite[Corollary 7.3.2.7]{HA} that the right adjoint of a $\Cat$-linear
    functor automatically inherits a lax $\Cat$-linear structure,
    so that any adjunction $L \colon \cC \rightleftarrows \cD \noloc R$
    with $\Cat$-linear $L$ upgrades to an adjunction of 2-categories
    with a natural equivalence $\hom_{\cC}(-,R-) \simeq \hom_{\cD}(L-,-)$ of hom-categories.

    For any category $T$, the slice $\Cat_{/T}$ admits a cartesian symmetric monoidal structure.
    The (symmetric monoidal) right adjoint $T \times - \colon \Cat \to \Cat_{/T}$
    then exhibits $\Cat_{/T}$ as a $\Cat$-module whose tensoring admits right adjoints
    as $\Cat$ is locally cartesian closed.
    The same tensoring is already defined on
    the subcategory $\Cat_{/T}^{E\dcc} \subset \Cat_{/T}$
    spanned by categories over $T$ admitting cocartesian lifts over the subcategory $E \subset T$ and functors preserving them.
    We denote the hom-categories by
    \[
        \Fun_{/T}(-,-) \qquad\text{respectively}\qquad \Fun_{/T}^{E\dcc}(-,-).
    \]
    In the case $E = T$, the unstraightening equivalence identifies $T \times -$
    with $\const \colon \Cat \to \Fun(T,\Cat)$.
    Since this commutes with restriction along any functor $f \colon T \to S$,
    the restriction / right Kan extension adjunction is canonically $\Cat$-linear:
    \[
        f^* \colon \Fun(S,\Cat) \rightleftarrows \Fun(T,\Cat) \noloc f_*.
    \]
\end{remark}

\begin{notation}
    In view of the above remark, $\PreCat{\cF}$ and $\PreNmCat{\cF}{\cN}$
    inherit 2-categorical structures in a preferred way, whose hom-categories we denote by
    \[
        \Fun_\cF(-,-)
        \quad\text{and}\quad
        \Fun_{\cF}^{\Nstr}(-,-).
    \]
    We will view $\PCat{\cF}$ and $\Mack_\cF^\cN(\Cat)$
    as full 2-subcategories, i.e.~having the same hom-categories.
\end{notation}

We now recall two constructions from \cite{LLP}
for building normed (pre)categories out of (pre)categories.
Both constructions work in a fibrational setting and first land in the category $\PreOp{\cF}{\cN} \coloneqq \Cat_{/\Span_\cN(\cF)}^{\cF^\op\dcc}$ of so-called $(\cF,\cN)$-preoperads.
As for (normed) precategories, we consider the full subcategory $\POp{\cF}{\cN} \subseteq \PreOp{\cF}{\cN}$ spanned by $(\cF,\cN)$-operads $\cO$,
which are preoperads such that (the cocartesian straightening of) $\cO|_{\cF^\op}$
is an $\cF$-category, i.e.~preserves finite products.

\begin{definition}[{{\cite[Definition 2.15]{LLP}}}]\label{def:n-cocart}
    Let $\cC\colon \cF^{\op}\to \Cat$ be an $\cF$-precategory.
    \begin{enumerate}
        \item We denote the cocartesian unstraightening $\int \cC \to \cF^\op$,
            viewed as a category \mbox{over $\Span_\cN(\cF)$ via the inclusion
            $\cF^\op \subset \Span_\cN(\cF)$, by $\Triv(\cC) \to \Span_\cN(\cF)$.}

        \item By \cite[Proposition 2.6]{HHLN-Two} there is an adequate
            triple on the cartesian unstraightening $p \colon \Un^\cart(\cC) \to \cF$
            given by the two subcategories
            $\cart, \Un^\ct(\cC) \times_\cF \cN \subset \Un^\cart(\cC)$
            of cartesian maps respectively maps lying over $\cN$.
            Moreover, $p$ induces a map of span categories
            \[
                \Span(p) \colon \Span_{\cart,\cN}(\Un^\ct(\cC)) \to \Span_\cN(\cF).
            \]
            We will denote this by $\cC^{\Ncoprod} \to \Span_\cN(\cF)$.
    \end{enumerate}
\end{definition}

\begin{lemma}[{{\cite[Lemma 2.16]{LLP}}}]\label{lem:unfurl}
    If $\cC$ is an $\cF$-precategory, then $\Triv(\cC)$ and $\cC^\Ncoprod$
    are $(\cF,\cN)$-preoperads, i.e.~admit cocartesian lifts over $\cF^\op$.
    Moreover, there is a pullback square
    \[\begin{tikzcd}
        {\int\cC} & {\cC^{\Ncoprod}} \\
        {\cF^\op} & {\Span_\cN(\cF)}
        \arrow[from=1-1, to=1-2]
        \arrow[from=1-1, to=2-1]
        \arrow["\lrcorner"{anchor=center, pos=0.125}, draw=none, from=1-1, to=2-2]
        \arrow[from=1-2, to=2-2]
        \arrow[from=2-1, to=2-2]
    \end{tikzcd}\]
    which induces a natural map $\incl \colon \Triv(\cC) \to \cC^{\Ncoprod}$ of $(\cF,\cN)$-preoperads.
\end{lemma}

\begin{example}\label{ex:n-a}
    Let $A \in \cF$ and denote by $\und{A} = \cF(-,A)$ the represented $\cF$-category.
    The cartesian unstraightening of $\und{A}$ is given by the right fibration $\cF_{/A} \to \cF$.
    In particular, the above pullback square becomes
    \[\begin{tikzcd}[cramped]
        {\Triv(\und{A})} & {(\cF_{/A})^\op} & {\Span_\cN(\cF_{/A})} & {\und{A}^{\Ncoprod}} \\
        & {\cF^\op} & {\Span_\cN(\cF)}
        \arrow[Rightarrow, no head, from=1-1, to=1-2]
        \arrow[from=1-2, to=1-3]
        \arrow[from=1-2, to=2-2]
        \arrow["\lrcorner"{anchor=center, pos=0.125}, draw=none, from=1-2, to=2-3]
        \arrow[Rightarrow, no head, from=1-3, to=1-4]
        \arrow[from=1-3, to=2-3]
        \arrow[from=2-2, to=2-3]
    \end{tikzcd}\]
\end{example}

\begin{theorem}[{{\cite[Prop.~2.26 and Thm.~2.57]{LLP}}}]\label{thm:ncoprod-triv}
    Let $(\cF,\cN)$ be a span pair.
    \begin{enumerate}
        \item The above constructions form $\Cat$-linear left adjoints
            \[
                \Triv, (-)^{\Ncoprod} \colon \PreCat{\cF} \to \PreOp{\cF}{\cN}
            \]
            such that the natural transformation $\incl \colon \Triv \Rightarrow (-)^{\Ncoprod}$ is $\Cat$-linear.

        \item The right adjoint of $\Triv$ is the forgetful functor $\PreOp{\cF}{\cN} \to \PreCat{\cF}$
            with unit given by the obvious equivalence $\id \simeq \fgt \circ \Triv$.

        \item If $(\cF,\cN)$ is extensive, both adjunctions restrict to $\PCat{\cF} \rightleftarrows \POp{\cF}{\cN}$.
    \end{enumerate}
\end{theorem}

It is often more convenient to work in the straightened world,
i.e.~with fibrations admitting all cocartesian lifts.
The following allows us to achieve this for the above constructions
by freely adding the missing cocartesian lifts.

\begin{definition}\label{def:envelope}
    Let $\Ar_\cN(\Span_\cN(\cF))$ denote the full subcategory of $\Ar(\Span_\cN(\cF))$ on the forwards maps.
    By \cite[Corollary 2.22]{HHLN-Two} this is equivalent to the span category $\Span_{\pb,\cN}(\Ar_\cN(\cF))$ with morphisms of the form
    \[\begin{tikzcd}[cramped]
        \cdot & \cdot & \cdot \\
        \cdot & \cdot & \cdot
        \arrow["\cN", from=1-1, to=2-1]
        \arrow[from=1-2, to=1-1]
        \arrow["\cN", from=1-2, to=1-3]
        \arrow["\lrcorner"{anchor=center, pos=0.125, rotate=-90}, draw=none, from=1-2, to=2-1]
        \arrow["\cN", from=1-2, to=2-2]
        \arrow["\cN", from=1-3, to=2-3]
        \arrow[from=2-2, to=2-1]
        \arrow["\cN"', from=2-2, to=2-3]
    \end{tikzcd}\]
    i.e.~where backwards maps are pullback squares and forwards maps lie in $\Ar(\cN)$.
    Given $\cO \in \PreOp{\cF}{\cN}$, its envelope $\Env(\cO) \to \Span_\cN(\cF)$ is the horizontal composite
    \[\begin{tikzcd}[cramped]
        {\Env(\cO)} & {\Ar_\cN(\Span_\cN(\cF))} & {\Span_\cN(\cF)} \\
        \cO & {\Span_\cN(\cF)}
        \arrow[from=1-1, to=1-2]
        \arrow[from=1-1, to=2-1]
        \arrow["\lrcorner"{anchor=center, pos=0.125}, draw=none, from=1-1, to=2-2]
        \arrow["t"', from=1-2, to=1-3]
        \arrow["s", from=1-2, to=2-2]
        \arrow[from=2-1, to=2-2]
    \end{tikzcd}\]
\end{definition}

\begin{proposition}
    The above defines a $\Cat$-linear adjunction
    \[
        \Env \colon \PreOp{\cF}{\cN} \rightleftarrows \PreNmCat{\cF}{\cN} \noloc \fgt.
    \]
    If $(\cF,\cN)$ is even extensive, this restricts to an adjunction $\POp{\cF}{\cN} \rightleftarrows \NmCat{\cF}{\cN}$.
\end{proposition}
\begin{proof}
    The existence of this adjunction is a special case of \cite[Theorem E]{BHS-Env},
    c.f.~\cite[Theorem 2.23]{LLP}. The addendum is \cite[Theorem 2.57(4)]{LLP}.
\end{proof}

\begin{definition}
    We denote the right adjoint of the 2-functor $\Env \circ\,(-)^{\Ncoprod}$ by
    \[
        \und{\CAlg}_\cF^\cN \colon \PreNmCat{\cF}{\cN} \to \PreCat{\cF}
    \]
    and refer to $\und{\CAlg}_\cF^\cN(\cC)$ as the
    \emph{$\cF$-precategory of $\cN$-normed algebras in $\cC$.}
    By the above, if $(\cF,\cN)$ is extensive,
    this restricts to $\und{\CAlg}_\cF^\cN \colon \NmCat{\cF}{\cN} \to \PCat{\cF}$.
\end{definition}

We now give a more concrete description of $\und{\CAlg}_\cF^\cN(\cC)$
for some $\cN$-normed $\cF$-precategory $\cC$.
Let $\und{A} = \cF(-,A)$ denote the represented $\cF$-precategory.
It was shown in \cite[Lemma 2.2.7]{CLL-Global} that there is an equivalence
\[
    \ev_{\id_A} \colon \Fun_\cF(\und{A},\cC) \simeq \cC(A)
\]
natural in both $\cC \in \PreCat{\cF}$ and $A \in \cF^\op$.
We thus obtain natural equivalences
\begin{align*}
    \und{\CAlg}_\cF^\cN(\cC)(A)
    &\simeq \Fun_\cF(\und{A}, \und{\CAlg}_\cF^\cN(\cC))\\
    &\simeq \Fun_{/\Span_\cN(\cF)}^{\cF^\op\dcc}(\Span_\cN(\cF_{/A}),\smallint\cC)
    \simeq \Fun_{\cF}^{\Nstr}(\Env(\und{A}^{\Ncoprod}),\cC)
\end{align*}
Analogously, one has natural equivalences
\[
    \cC(A)
    \simeq \Fun_\cF(\und{A}, \cC)\\
    \simeq \Fun_{/\Span_\cN(\cF)}^{\cF^\op\dcc}((\cF_{/A})^\op,\smallint \cC)
    \simeq \Fun_{\cF}^{\Nstr}(\Env(\Triv(\und{A})),\cC).
\]
Furthermore, using that $(-)^{\Ncoprod}$ is a left adjoint and the colimit of the Yoneda embedding $\cF \to \PSh(\cF)$
is terminal, one easily shows (c.f.~\cite[Lemma 2.31]{LLP})
that the underlying category $\CAlg_\cF^\cN(\cC) \coloneqq \lim_{\cF^\op}\und{\CAlg}_\cF^\cN(\cC)$ is given by
\[
    \CAlg_\cF^\cN(\cC)
    \simeq \Fun_{/\Span_\cN(\cF)}^{\cF^\op\dcc}(\Span_\cN(\cF), \smallint\cC)
    \simeq \Fun_{\cF}^{\Nstr}(\Env((*)^\Ncoprod), \cC)
\]
where $*$ denotes the terminal $\cF$-category.
Moreover, this equivalence is compatible with the forgetful functors to $\lim_{\cF^\op} \cC \simeq \Fun_{/\cF^\op}^\cocart(\cF^\op,\cC)$.

\begin{remark}
    Normed $G$-categories are also known
    as $G$-symmetric monoidal categories \cite{Nardin-Shah}.
    The definition $\CAlg_G(\cC) = \Fun_{/\Span(G)}^{\F_G^\op\dcc}(\Span(G),\int \cC)$ is originally due to Bachmann--Hoyois \cite{Bachmann-Hoyois},
    and can also be shown to agree with the one considered by Nardin--Shah \cite{Nardin-Shah}, c.f.~\cite[Remark 5.6.4]{Tambara} or \cite[Corollary 2.59]{puetzstueck}.
\end{remark}

\begin{definition}\label{def:calg-forget}
    The forgetful $\cF$-functor
    $
        \bbU \colon \und{\CAlg}_\cF^\cN(\cC) \to \cC
    $
    is the unique one making the following square commute naturally in $\cC \in \PreNmCat{\cF}{\cN}$ and $A \in \cF^\op$:
    \[\begin{tikzcd}
    {\Fun_{\cF}^{\Nstr}(\Env(\und{A}^{\Ncoprod}), \cC)} & {\und{\CAlg}_\cF^\cN(\cC)(A)} \\
        {\Fun_\cF^{\Nstr}(\Env(\Triv(\und{A})),\cC)} & {\cC(A)}
        \arrow["\sim", from=1-1, to=1-2]
        \arrow["\Env(\incl)^*"', from=1-1, to=2-1]
        \arrow["{\mathbb{U}(A)}", from=1-2, to=2-2]
        \arrow["\sim", from=2-1, to=2-2]
        \arrow[ draw=none, from=2-1, to=2-2]
    \end{tikzcd}\]
\end{definition}

\begin{example}\label{ex:finset}
    In the special case $(\cF,\cN) = (\F,\F)$ and $A = 1 \in \F$
    one directly computes that $\Env(\incl) \colon \F^{\sqcup,\simeq} \to \F^{\sqcup}$
    is the inclusion of the groupoid core of the category of finite sets
    equipped with its cocartesian symmetric monoidal structure,
    which as morphism of cocartesian fibrations over $\Span(\F)$
    is given by $\Ar(\F)^{\pb,\op} \to \Span_{\pb,\all}(\Ar(\F)) \simeq \Ar_\F(\Span(\F))$ where `$\pb$' denotes the class of morphisms in $\Ar(\F)$
    which are pullback squares, c.f.~\cite[Example 2.41]{LLP}.
    In particular, we recognize $\F^{\sqcup,\simeq}$ as the free symmetric
    monoidal category on one object, and $\F^{\sqcup}$ as the free symmetric
    monoidal category on one algebra.
    For any symmetric monoidal category $\cC \in \Mack(\Cat)$,
    we can identify the forgetful functor $\CAlg(\cC) \to \cC$
    with the restriction $\Fun^\tensor(\F^{\sqcup},\cC) \to \Fun^\tensor(\F^{\sqcup,\simeq},\cC^\tensor)$, where both equivalences are induced by evaluating at $1 \in \F$.
\end{example}

\section{Borel Parametrized Algebra}\label{sec:borel}

In this section we define an analogue of the theory of Borel-equivariant objects
in the setting of parametrized higher algebra.
The following is the main example to keep in mind throughout this section,
and will be vital for constructions of normed global categories from 1-categorical
data in \cref{sec:ex}.

\begin{example}\label{ex:borel-global}
    Given a category $\cC$, we can right Kan extend along the inclusion
    $* \subseteq \Glo$ to obtain a Borel global category
    $\cC^\flat \in \Fun(\Glo^\op,\Cat) \simeq \Cat_\gl$
    which by the pointwise formula for right Kan extensions
    sends $G$ to $\cC^{BG} = \Fun(BG,\cC)$ with the obvious restriction functoriality.
    We will see below that the inclusions $i \colon * \subseteq \Orb \subset \Glo$
    induce a fully faithful inclusion
    $\Span(i) \colon \Span(\F) \hookrightarrow \Span_{\Forb}(\Fglo)$,
    and given a symmetric monoidal category $\cC^\tensor \in \Mack(\Cat)$,
    we may consider its monoidal global Borelification
    $\cC^{\tensor,\flat} \coloneqq \Span(i)_*\cC^\tensor \in \Mack_\gl(\Cat)$.
    As we will learn throughout this section
    (c.f.~\cref{rem:borel-glo-functoriality} below),
    this normed global category admits the following description:
    \begin{enumerate}
        \item For a map of groups $\alpha \colon G \to K$,
            the functor $\cC^{\tensor,\flat}(K \leftarrow G = G) \colon \cC^{\tensor,\flat}(K) \to \cC^{\tensor,\flat}(G)$ agrees with the usual restriction
            $\alpha^* \colon \cC^{BK} \to \cC^{BG}$.

        \item The restriction of $\cC^{\tensor,\flat}$
            along the functor $\Span(\F) \to \Span_{\Forb}(\Fglo)$
            induced by the unique coproduct-preserving functor
            $\F \to \Fglo,\ * \mapsto G$
            encodes the pointwise symmetric monoidal structure on $\cC^{BG}$.

        \item The remaining functoriality to be described is that in forwards maps
            \[
                \cC^{\tensor,\flat}(H = H \to G) \colon \cC^{BH} \to \cC^{BG}
            \]
            for inclusions of subgroups $H \leq G$.
            This sends an object $X \in \cC^{BH}$ to the $G/H$-fold tensor product
            $X^{\tensor|G/H|}$ equipped with a certain $G$-action.
            In view of two uniqueness results we will discuss below,
            it is reasonable to refer to this as being given by `the'
            symmetric monoidal norm.
            In the case where $\cC^\tensor$ is an ordinary symmetric monoidal
            1-category, this guarantees that the functor agrees with the classical symmetric monoidal norm where the $G$-action is restricted in a certain way
            from the canonical $\Sigma_{|G/H|} \wr H$ action,
            c.f.~\cite[Construction A.3.4]{Tambara}.
            In particular, if we consider some symmetric monoidal model of the 1-category
            of spectra such as symmetric spectra $\Sp^{\Sigma,\tensor}$ (see \cref{subsec:ex}),
            then we recognize $(\Sp^{\Sigma,\tensor})^\flat(H = H \to G)$ as the well-known
            Hill--Hopkins--Ravenel norm $N_H^G \colon H\Sp^\Sigma \to G\Sp^\Sigma$ from \cite{HHR}.
    \end{enumerate}
\end{example}

Before we specialize to the case of Borelifications in equivariant and global homotopy theory in \cref{subsec:borel-eq}, we will develop some general theory in the following subsection.

\subsection{General Borel Theory}

In this subsection, we define the general notion of a Borel inclusion of span pairs
and prove the general version of \cref{introthm:borel} from the introduction.

\begin{definition}\label{def:borel-inc}
    Let $(\blf,\bln)$ be a
    span pair and $i \colon \bsf \subseteq \blf$ a full subcategory
    so that $\bsn \coloneqq \bsf \cap \bln \subseteq \bln$ forms a sieve,
    i.e.~every morphism in $\bln$ with target in $\bsn$ also has source in $\bsn$.
    Then $(\bsf,\bsn)$ is again a span pair, and we call $i \colon (\bsf,\bsn) \subseteq (\blf,\bln)$ a \emph{Borel inclusion}.
\end{definition}

\begin{lemma}\label{lem:borel-ext}
    If $(\cF,\cN)$ is furthermore extensive and
    $\bsf \subseteq \blf$ is closed under coproducts,
    then the equivalence $\amalg \colon \cF_{/A} \times \cF_{/B} \simeq \cF_{/A \amalg B}$ restricts to an equivalence
    \[
        \amalg^\cE \colon (\cE \times_{\cF} \cF_{/A}) \times (\cE \times_{\cF} \cF_{/B}) \simeq \cE \times_\cF \cF_{/A \amalg B}
    \]
    In particular, picking $A,B \in \cE$, we see that also $(\cE,\cM)$ is extensive.
\end{lemma}
\begin{proof}
    Since $\cE \subseteq \cF$ is closed under coproducts
    $\amalg$ restricts to a functor $\amalg^\cE$ between the full subcategories.
    By generalities about extensive categories (c.f.\cite[Remark 2.2.3]{Tambara}),
    pulling back along the summand inclusions $j_A \colon A \to A \amalg B$
    and $j_B \colon B \to A \amalg B$ defines an inverse $(j_A^*,j_B^*) \colon \cF_{/A \amalg B} \to \cF_{/A} \times \cF_{/B}$ of $\amalg$.
    Since $\cE \subseteq \cF$ is a sieve, it follows that $(j_A^*,j_B^*)$ restricts
    to an inverse of $\amalg^\cE$.
    Finally, if $A,B \in \cE$, then $\amalg^\cE$ is simply the coproduct functor
    $\cE_{/A} \times \cE_{/B} \to \cE_{/ A \amalg B}$, so the above implies that $\cE$ is extensive.
    Since we defined $\cM \coloneqq \cE \cap \cN$, it is clear that then $(\cE,\cM)$ is an extensive span pair.
\end{proof}

\begin{definition}
    In the above situation of the above Lemma, we call $i \colon (\cE,\cM) \subseteq (\cF,\cN)$ an \emph{extensive Borel inclusion.}
\end{definition}

\begin{example}\label{ex:free-cop-borel}
    If $Q \subset S$ is an orbital subcategory (c.f.~\cref{ex:extensive}) and $T \subseteq S$ is a full subcategory so that $P \coloneqq T \cap Q \subseteq Q$ is a sieve,
    then $(\F[T],\F[P]) \subseteq (\F[S],\F[Q])$ is an extensive Borel inclusion.
\end{example}

\begin{example}\label{ex:borel-inc}
    Let $G$ be a finite group. The following examples are easily verified:
    \begin{enumerate}
        \item The inclusion $BG \subseteq \Orb_G$ induces via the above example the extensive Borel inclusion
            $(\F[BG],\F[BG]) \subseteq (\F_G,\F_G)$.
            This extends objects with $G$-action to genuinely $G$-equivariant objects.

        \item Similarly, we can apply the above example to the inclusion
            $(\Orb_G,\Orb_G) \subseteq (\Glo_{/G}, \Glo_{/G} \times_\Glo \Orb)$
            to obtain the extensive Borel inclusion $(\F_G,\F_G) \subseteq (\Fglo_{/G}, \FGlo_{/G} \times_{\Fglo}\Forb)$.
            This extends (genuinely) $G$-equivariant objects to $G$-global objects.

        \item As a special case of the previous point,
            we have the extensive Borel inclusion $(\F,\F) \subseteq (\Fglo,\Forb)$, which extends objects to global ones.
    \end{enumerate}
\end{example}

\begin{lemma}\label{lem:borel-inc-span-ff}
    Let $i \colon (\bsf, \bsn) \subseteq (\blf,\bln)$ be a Borel inclusion.
    Then the induced functor $\Span(i) \colon \Span_{\bsn}(\bsf) \to \Span_{\bln}(\blf)$ is fully faithful.
\end{lemma}
\begin{proof}
    Recall that $\map_{\Span_\bsn(\bsf)}(X,Y) \simeq (\bsf_{/X} \times_\bsf \bsn_{/Y})^\simeq$. We claim that under the given hypotheses we already have an equivalence of categories
    \[
        i \colon \bsf_{/X} \times_{\bsf} \bsn_{/Y} \to \blf_{/X} \times_{\blf} \bln_{/Y}
    \]
    for all $X,Y \in \bsf$.
    Since $\bsf \subseteq \blf$ is fully faithful, also $\bsf_{/X} \subseteq \blf_{/X}$ is for any $X \in \bsf$.
    Moreover, the sieve condition implies that the inclusion $\bsn_{/Y} \to \bln_{/Y}$
    is in fact an equivalence for all $Y \in \bsf$.
    This yields essential surjectivity of $i$,
    and we get fully faithfulness from the fact that fully faithful functors
    are closed under limits in $\Ar(\Cat)$.
\end{proof}

\begin{definition}
    Given a Borel inclusion $i \colon (\bsf,\bsn) \subseteq (\blf,\bln)$,
    right Kan extension along $i$ respectively $\Span(i)$
    defines fully faithful 2-functors (c.f.~\cref{rem:2cat})
    \[
        (-)^\flat \coloneqq i_* \colon
        \PreCat{\bsf} \subseteq \PreCat{\blf}
        \quad\text{and}\quad
        (-)^\flat \coloneqq \Span(i)_* \colon \PreNmCat{\bsf}{\bsn} \subseteq \PreNmCat{\blf}{\bln}
    \]
    which we will refer to as \emph{Borelification} respectively \emph{monoidal Borelification.}
\end{definition}

\begin{lemma}\label{lem:borel-product}
    If $i$ is even an extensive Borel inclusion,
    the above functors restrict to
    \[
        (-)^\flat \colon \PCat{\bsf} \subseteq \PCat{\blf}
        \quad\text{and}\quad
        (-)^\flat \colon \Mack_\bsf^\bsn(\Cat) \subseteq \Mack_\blf^\bln(\Cat).
    \]
\end{lemma}
\begin{proof}
    Note that this is obvious when the Borel inclusion is of the form discussed in \cref{ex:free-cop-borel}.
    We will postpone the proof of the case of the monoidal Borelification to the proof of \cref{prop:borel-ran-img} below.

    Let $\cC \in \PCat{\cE}$.
    For $X \in \cF$ we write $\cE_{/X} \coloneqq \cE \times_\cF \cF_{/X}$
    and $\pi_X \colon \cE_{/X} \to \cE$ for the projection.
    Let $A,B \in \cF$ and consider the coproduct inclusions $j_A \colon A \to A \amalg B$ and $j_B \colon B \to A \amalg B$.
    By the pointwise formula for right Kan extensions we can describe
    the map $\cC^\flat(A \amalg B) \to \cC^\flat(A) \times \cC^\flat(B)$ as
    \begin{equation}\label{eq:idkman}
        (\cE_{/j_A}^{\op,*}, \cE_{/j_B}^{\op,*})
        \colon \lim_{(\cE_{/A \amalg B})^\op} \pi_{A\amalg B}^*\cC
        \to \left(\lim_{(\cE_{/A})^\op} \pi_A^*\cC\right) \times \left(\lim_{(\cE_{/B})^\op} \pi_B^*\cC\right),
    \end{equation}
    given by restricting along $(\cE_{/j_A})^\op \colon (\cE_{/A})^\op \to (\cE_{/A \amalg B})^\op$
    and $(\cE_{/j_B})^\op$.
    By \cref{lem:borel-ext} and its proof, we note that
    pulling back along $j_A$ defines a right adjoint $j_A^* \colon \cE_{/A \amalg B} \to \cE_{/A}$ of $\cE_{/j_A}$,
    and analogously for $B$.
    Thus $(j_A^*)^\op$ and $(j_B^*)^\op$ are limit-cofinal, and we can postcompose (\ref{eq:idkman}) to get
    \begin{align*}
        \lim_{(\cE_{/A \amalg B})^\op} \pi_{A\amalg B}^*\cC
        \xto{(\cE_{/j_A}^{\op,*},\cE_{/j_B}^{\op,*})} &\left(\lim_{(\cE_{/A})^\op} \pi_A^*\cC\right) \times \left(\lim_{(\cE_{/B})^\op} \pi_B^*\cC\right)\\
        \xto[\simeq]{((j_A^*)^{\op,*},(j_B^*)^{\op,*})} &\left(\lim_{(\cE_{/A\amalg B})^\op} \cC(\pi_A(j_A^*-))\right) \times \left(\lim_{(\cE_{/A \amalg B})^\op} \cC(\pi_B(j_B^*-))\right)\\
        \simeq\ &\lim_{X \in (\cE_{/A \amalg B})^\op} \cC(\pi_A(j_A^*X)) \times \cC(\pi_B(j_B^*X)).
    \end{align*}
    The entire composite is induced by a map
    $\phi \colon \cC(\pi_{A \amalg B}-) \Rightarrow \cC(\pi_{A}j_A^*-) \times \cC(\pi_Bj_B^*-)$
    in $\Fun^\times((\cE_{/A \amalg B})^\op,\Cat)$,
    which at $X \to A \amalg B$ comes from the summand inclusions $X \times_{A \amalg B} A \to X$
    and likewise for $B$.
    We claim that this is in fact an equivalence, so that by 2-out-of-3 also (\ref{eq:idkman})
    is an equivalence.

    To this end, note that $\cC(\pi_{A \amalg B}-) \Rightarrow \cC(\pi_Aj_A^*-)$
    is the unit of the adjunction induced by precomposing with the adjunction
    $\cE_{/j_A} \colon\cE_{A} \rightleftarrows \cE_{A \amalg B} \noloc j_A^*$, evaluated at $\pi_{A \amalg B}^*\cC$.
    By \cite[Theorem B]{Horev-Yanovski} it follows that also the functor
    \[
        (\cE_{/j_A}^{\op,*},\cE_{/j_B}^{\op,*}) \colon \Fun^\times((\cE_{/A \amalg B})^\op,\Cat)
        \to \Fun^\times((\cE_{/A})^\op,\Cat) \times \Fun^\times((\cE_{/B})^\op,\Cat)
    \]
    admits a right adjoint sending $(\cA,\cB)$ to $\cA(j_A^*-) \times \cB(j_B^*-)$,
    and that the unit at $\pi_{A \amalg B}^*\cC$ is precisely $\phi$.
    We claim that this is actually an adjoint equivalence, hence $\phi$ is an equivalence.

    Namely, this is expressing the semiadditivity of $\Cat^\times$;
    for any $\cC,\cD \in \Cat^\times$ the maps $j_\cC \coloneqq (-,*) \colon \cC \to \cC \times \cD$
    and $j_\cD \coloneqq (*,-) \colon \cD \to \cC \times \cD$ exhibit $\cC \times \cD$ as coproduct of $\cC,\cD$ in $\Cat^\times$. Indeed,
    $(j_\cC^*,j_\cD^*) \colon \Fun^\times(\cC \times \cD,\cB) \to \Fun^\times(\cC,\cB) \times \Fun^\times(\cD,\cB)$
    is inverse to sending $(F,G)$ to $F \times G \colon \cC \times \cD \to \cB \times \cB \xto{\times} \cB$
    by direct inspection.
\end{proof}

The following proposition characterizes the image of the monoidal Borelification,
and shows that writing $(-)^\flat$ for both kinds of Borelifications is unproblematic.

\begin{proposition}\label{prop:borel-ran-img}
    Given a Borel inclusion $(\bsf,\bsn) \subseteq (\blf,\bln)$
    the Beck--Chevalley transformation in the following left square is an equivalence:
    \[\begin{tikzcd}[cramped]
        {\PreCat{\cE}} & {\PreNmCat{\cE}{\cM}} & {\PCat{\cE}} & {\Mack_\bsf^\bsn(\Cat)} \\
        {\PreCat{\cF}} & {\PreNmCat{\cF}{\cN}} & {\PCat{\cF}} & {\Mack_\blf^\bln(\Cat)}
        \arrow["{(-)^\flat}"', hook, from=1-1, to=2-1]
        \arrow["\res"', from=1-2, to=1-1]
        \arrow["{(-)^\flat}", hook', from=1-2, to=2-2]
        \arrow["{(-)^\flat}"', hook, from=1-3, to=2-3]
        \arrow["\res"', from=1-4, to=1-3]
        \arrow["{(-)^\flat}", hook', from=1-4, to=2-4]
        \arrow["BC"{description}, Rightarrow, from=2-2, to=1-1]
        \arrow["\res", from=2-2, to=2-1]
        \arrow["BC"{description}, Rightarrow, from=2-4, to=1-3]
        \arrow["\res", from=2-4, to=2-3]
    \end{tikzcd}\]
    If $i$ is furthermore an extensive Borel inclusion,
    then also the right square exists and commutes via the Beck--Chevalley map.
    In particular, an $\bln$-normed $\blf$-(pre)category
    is Borel if and only if its underlying $\blf$-(pre)category is Borel.
    In other words, the above squares are cartesian.
\end{proposition}
\begin{proof}
    Given the case for normed precategories,
    we can provide the proof that monoidal Borelification restricts to normed categories,
    as claimed in \cref{lem:borel-product}.
    Namely, recall from \cref{prop:span-semiadd} that a functor $\Span_\cN(\cF) \to \Cat$ preserves finite products
    if and only if its restriction to $\cF^\op$ does.
    Now for $\cC \in \Mack_\cE^\cM(\Cat)$ the left square above tells us that $\cC^\flat|_{\cF^\op} \simeq (\cC|_{\cE^\op})^\flat$, and since the latter is actually an $\cF$-category
    by the first case of \cref{lem:borel-product},
    we see that also $\cC^\flat \in \Mack_\cF^\cN(\Cat)$.
    From this also the commutativity of the right square above immediately follows from that of the left one,
    so it remains to argue the case of normed precategories.

    By the pointwise formula for right Kan extensions,
    we need to show that the following functor is limit-cofinal for each $X \in \blf^\op$:
    \[
        \phi \colon \bsf^\op  \times_{\blf^\op} (\blf^\op)_{X/} \to \Span_{\bsn}(\bsf) \times_{\Span_{\bln}(\blf)} \Span_{\bln}(\blf)_{X/}.
    \]
    Using the mapping space formula
    $\Span_{\bln}(\blf)(A,B) \simeq (\cF_{/A} \times_\cF \cN_{/B})^\simeq$
    one easily checks that the inclusion $L \colon (\blf^\op)_{X/} \to \Span_{\bln}(\blf)_{X/}$
    admits a right adjoint $R$ which sends
    a span $(X \leftarrow B \to A)$ to $X \leftarrow B$.
    Now on the above full subcategories,
    $L$ restricts to $\phi$, and also $R$ will restrict to a functor right adjoint to $\phi$;
    If $\alpha = (X \leftarrow B \to A)$ is an object in the target
    of $\phi$, then $A \in \bsf$ and $B \to A$ lies in $\bln$,
    so by definition of a Borel inclusion we have that $B$ lies in $\bsf$,
    hence $R(\alpha) = (X \leftarrow B)$ lies in
    $\bsf^\op \times_{\blf^\op} (\blf^\op)_{X/}$ as well.
    This proves that $\phi$ is a left adjoint,
    hence limit-cofinal, and so we are done.

    For the last statement, let us fix the following notation for the inclusions
    \[\begin{tikzcd}[ampersand replacement=\&,cramped]
        {\bsf^\op} \& {\Span_\bsn(\bsf)} \\
        {\blf^\op} \& {\Span_\bln(\blf)}
        \arrow["k", from=1-1, to=1-2]
        \arrow["i"', hook, from=1-1, to=2-1]
        \arrow["I", hook', from=1-2, to=2-2]
        \arrow["\ell"', from=2-1, to=2-2]
    \end{tikzcd}\]
    and consider $\cC^\tensor \colon \Span_\bln(\blf) \to \Cat$
    such that $\cC \coloneqq \ell^*\cC^\tensor$ is right Kan extended from $\bsf^\op$,
    i.e.~such that $\eta^i_{\cC} \colon \cC \to i_*i^*\cC$ is an equivalence.
    We need to show that $\cC^\tensor$ is right Kan extended from $\Span_\bsn(T)$,
    i.e.~that $\eta^I_{\cC^\tensor} \colon \cC^\tensor \to I_*I^*\cC^\tensor$ is an equivalence.
    Note that $\ell^*$ is conservative as $\ell$ is essentially surjective,
    so it suffices to check that $\ell^*\eta^I_{\cC^\tensor}$ is an equivalence.
    This follows from the following commutative square,
    which exists by general facts on Beck--Chevalley maps
    (see \cite[Lemma C.2]{CLL-Adams}):
    \[\begin{tikzcd}[cramped]
        {i_*k^*I^*\cC^\tensor} & {i_*i^*\ell^*\cC^\tensor} \\
        {\ell^*I_*I^*\cC^\tensor} & {\ell^*\cC^\tensor}
        \arrow["\simeq"', from=1-2, to=1-1]
        \arrow["(BC)I_*", from=2-1, to=1-1]
        \arrow["\simeq"', draw=none, from=2-1, to=1-1]
        \arrow["{\eta^i_{\ell^*\cC^\tensor}}"', from=2-2, to=1-2]
        \arrow["\simeq", draw=none, from=2-2, to=1-2]
        \arrow["{\ell^*\eta^I_{\cC^\tensor}}"', from=2-2, to=2-1]
    \end{tikzcd}\]
\end{proof}

\begin{remark}\label{rem:monoid-unique-lift}
    Taking horizontal fibers over some $\bsf$-(pre)category $\cC$
    in the above pullback squares tells us that Borelification
    induces an equivalence between
    the spaces of $\bsn$-normed structures on $\cC$
    and $\bln$-normed structures on $\cC^\flat$.
    In particular, an $\bsn$-normed structure on $\cC$
    lifts \emph{uniquely} to an $\bln$-normed structure on $\cC^\flat$.
\end{remark}

Another crucial property of monoidal Borelifications is that they essentially
don't change categories of normed algebras.

\begin{proposition}\label{prop:borel-sections}
    For a Borel inclusion
    $i \colon (\bsf,\bsn) \subseteq (\blf,\bln)$
    and $\cN$-normed $\cF$-precategory $\cC^\tensor$,
    restriction along $\Span(i)$ induces
    a natural equivalence of $\blf$-precategories compatible with the forgetful functors:
    \[\begin{tikzcd}[cramped]
        {\und{\CAlg}_{\blf}^\bln(\cC^{\tensor,\flat})} && {\und{\CAlg}_{\bsf}^\bsn(\cC^\tensor)^\flat} \\
        & {\cC^\flat}
        \arrow["{\Span(i)^*}", from=1-1, to=1-3]
        \arrow["\simeq"', draw=none, from=1-1, to=1-3]
        \arrow["\bbU"', from=1-1, to=2-2]
        \arrow["{\bbU^\flat}", from=1-3, to=2-2]
    \end{tikzcd}\]
    where we identify the underlying $\cF$-precategory of $\cC^{\tensor,\flat}$ with $\cC^\flat$ via \cref{prop:borel-ran-img}.
\end{proposition}
\begin{proof}
    Recall that by definition we have natural equivalences
    \[
        \und{\CAlg}_\cF^\cN(\cC^{\tensor,\flat})(A)
        \simeq \Fun_\cF^{\Nstr}(\Env(\und{A}^\Ncoprod), \cC^{\tensor,\flat}).
    \]
    Similarly, using the adjunction $i^* \dashv i_* = (-)^\flat$
    we have natural equivalences
    \begin{equation}\label{eq:borel-calg}
        \und{\CAlg}_\cE^\cM(\cC^{\tensor})^\flat(A)
        \simeq \Fun_\cE(i^*\und{A}, \und{\CAlg}_\cE^\cM(\cC^\tensor))
        \simeq \Fun_\cE^{\cM\dtensor}(\Env((i^*\und{A})^{\cM\damalg}),\cC^\tensor).
    \end{equation}
    For $\cD \in \Cat_\cF$ it is straightforward to check that
    $(i^*\cD)^{\bsn\damalg} \simeq \Span(i)^*(\cD^{\bln\damalg})$
    (this does not even require $i$ to be a Borel inclusion).
    In view of \cref{lem:unfurl} we then have a commutative
    cube where all faces are pullbacks:
    \[\begin{tikzcd}[cramped,sep=small]
        & {\Triv_\bsf^\bsn(i^*\cD)} && {(i^*\cD)^{\bsn\damalg}} \\
        {\Triv_\blf^\bln(\cD)} && {\cD^{\bln\damalg}} \\
        & \bsf && {\Span_\bsn(\bsf)} \\
        \blf && {\Span_\bln(\blf)}
        \arrow[from=1-2, to=1-4]
        \arrow[hook, from=1-2, to=2-1]
        \arrow[from=1-2, to=3-2]
        \arrow[hook', from=1-4, to=2-3]
        \arrow[from=1-4, to=3-4]
        \arrow[from=2-1, to=2-3]
        \arrow[from=2-1, to=4-1]
        \arrow[from=2-3, to=4-3]
        \arrow[from=3-2, to=3-4]
        \arrow["i"', hook, from=3-2, to=4-1]
        \arrow["{\Span(i)}", hook', from=3-4, to=4-3]
        \arrow[from=4-1, to=4-3]
    \end{tikzcd}\]
    Since the top horizontal maps in this cube induce the natural transformations
    $\Triv_\bsf^\bsn \Rightarrow (-)^{\bsn\damalg}$ and similarly for $(\blf,\bln)$,
    the left and right faces respectively entire volume of the above cube
    tells us that the back and front faces respectively entire volume
    of the right part in the following diagram commute:
        \[\begin{tikzcd}[cramped]
        {\PreNmCat{\bsf}{\bsn}} && {\PreOp{\bsf}{\bsn}} &&& {\PreCat{\bsf}} \\
        \\
        {\PreNmCat{\blf}{\bln}} && {\PreOp{\blf}{\bln}} &&& {\PreCat{\blf}}
        \arrow["\Env"', from=1-3, to=1-1]
        \arrow[""{name=0, anchor=center, inner sep=0}, "{(-)^{\bsn\damalg}}"', curve={height=12pt}, from=1-6, to=1-3]
        \arrow[""{name=1, anchor=center, inner sep=0}, "\Triv", curve={height=-12pt}, from=1-6, to=1-3]
        \arrow["{\Span(i)^*}", from=3-1, to=1-1]
        \arrow["{\Span(i)^*}", from=3-3, to=1-3]
        \arrow["\Env", from=3-3, to=3-1]
        \arrow["{i^*}"', from=3-6, to=1-6]
        \arrow[""{name=2, anchor=center, inner sep=0}, "\Triv", curve={height=-12pt}, from=3-6, to=3-3]
        \arrow[""{name=3, anchor=center, inner sep=0}, "{(-)^{\bln\damalg}}"', curve={height=12pt}, from=3-6, to=3-3]
        \arrow[shorten <=3pt, shorten >=3pt, Rightarrow, from=1, to=0]
        \arrow[shorten <=3pt, shorten >=3pt, Rightarrow, from=2, to=3]
    \end{tikzcd}\]
    Moreover, using the definition of a Borel inclusion, one may check that
    \[
        \Ar_\bsn(\Span_\bsn(\bsf))
        \simeq \Span_\bsn(\bsf) \times_{\Span_\bln(\blf)} \Ar_\bln(\Span_\bln(\blf)).
    \]
    By pullback pasting, this proves that also the left square commutes, c.f.~\cref{def:envelope}.
    In particular, we can extend the natural equivalences (\ref{eq:borel-calg}) to
    \[
        \und{\CAlg}_\cE^\cM(\cC^\tensor)^\flat(A)
        \simeq \Fun_\cE^{\cM\dtensor}(\Span(i)^*\Env(\und{A}^{\Ncoprod}), \cC^\tensor)
    \]
    Since $\Span(i)^*$ is a 2-functor and $\Span(i)^*\cC^{\tensor,\flat} \simeq \cC^\tensor$ via the counit, by adjunction we ultimately obtain
    the following commutative diagram
    \[\begin{tikzcd}[cramped]
        {\und{\CAlg}_\cF^\cN(\cC^{\tensor,\flat})(A)} && {\und{\CAlg}_\cE^\cM(\cC^\tensor)^\flat(A)} \\
        {\Fun_{\cF}^{\cN-\tensor}(\Env(\und{A}^{\cN-\amalg})),  \cC^{\tensor,\flat})} && {\Fun_{\cE}^{\cM-\tensor}(\Span(i)^*\Env(\und{A}^{\cN-\amalg}), \cC^\tensor)} \\
        {\Fun_{\cF}^{\cN-\tensor}(\Env(\Triv(\und{A})),  \cC^{\tensor,\flat})} && {\Fun_{\cE}^{\cM-\tensor}(\Span(i)^*\Env(\Triv(\und{A})), \cC^\tensor)} \\
        {\Fun_{\cF}(\und{A},  \cC^{\flat})} & {\cC^\flat(A)} & {\Fun_\cE(i^*\und{A},\cC)}
        \arrow["\simeq", from=2-1, to=1-1]
        \arrow["\simeq"', from=2-1, to=2-3]
        \arrow["{\Span(i)^*}", draw=none, from=2-1, to=2-3]
        \arrow["{\Env(\incl)^*}"', from=2-1, to=3-1]
        \arrow["\simeq", from=2-3, to=1-3]
        \arrow["{(\Span(i)^*\Env(\incl))^*}", from=2-3, to=3-3]
        \arrow["{\Span(i)^*}"', from=3-1, to=3-3]
        \arrow["\simeq", draw=none, from=3-1, to=3-3]
        \arrow["\simeq"', from=3-1, to=4-1]
        \arrow["\simeq", from=3-3, to=4-3]
        \arrow["\simeq"', from=4-1, to=4-2]
        \arrow["\simeq", from=4-3, to=4-2]
    \end{tikzcd}\]
    Note that the left vertical composite $\und{\CAlg}_\cF^\cN(\cC^{\tensor,\flat})(A)
    \to \cC^\flat(A)$ is by definition $\bbU(A)$.
    The right vertical composite identifies under $\Span(i)^*\Env(\Triv(\und{A}))
    \simeq \Env(\Triv(i^*\und{A}))$ and $\Env\Triv \dashv \fgt$ with
    \[
        \und{\CAlg}_\cE^\cM(\cC^\tensor)^\flat(A)
        \simeq \Fun_\cE(i^*\und{A}, \und{\CAlg}_\cE^\cM(\cC^\tensor))
        \xto[\simeq]{\bbU_*} \Fun_\cE(i^*\und{A},\cC)
        \simeq \cC^\flat(A),
    \]
    i.e.~with $\bbU^\flat(A)$, which yields the claim.
\end{proof}

\subsection{Equivariant and Global Borel Theory}\label{subsec:borel-eq}

In this subsection, we want to take a closer look
at the eponymous extensive Borel inclusion $(BG,BG) \to (\Orb_G,\Orb_G)$.
Let us begin by fixing some notation and conventions.
Recall that $\F[BG] \subseteq \PSh(BG) = \An^{BG^\op}$ is the free finite coproduct completion of $BG$,
and can equivalently be viewed as the category of finite free (right) $G$-sets.
In particular it is a full subcategory of the category of finite $G$-sets $\F_G = \F[\Orb_G]$.
For ease of notation, we write $\Span(G) = \Span(\F_G)$ and
\[
    \Mack_G(-) \coloneqq \Fun^\times(\Span(G),-)
    \quad\text{and}\quad
    \Mack_{BG}(-) \coloneqq \Fun^\times(\Span(\F[BG]),-).
\]
The above Borel inclusion thus induces a fully faithful functor $\Span(\F[BG]) \subseteq \Span(G)$, and right Kan extending along it defines the $G$-equivariant Borelification
\[
    (-)^\flat_G \colon \Mack_{BG}(\Cat) \subseteq \Mack_G(\Cat).
\]
If clear from context, we will leave out the subscript and just write $(-)^\flat$.

\begin{remark}\label{rem:left-right}
    The theory of normed categories as we recalled it in \cref{sec:algebra}
    is set up in such a way that the underlying $\F[BG]$-category of a normed category
    $\cC^\tensor \in \Mack_{BG}(\Cat)$ lies in $\PCat{\F[BG]} = \Fun^\times(\F[BG]^\op,\Cat) \simeq \Cat^{BG^\op}$,
    i.e.~is a category with \emph{right} $G$-action.
    Similarly, all the related constructions such as the functor
    $\und{\CAlg}_{BG} \colon \Mack_{BG}(\Cat) \to \PCat{\F[BG]} \simeq \Cat^{BG^\op}$
    take values in categories with right $G$-action by definition.
    To reduce the amount of $(-)^\op$'s occurring,
    we will implicitly pass between object with left and right $G$-action
    using the natural equivalence $BG \simeq BG^\op, g \mapsto g^{-1}$
    which induces an equivalence $\Cat^{BG} \simeq \Cat^{BG^\op}$ natural in $G \in \Glo^\op$.
    In particular, we will consider the forgetful functor
    $\fgt \coloneqq \fgt_{BG} \colon \Mack_{BG}(\Cat) \to \Cat^{BG}$ defined as restriction along the diagonal map in the following commutative diagram of inclusions:
    \[\begin{tikzcd}[cramped]
        BG && {\F[BG]} \\
        {BG^\op} & {\F[BG]^\op} & {\Span(\F[BG])}
        \arrow[hook, from=1-1, to=1-3]
        \arrow["\simeq"', from=1-1, to=2-1]
        \arrow[from=1-1, to=2-3]
        \arrow[from=1-3, to=2-3]
        \arrow[hook, from=2-1, to=2-2]
        \arrow[from=2-2, to=2-3]
    \end{tikzcd}\]
\end{remark}

We begin by noting that normed $\F[BG]$-categories
are really just symmetric monoidal categories with a $G$-action.

\begin{proposition}\label{prop:bg-normed}
    There is an equivalence of 2-categories $\Mack_{BG}(\Cat) \simeq \Mack(\Cat)^{BG}$
    making the following diagram commute:
    \[\begin{tikzcd}[cramped]
        & {\Cat^{BG}} \\
        {\Mack_{BG}(\Cat)} && {\Mack(\Cat)^{BG}} \\
        & {\Mack(\Cat)}
        \arrow["{\fgt_{BG}}", from=2-1, to=1-2]
        \arrow["\simeq", from=2-1, to=2-3]
        \arrow["\res"', from=2-1, to=3-2]
        \arrow["{\fgt_*}"', from=2-3, to=1-2]
        \arrow["\res", from=2-3, to=3-2]
    \end{tikzcd}\]
    Here the left $\res$ is given by restricting along the map
    $\Span(\F) \to \Span(\F[BG])$ induced by $* \to BG$,
    and $\fgt_{BG}$ is as just defined in the previous remark.
    Moreover, we have the following commutative diagram
    \[\begin{tikzcd}[cramped]
        {\Mack_{BG}(\Cat)} & {\Mack_G(\Cat)} \\
        {\Mack(\Cat)^{BG}} & {\Mack_G(\Mack(\Cat))} & {\Mack(\Cat)}
        \arrow["{(-)^\flat}", hook', from=1-1, to=1-2]
        \arrow["\simeq"', from=1-1, to=2-1]
        \arrow["{s_H^*}", curve={height=-12pt}, from=1-2, to=2-3]
        \arrow["{I_*}", hook, from=2-1, to=2-2]
        \arrow["{(-)^{hH}}"{description}, curve={height=18pt}, from=2-1, to=2-3]
        \arrow["{\fgt_*}", from=2-2, to=1-2]
        \arrow["\simeq"', draw=none, from=2-2, to=1-2]
        \arrow["{\ev_{G/H}}", from=2-2, to=2-3]
    \end{tikzcd}\]
    where $s_H \colon \Span(\F) \to \Span(G)$
    is induced by $G/H \colon * \to \Orb_G$
    and $I$ is the fully faithful inclusion
    $BG \subseteq \Span(\F[BG]) \subseteq \Span(G)$.
\end{proposition}
\begin{proof}
    We define the equivalence as the composite
    \[
        \Mack_{BG}(\Cat)
        \xleftarrow[\simeq]{\fgt_*} \Mack_{BG}(\Mack(\Cat))
        \xrightarrow[\simeq]{j^*} \Mack(\Cat)^{BG}
    \]
    where $j \colon BG \subseteq \F[BG] \to \Span(\F[BG])$ is the inclusion.
    Since $\fgt$ is defined via restriction, it is clear from \cref{rem:2cat}
    that both $\fgt_*$ and $j^*$ are 2-functors.
    Recall from \cref{prop:mack} that $\Mack(-) \colon \Cat^\times \to \Cat^\sadd$
    is a right Bousfield localization with counit $\fgt$
    given by evaluating at $1 \in \Span(\F)$,
    and that the two forgetful forgetful functors
    $\fgt_*,\fgt \colon \Mack(\Mack(-)) \Rightarrow \Mack(-)$ are homotopic.
    Since $\Span(\F[BG])$ is semiadditive, this shows that the first map
    is an equivalence.
    The fact that $j^*$ is an equivalence is an instance of a theorem originally due to Glasman \cite[Theorem A.1]{Glasman-Stratified}
    which shows that $j$ exhibits $\Span(\F[BG])$ as the free semiadditive category on $BG$. This is also a special case of \cite[Theorem 5.29]{Harpaz}.

    Hence, everything regarding the first diagram
    follows once we verify that the following diagram commutes:
    \[\begin{tikzcd}[cramped]
        & {\Cat^{BG}} \\
        {\Mack_{BG}(\Cat)} & {\Mack_{BG}(\Mack(\Cat))} & {\Mack(\Cat)^{BG}} \\
        {\Mack(\Cat)} & {\Mack(\Mack(\Cat))} & {\Mack(\Cat)}
        \arrow["{\fgt_{BG} = j^*}", from=2-1, to=1-2]
        \arrow["\res"', from=2-1, to=3-1]
        \arrow["{\fgt_*}"', from=2-2, to=2-1]
        \arrow["\simeq", draw=none, from=2-2, to=2-1]
        \arrow["{j^*}", from=2-2, to=2-3]
        \arrow["\simeq"', draw=none, from=2-2, to=2-3]
        \arrow["\res"', from=2-2, to=3-2]
        \arrow["{\fgt_*}"', from=2-3, to=1-2]
        \arrow["\res"', from=2-3, to=3-3]
        \arrow[curve={height=18pt}, Rightarrow, no head, from=3-1, to=3-3]
        \arrow["{\fgt_*}"', from=3-2, to=3-1]
        \arrow["\simeq", draw=none, from=3-2, to=3-1]
        \arrow["\fgt", from=3-2, to=3-3]
        \arrow["\simeq"', draw=none, from=3-2, to=3-3]
    \end{tikzcd}\]
    Commutativity of the top triangle and left square is clear.
    The bottom triangle commutes by the fact that $\fgt_* \simeq \fgt$ as mentioned above,
    and the right square is given by precomposing with the equivalent composites $* \to BG \subseteq \Span(\F[BG])$
    and $* \subseteq \Span(\F) \to \Span(\F[BG])$.

    Verifying the commutativity of the second diagram in the proposition is similar;
    commutativity of the bottom triangle comes from the pointwise formula
    of right Kan extension along $I$,
    and the right triangle commutes using the facts
    $\fgt \circ s_H^* \simeq \ev_{G/H}$ and $\fgt_* \simeq \fgt \colon \Mack(\Mack(\Cat)) \to \Mack(\Cat)$.
    The square commutes by definition of the equivalence $\Mack_{BG}(\Cat) \simeq \Mack(\Cat)^{BG}$, the fact that the inverse of $j^*$ is the right Kan extension $j_*$, and that $I_* = (-)^\flat \circ j_*$.
\end{proof}

When no confusion can arise, we will use the above to implicitly identify the categories
$\Mack_{BG}(\Cat)$ and $\Mack(\Cat)^{BG}$, and refer to either of them
as the category of symmetric monoidal categories with $G$-action.

\begin{remark}\label{rem:borel-eqv-functoriality}
    Let $\cC^\tensor$ be a symmetric monoidal category with $G$-action.
    The functoriality of $\cC^{\tensor,\flat} \colon \Span(G) \to \Cat$ can be described as follows:
    \begin{enumerate}
        \item By \cref{prop:borel-ran-img} the underlying $G$-category
            of $\cC^{\tensor,\flat}$ agrees with $\cC^\flat$.
            The pointwise formula for right Kan extensions together with the pullback squares
            \[\begin{tikzcd}[cramped]
                BH & {(\Orb_G)_{/(G/H)}} \\
                BG & {\Orb_G}
                \arrow[hook, from=1-1, to=1-2]
                \arrow[from=1-1, to=2-1]
                \arrow["\lrcorner"{anchor=center, pos=0.125}, draw=none, from=1-1, to=2-2]
                \arrow[from=1-2, to=2-2]
                \arrow[hook, from=2-1, to=2-2]
            \end{tikzcd}\]
            then yields that the restriction $\cC^{\tensor,\flat}(G/H \leftarrow G/K = G/K)
            \simeq \cC^\flat(G/K \to G/H)$ agrees with the canonical (Beck--Chevalley) map
            $\cC^{hH} \to \cC^{hK}$.

        \item By \cref{prop:bg-normed}, we see that the restriction
            of $\cC^{\tensor,\flat}$ along the functor $s_H \colon \Span(\F) \to \Span(G)$ from above
            encodes the symmetric monoidal
            on $\cC^{hH}$ given by $(\cC^\tensor)^{hH}$.
            This describes the functoriality in fold maps.

        \item It remains to describe the functoriality in forwards maps of the form
            \[
                \cC^{\tensor,\flat}(G/K = G/K \to G/H) : \cC^{hK} \to \cC^{hH}.
            \]
            This lifts the twisted tensor product
            $X \mapsto \bigotimes_{[h] \in H/K} h.X$,
            as one checks immediately by computing the composition
            \[\begin{tikzcd}
                & {G/K} && G/e \\
                {G/K} && {G/H} && G/e
                \arrow[Rightarrow, no head, from=1-2, to=2-1]
                \arrow[from=1-2, to=2-3]
                \arrow[from=1-4, to=2-3]
                \arrow[Rightarrow, no head, from=2-5, to=1-4]
            \end{tikzcd}\]
            in $\Span(G)$, c.f.~\cite[Observation 3.2.1]{Hilman-phd}.
            It was shown in \cite[Proposition A.3.2]{Tambara}
            that if $\cC^\tensor$ is an ordinary symmetric monoidal 1-category
            with $G$-action, then there is a \emph{unique} natural transformation
            $\cC^{hK} \to \cC^{hH}$ lifting the twisted tensor product.
            Since the classical symmetric monoidal norm construction provides
            such a natural transformation, it must agree
            with the covariant functoriality of
            $\cC^{\tensor,\flat}$ in the case
            that $\cC^\tensor$ is an ordinary symmetric monoidal 1-category.
    \end{enumerate}
\end{remark}

In fact, the same argument as in \cite[Proposition A.3.2]{Tambara},
carried out in a 2-categorical setting, yields a uniqueness result
for the functors $\cC^\flat(n_K^H) \colon \cC^{hK} \to \cC^{hH}$,
where $n_K^H$ denotes the map $G/K = G/K \to G/H$ in $\Span(G)$.
To state this result, recall the notion of a 2-natural transformation
from \cref{appendix:lax}. We have the following composite of 2-functors
\[
    \Mack(\Cat)^{BG}
    \simeq \Mack_{BG}(\Cat)
    \xto{(-)^\flat} \Mack_G(\Cat)
    \xto{\ev_{n_K^H}} \Ar(\Cat)
\]
which shows that $(-)^\flat(n_K^H) \colon (\fgt)^{hK} \simeq (-)^\flat(G/K) \Rightarrow (-)^\flat(G/H) \simeq (\fgt)^{hH}$ is a 2-natural transformation
of functors $\Mack(\Cat)^{BG} \to \Cat$.
We now show that it is uniquely characterized by lifting the $|H/K|$-fold
tensor product.

\begin{proposition}\label{prop:unique-norm}
    For $K \leq H \leq G$, there is a unique 2-natural transformation
    $(\fgt)^{hK} \Rightarrow (\fgt)^{hH}$
    of 2-functors $\Mack(\Cat)^{BG} \to \Cat$ lifting the $|H/K|$-fold tensor product.
\end{proposition}
\begin{proof}
    As noted above, one such natural transformation is given by $(-)^\flat(n_K^H)$,
    so it remains to see that there is at most one.
    We will prove this by noting that $(\fgt)^{hK}$ is corepresented
    and using the 2-categorical Yoneda lemma.
    Indeed, note that by adjunction we have natural equivalences
    \[
        (\fgt)^{hK}
        \simeq \Fun^\tensor(\F^{\sqcup,\simeq},(-)^{hK})
        \simeq \Fun^\tensor_{BK}(\infl_K\F^{\sqcup,\simeq},\res^G_K-)
        \simeq \Fun_{BG}^\tensor(\Ind_K^G\infl_K\F^{\sqcup,\simeq},-)
    \]
    where we also used that $\F^{\sqcup,\simeq}$ is the free symmetric monoidal
    category on one object, c.f.~\cref{ex:finset}.
    Note also that by semiadditivity of $\Mack(\Cat)^{BG}$ we have
    \[
        \fgt\Ind_K^G\infl_K \F^{\sqcup,\simeq}
        \simeq \fgt\Coind_K^G\infl_K\F^{\sqcup,\simeq}
        \simeq \Coind_K^G\infl_K\F^{\simeq}
    \]
    which is $\prod_{G/K}\F^\simeq \simeq (\F_{/(G/K)})^\simeq$
    (by extensiveness), equipped with a certain $G$-action.
    By the 2-categorical Yoneda lemma \cref{thm:2-yoneda}
    we now have equivalences of categories
    \[
        \Nat_2((\fgt)^{hK},(\fgt)^{hH}) \simeq (\F_{/(G/K)})^{\simeq,hH}
        \quad\text{and}\quad
        \Nat_2((\fgt)^{hK}, \fgt) \simeq (\F_{/(G/K)})^{\simeq}.
    \]
    By inspection, the transformation
    $(\fgt)^{hK} \xRightarrow{\pr} \fgt \xRightarrow{(-)^{\tensor|H/K|}} \fgt$
    corresponds to the object $(1 \to G/K)^{\tensor |H/K|} \cong (H/K \to G/K)$ in $(\F_{/(G/K)})^\simeq$ under the second equivalence.
    The category $N$ of 2-natural transformations $(\fgt)^{hK} \Rightarrow (\fgt)^{hH}$ lifting the $|H/K|$-fold tensor product is thus given by the following pullback:
    \[\begin{tikzcd}[cramped]
        N && {(\F_{/(G/K)})^{\simeq,hH}} \\
        {*} && {(\F_{/(G/K)})^\simeq}
        \arrow[from=1-1, to=1-3]
        \arrow[from=1-1, to=2-1]
        \arrow[from=1-3, to=2-3]
        \arrow[""{name=0, anchor=center, inner sep=0}, "{(H/K \to G/K)}"', from=2-1, to=2-3]
        \arrow["\lrcorner"{anchor=center, pos=0.125}, draw=none, from=1-1, to=0]
    \end{tikzcd}\]
    In particular, we see that $N$ is a space, and we claim it is contractible.
    Note that $H/K \to G/K$ does not admit any non-trivial automorphisms
    as an object in $\F_{/(G/K)}$ since it is a subobject of the terminal object $G/K$.
    Since we have already noted above that $N$ is non-empty,
    it thus suffices to show the following general fact:
    Let $\cC \in \Cat^{BG}$ and $X \in \cC$ with $\Aut_\cC(X) = \map_{\cC^\simeq}(X,X) \simeq *$.
    If the space $\cC^{hG}_X = \cC^{hG} \times_\cC \{X\}$ of lifts of $X$ to $\cC^{hG}$
    is nonempty, then it is contractible.
    To see this, note that $\cC^{hG} \to \cC$ is conservative since $BG$ only has a single object. Since $X \colon * \to \cC$ factors through $\cC^\simeq$,
    pullback pasting yields $\cC^{hG}_X \simeq \cC^{\simeq,hG}_X$.
    Now $X \colon * \to \cC^\simeq$ is fully faithful by assumption,
    so also $\cC^{\simeq,hG}_X \subseteq \cC^{\simeq,hG}$ is,
    and given two lifts $Y,Z \in \cC^{\simeq,hG}_X$, we compute
    \[
        \map_{\cC^{\simeq,hG}_X}(Y,Z)
        \simeq \map_{\cC^{\simeq,hG}}(Y,Z)
        \simeq \map_{\cC^\simeq}(X,X)^{hG}
        \simeq *^{hG}
        \simeq *
    \]
    using that the mapping space in a limit of categories is
    the limit of mapping spaces, and the underlying object of $Y$ and $Z$ is $X$.
    This shows that if $\cC^{hG}_X$ is non-empty, then it is contractible.
\end{proof}

\begin{remark}\label{rem:unique-norm}
    In view of the above uniqueness result,
    it is reasonable to say that the functoriality
    in forwards maps of the $G$-equivariant
    monoidal Borelification is given by `the' symmetric monoidal norms.
\end{remark}

We now come to a result
which gives a more concrete description of normed $G$-algebras in the monoidal Borelification of a symmetric monoidal category with $G$-action.
A version of this result has previously appeared in \cite[Theorem A]{Hilman-NMot},
although there the author works in the framework of parametrized higher algebra
as set up by Nardin--Shah \cite{Nardin-Shah},
the equivalence is only shown for the non-parametrized categories of algebras,
and without compatibility with the forgetful functors. We fix a finite group $G$.

\begin{theorem}\label{thm:borel-eqv}
    Let $\cC^\tensor$ be a symmetric monoidal category with $G$-action.
    We have a natural equivalence of categories with $G$-actions
    \begin{equation}\label[diag]{diag:borel-algebras}\begin{tikzcd}[cramped]
        {\und{\CAlg}_{BG}(\cC^\tensor)} && {\CAlg(\cC^\tensor)} \\
        & \cC
        \arrow["\simeq", from=1-1, to=1-3]
        \arrow["\bbU"', from=1-1, to=2-2]
        \arrow["\bbU", from=1-3, to=2-2]
    \end{tikzcd}\end{equation}
    The $G$-action on $\CAlg(\cC^\tensor)$ is inherited from the one on $\cC^\tensor$.
    Combining this with \cref{prop:borel-sections}, we obtain natural equivalences of $G$-categories
    \[\begin{tikzcd}[cramped]
        {\und{\CAlg}_G(\cC^{\tensor,\flat})} & {\und{\CAlg}_{BG}(\cC^\tensor)^\flat} & {\CAlg(\cC^\tensor)^\flat} \\
        & {\cC^\flat}
        \arrow["\simeq", from=1-1, to=1-2]
        \arrow["\bbU"', from=1-1, to=2-2]
        \arrow["\simeq", from=1-2, to=1-3]
        \arrow["{\bbU^\flat}", from=1-2, to=2-2]
        \arrow["{\bbU^\flat}", from=1-3, to=2-2]
    \end{tikzcd}\]
    At $G/H$ this gives an equivalence $\und{\CAlg}_G(\cC^{\tensor,\flat})(G/H) \simeq \CAlg(\cC^{\tensor,hH})$ over $\cC^{hH}$.
\end{theorem}

\begin{remark}
    One can show
    that in the fibrational picture,
    the equivalence $\und{\CAlg}_G(\cC^{\tensor,\flat})(G/H) \simeq \CAlg(\cC^{\tensor,hH})$
    is given by pulling back along $s_H \colon \Span(\F) \to \Span(G)$ induced
    by $* \mapsto G/H$. In other words, we have a commutative diagram
    \[\begin{tikzcd}[cramped]
        {\und{\CAlg}_G(\cC^{\tensor,\flat})(G/H)} & {\CAlg(\cC^{\tensor,hH})} \\
        {\Fun_{/\Span(G)}^{\F_G^\op\dcc}(\Span(G),\int \cC^{\tensor,\flat})} & {\Fun_{/\Span(\F)}^{\F^\op\dcc}(\Span(\F),\int\cC^{\tensor,hH})}
        \arrow["\simeq", from=1-1, to=1-2]
        \arrow["\simeq"', from=1-1, to=2-1]
        \arrow["\simeq", from=1-2, to=2-2]
        \arrow["{s^*_H}", from=2-1, to=2-2]
        \arrow["\simeq"', draw=none, from=2-1, to=2-2]
    \end{tikzcd}\]
    This was the original approach taken in the author's master's thesis,
    see \cite[Theorem 3.15]{puetzstueck}.
\end{remark}

For the proof of \cref{thm:borel-eqv} it remains to show the equivalence in \cref{diag:borel-algebras}.
The way we will go about this is to first give the description
of underlying categories, and then deduce the $G$-equivariant statement via a shifting trick.

\begin{lemma}\label{lem:borel-g-underlying}
    There is a natural equivalence $\CAlg_{BG}(\cC^\tensor) \simeq \CAlg(\cC^\tensor)^{hG} \simeq \CAlg(\cC^{\tensor,hG})$ compatible with the forgetful functors to $\cC^{hG}$.
\end{lemma}
\begin{proof}
    Recall our conventions of identifying $\Cat(\F[BG]) \simeq \Cat^{BG^\op}$ with $\Cat^{BG}$
    from \cref{rem:left-right}.
    Note that we have a commutative diagram
    \[\begin{tikzcd}[cramped]
        {\CAlg_{BG}(\cC^\tensor)} & {\Fun_{BG}(*,\und{\CAlg}_{BG}(\cC^\tensor))} & {\Fun_{BG}^\tensor(\Env((*)^{BG\damalg}), \cC^\tensor)} \\
        {\cC^{hG}} & {\Fun_{BG}(*,\cC)} & {\Fun^\tensor_{BG}(\Env(\Triv(*)), \cC^\tensor)}
        \arrow["\simeq", from=1-1, to=1-2]
        \arrow["\bbU"', from=1-1, to=2-1]
        \arrow["\simeq", from=1-2, to=1-3]
        \arrow["{\bbU_*}"', from=1-2, to=2-2]
        \arrow["{\incl^*}", from=1-3, to=2-3]
        \arrow["\simeq", from=2-1, to=2-2]
        \arrow["\simeq", from=2-2, to=2-3]
    \end{tikzcd}\]
    It thus suffices to identify the map
    $\incl \colon \Env(\Triv(*)) \to \Env((*)^{BG\damalg})$
    with $\infl_{G}(\F^{\sqcup,\simeq} \to \F^\sqcup)$
    under the equivalence of 2-categories from \cref{prop:bg-normed}.
    Indeed, by adjunction this will produce a commutative diagram
    \[\begin{tikzcd}[cramped]
        {\CAlg_{BG}(\cC^\tensor)} & {\Fun_{BG}^\tensor(\infl_G\F^{\sqcup,\simeq}, \cC^\tensor)} & {\Fun^\tensor(\F^{\sqcup},\cC^{\tensor,hG})} & {\CAlg(\cC^{\tensor,hG})} \\
        {\cC^{hG}} & {\Fun^\tensor_{BG}(\infl_G\F^{\sqcup}, \cC^\tensor)} & {\Fun^\tensor(\F^{\sqcup,\simeq},\cC^{\tensor,hG})} & {\cC^{hG}}
        \arrow["\simeq", from=1-1, to=1-2]
        \arrow["\bbU"', from=1-1, to=2-1]
        \arrow["\simeq", from=1-2, to=1-3]
        \arrow["{(\infl_G \inc)^*}"', from=1-2, to=2-2]
        \arrow["\simeq", from=1-3, to=1-4]
        \arrow["{\inc^*}"', from=1-3, to=2-3]
        \arrow["\bbU", from=1-4, to=2-4]
        \arrow["\simeq", from=2-1, to=2-2]
        \arrow["\simeq", from=2-2, to=2-3]
        \arrow["\simeq", from=2-3, to=2-4]
    \end{tikzcd}\]
    where the bottom horizontal composite is naturally equivalent to the identity
    by the fact that taking fixed points and forgetting the monoidal structure on $\cC$ commute.

    To this end, it is enough to show the following two things;
    point (1) shows that the underlying symmetric monoidal functor
    is $\F^{\sqcup,\simeq} \to \F^\sqcup$, and by point (2)
    it must have the trivial $G$-action:
    \begin{enumerate}
        \item The functor $\res \colon \Mack_{BG}(\Cat) \to \Mack(\Cat)$ forgetting the $G$-action
            sends $i$ to the inclusion $\inc \colon \F^{\sqcup,\simeq} \to \F^{\sqcup}$.

        \item The space of automorphisms of $\inc \in\Ar(\Mack(\Cat))$ is contractible.
    \end{enumerate}
    We begin with the second point.
    Recall that $\Fun^\tensor(\F^{\sqcup},\F^\sqcup) \simeq \CAlg(\F^\sqcup) \simeq \F$
    is induced by evaluating at $1 \in \F$. Since equivalences preserve terminal objects,
    we see that $\F^{\sqcup} \in \Mack(\Cat)$ has no non-trivial automorphisms.
    Next, note that $\inc$ is a monomorphism in $\Mack(\Cat)$,
    since each $\prod_n \F^\simeq \to \prod_n \F$ is one
    and monomorphisms in functor categories are pointwise.
    Thus we have a pullback with vertical monomorphisms
    \[\begin{tikzcd}[cramped]
        & {\Aut_{\Ar(\Mack(\Cat))}(\inc)} & {\map_{\Mack(\Cat)^{\simeq}}(\F^{\simeq,\sqcup},\F^{\simeq,\sqcup})} \\
        {*} & {\map_{\Mack(\Cat)^\simeq}(\F^{\sqcup},\F^{\sqcup})} & {\map_{\Mack(\Cat)}(\F^{\simeq,\sqcup},\F^\sqcup)}
        \arrow[from=1-2, to=1-3]
        \arrow[hook, from=1-2, to=2-2]
        \arrow["{\inc_*}", hook', from=1-3, to=2-3]
        \arrow["\simeq"', from=2-1, to=2-2]
        \arrow[""{name=0, anchor=center, inner sep=0}, "{\inc^*}"', from=2-2, to=2-3]
        \arrow["\lrcorner"{anchor=center, pos=0.125}, draw=none, from=1-2, to=0]
    \end{tikzcd}\]
    This shows $\Aut_{\Ar(\Mack(\Cat))}(\inc) \simeq *$.

    To see the first point, note that the cartesian unstraightening of
    $* \in \Fun^\times(\F[BG]^\op, \Cat)$ is simply the identity on $\F[BG]$, with all maps being cartesian.
    Thus $\Triv(*) \to (*)^{BG\damalg}$ is given by the inclusion $\F[BG]^\op \to \Span(\F[BG])$ over $\Span(\F[BG])$ and on envelopes we get the map $i$ in the following diagram
    \[\hspace{-2em}\begin{tikzcd}[cramped,column sep=small]
        {\int\F^{\sqcup,\simeq}} & {\Ar(\F)^{\pb,\op}} & {\Env(\Triv(*))|_{\Span(\F)}} & {\Ar(\F[BG])^{\pb,\op}} & {\F[BG]^\op} \\
        {\int \F^{\sqcup}} & {\Span_{\pb,\all}(\Ar(\F))} & {\Env(*^{BG\damalg})|_{\Span(\F)}} & {\Span_{\pb,\all}(\Ar(\F[BG]))} & {\Span(\F[BG])} \\
        && {\Span(\F)} & {\Span(\F[BG])}
        \arrow[Rightarrow, no head, from=1-1, to=1-2]
        \arrow[from=1-1, to=2-1]
        \arrow[dashed, from=1-2, to=1-3]
        \arrow[from=1-2, to=2-2]
        \arrow[from=1-3, to=1-4]
        \arrow["{i|_{\Span(\F)}}"', from=1-3, to=2-3]
        \arrow[from=1-4, to=1-5]
        \arrow["i"', from=1-4, to=2-4]
        \arrow["\lrcorner"{anchor=center, pos=0.125}, draw=none, from=1-4, to=2-5]
        \arrow[from=1-5, to=2-5]
        \arrow[Rightarrow, no head, from=2-1, to=2-2]
        \arrow[""{name=0, anchor=center, inner sep=0}, dashed, from=2-2, to=2-3]
        \arrow[from=2-2, to=3-3]
        \arrow[""{name=1, anchor=center, inner sep=0}, from=2-3, to=2-4]
        \arrow[from=2-3, to=3-3]
        \arrow["s", from=2-4, to=2-5]
        \arrow["t"', from=2-4, to=3-4]
        \arrow[""{name=2, anchor=center, inner sep=0}, from=3-3, to=3-4]
        \arrow["\lrcorner"{anchor=center, pos=0.125}, draw=none, from=1-2, to=0]
        \arrow["\lrcorner"{anchor=center, pos=0.125}, draw=none, from=1-3, to=1]
        \arrow["\lrcorner"{anchor=center, pos=0.125}, draw=none, from=2-3, to=2]
    \end{tikzcd}\]
    Here we repeatedly use that $\Span(-)$ preserves pullbacks,
    which shows that the top rightmost square and the rectangle consisting
    of the 3 top rightmost squares are pullbacks.
    The middle and bottom squares are pullbacks by definition,
    so by pullback pasting we get that all squares are pullbacks.

    Recall that the forgetful functor $\Mack_{BG}(\Cat) \to \Mack(\Cat)$
    is given by restricting along $\Span(\F) \to \Span(\F[BG])$,
    so it remains to see that the dashed maps are equivalences.
    Since the relevant square is a pullback,
    it suffices to check this for the bottom dashed map.
    For this, it is in turn enough to check that the comparison functor $\phi$ in the following diagram is an equivalence:
    \[\begin{tikzcd}[cramped]
        {\Ar(\F)} & P & {\Ar(\F[BG])} \\
        & \F & {\F[BG]}
        \arrow["\phi"{description}, from=1-1, to=1-2]
        \arrow[curve={height=-12pt}, from=1-1, to=1-3]
        \arrow["t"', from=1-1, to=2-2]
        \arrow[from=1-2, to=1-3]
        \arrow[from=1-2, to=2-2]
        \arrow["\lrcorner"{anchor=center, pos=0.125}, draw=none, from=1-2, to=2-3]
        \arrow["t", from=1-3, to=2-3]
        \arrow[from=2-2, to=2-3]
    \end{tikzcd}\]
    This is essentially another version of the equivalence $\F \simeq \F[BG]_{/G}$.
    Indeed, note that the above equivalence tells us that $\Ar(\F) \to \Ar(\F[BG])$ is essentially surjective;
    every map in $\F[BG]$ factors as an equivalence followed by a map in the image of $\F \to \F[BG]$.
    Since we can take $P \to \Ar(\F[BG])$ to be the identity on objects (analogous to $* \to BG$),
    it follows that $\phi$ is essentially surjective.
    To see that $\phi$ is fully faithful, we need to show that the following square is a pullback:
    \[\begin{tikzcd}[cramped]
        {\map_{\Ar(\F)}(n' \to n, m' \to m)} & {\map_{\Ar(\F[BG])}(\coprod_{n'} G \to \coprod_n G, \coprod_{m'}G \to \coprod_m G)} \\
        {\map_{\F}(n,m)} & {\map_{\F[BG]}(\coprod_n G, \coprod_m G)}
        \arrow[from=1-1, to=1-2]
        \arrow[from=1-1, to=2-1]
        \arrow[from=1-2, to=2-2]
        \arrow[from=2-1, to=2-2]
    \end{tikzcd}\]
    Taking the fiber over some map $n \to m$ in $\F$, this follows from the equivalence
    $\F_{/m} \simeq \F[BG]_{/\coprod_m G}$, which is in turn an instance
    of the fact that both $\F$ and $\F[BG]$ are extensive \cref{ex:extensive},
    that $\F[BG]_{/G} \simeq \F[BG_{/*}]$, and finally that $BG_{/*}$ is contractible (as is true
    for any slice $X_{/x}$ where $X$ is a space).
\end{proof}

For the following two proofs, it will be more convenient to work
with categories with $G$-action in the sense of finite product preserving functors
$\cC \colon \F[BG]^\op \to \Cat$,
since this allows us to consider the `shifted' $\F[BG]$-category $\cC(- \times G)$.
Recall that $\F[BG] \subseteq \An^{BG^\op}$ is by definition the category
of finite free \emph{right} $G$-sets. The Yoneda embedding $BG \subseteq \F[BG]$ defines
a left $G$-action on the right $G$-set $G$.

\begin{lemma}\label{lem:shear}
    Consider a category with right $G$-action $\cC \in \Fun^\times(\F[BG]^\op,\Cat)$.
    The left $G$-action on $G \in \F[BG]$ induces a right $G$-action on $\cC(- \times G)$, and this yields a natural equivalence
    \[
        \cC(- \times G)^{hG^\op} \simeq \cC.
    \]
\end{lemma}
\begin{proof}
    The equivalence $\Fun^\times(\F[BG]^\op,\Cat) \simeq \Cat^{BG^\op}$
    tells us that $\cC$ is uniquely determined by $\cC(G) \in \Cat^{BG^\op}$
    with right $G$-action induced by acting on $G$ from the left on the inside.
    In the following it will be notationally convenient to write $G_1,G_2,\dots$ for copies of $G$.
    Thus, it suffices to build a natural $G_1^\op$-equivariant equivalence
    \begin{equation}\label{eq:shear}
        \cC(G_1 \times G_2)^{hG_2^\op} \simeq \cC(G_1)
    \end{equation}
    We can view the product $G_1 \times G_2$ as a $G_1 \times G_2 \times G^\op$-equivariant set,
    where the right $G$-action is diagonal. Equivalently, if we only act on $G_1$ from the right,
    then we will denote this by $\Ind_{e}^{G_2}G_1 = \coprod_{G_2}G_1$.
    Recall that we have a $G_1 \times G_2 \times G^\op$ equivariant shearing isomorphism
    \[
        \Ind_e^{G_2} G_1 \xto{\cong} G_1 \times G_2, (g_1,g_2) \mapsto (g_1,g_2g_1)
    \]
    This yields natural $G_1^\op \times G_2^\op$-equivariant equivalences
    \[
        \cC(G_1 \times G_2)
        \simeq \cC(\Ind_e^{G_2} G_1)
        \simeq \Coind_e^{G_2^\op}\cC(G_1).
    \]
    Applying $(-)^{hG_2^\op}$ then gives a natural $G_1^\op$-equivariant equivalence (\ref{eq:shear}).
\end{proof}

The last ingredient for the proof of \cref{thm:borel-eqv} is the following
`shifting lemma':

\begin{lemma}[{{\cite[Lemma 2.33]{LLP}}}]\label{lem:shift}
    Let $(\cF,\cN)$ be a span pair where $\cF$ admits finite products.
    Then we have an equivalence natural in $\cN$-normed $\cF$-precategories $\cC$ and $A \in \cF^\op$:
    \[
        \und{\CAlg}_\cF^\cN(\cC)(A) \simeq \CAlg_\cF^\cN(\cC(A \times -))
    \]
    which is compatible with the forgetful functors to $\cC(A) \simeq \lim_{\cF^\op} \cC(A \times -)$.
\end{lemma}

\begin{proof}[Proof of \cref{thm:borel-eqv}.]
    It remains to prove that we have a natural equivalence as depicted in \cref{diag:borel-algebras}.
    Recall that originally this diagram lives in $\Fun^\times(\F[BG]^\op,\Cat)$
    (and we only considered it as one in the equivalent 2-category $\Cat^{BG}$ for notational clarity,
    see~\cref{rem:left-right}).
    We now have natural $G_1^\op$-equivariant equivalences
    \[
        \und{\CAlg}_{BG}(\cC)(G_1)
        \simeq \CAlg_{BG}(\cC(G_1 \times -))
        \simeq \CAlg(\cC(G_1 \times G_2)^{hG_2^\op})
        \simeq \CAlg(\cC(G_1))
    \]
    compatible with the forgetful functors to $\cC(G_1)$,
    where the first uses the above shifting lemma,
    the second is \cref{lem:borel-g-underlying}, and the last is \cref{lem:shear}.
\end{proof}

In the remainder of this section we discuss the case of global Borelification.
Recall from \cref{ex:borel-inc} that we have an extensive Borel inclusion
$(\F,\F) \subseteq (\Fglo,\Forb)$ which induces global Borelifications
$(-)^\flat_\gl \colon \Cat \subseteq \Cat_\gl \coloneqq \Fun^\times(\Fglo^\op,\Cat)$ and
\[
    (-)^\flat_\gl \colon \Mack(\Cat)
    \subseteq \Mack_\gl(\Cat)
    \coloneqq \Fun^\times(\Span_\Forb(\Fglo),\Cat).
\]
We also have Borel inclusions
$(\F[BG],\F[BG]) \subseteq (\F_G,\F_G) \subseteq (\Fglo_{/G},\Forb \times_{\Fglo}\Fglo_{/G})$ which give rise to Borelifications
\[\begin{tikzcd}[ampersand replacement=\&,cramped]
	{\Mack_{BG}(\Cat)} \&\& {\Mack_G(\Cat)} \\
	\& {\Mack_{G\dgl}(\Cat)}
	\arrow["{(-)^\flat_G}", hook', from=1-1, to=1-3]
	\arrow["{(-)^\flat_{G\dgl}}"', hook, from=1-1, to=2-2]
	\arrow["{(-)^\flat_{G\dgl}}", hook', from=1-3, to=2-2]
\end{tikzcd}\]
where $\Mack_{G\dgl}(\Cat) \coloneqq \Fun^\times(\Span_{\Forb}(\Fglo_{/G}),\Cat)$.
Applying \cref{prop:borel-sections} and \cref{thm:borel-eqv}
to this setting yields the following proposition.

\begin{proposition}\label{prop:gl-borel}
    There are equivalences compatible with the forgetful functors:
    \begin{enumerate}
        \item $\und{\CAlg}_\gl(\cC^\flat_\gl) \simeq \CAlg(\cC)^\flat_\gl$
            naturally in $\cC \in \Mack(\Cat)$;
        \item $\und{\CAlg}_{G\dgl}(\cC^\flat_{G\dgl}) \simeq \und{\CAlg}_G(\cC)^\flat_{G\dgl}$ naturally in $\cC \in \Mack_G(\Cat)$;
        \item $\und{\CAlg}_{G\dgl}(\cC^\flat_{G\dgl}) \simeq \CAlg(\cC)^\flat_{G\dgl}$
            naturally in $\cC \in \Mack_{BG}(\Cat)$.
    \end{enumerate}
\end{proposition}

The following proposition tells us how to relate
the monoidal Borelifications of the equivariant, global and $G$-global worlds.

\begin{proposition}\label{prop:borel-global-eq}
    Both of the following squares (and hence the rectangle)
    commute via the Beck-Chevalley transformations:
    \[\begin{tikzcd}[ampersand replacement=\&,cramped]
        {\Mack_{BG}(\Cat)} \& {\Mack_{BG}(\Cat)} \& {\Mack(\Cat)} \\
        {\Mack_G(\Cat)} \& {\Mack_{G\dgl}(\Cat)} \& {\Mack_\gl(\Cat)}
        \arrow[Rightarrow, no head, from=1-1, to=1-2]
        \arrow["{(-)^\flat_G}"', hook, from=1-1, to=2-1]
        \arrow["{(-)^\flat_{G\dgl}}"', hook, from=1-2, to=2-2]
        \arrow["{P^*}"', from=1-3, to=1-2]
        \arrow["{(-)^\flat_{\gl}}", hook', from=1-3, to=2-3]
        \arrow["{L^*}", from=2-2, to=2-1]
        \arrow["{\Pi_G^*}", from=2-3, to=2-2]
    \end{tikzcd}\]
    Here $L$ is induced by the extensive Borel inclusion
    $(\F_G,\F_G) \subseteq (\Fglo_{/G},\Forb \times_{\Fglo} \Fglo_{/G})$
    and $P \colon \Span(\F[BG]) \to \Span(\F)$
    respectively $\Pi_G$ are induced by the functors $BG \to *$
    respectively $\pi_G \colon \Fglo_{/G} \to \Fglo$.
\end{proposition}
\begin{proof}
    Recall from \cref{thm:spans} limits in $\AdTrip$ are computed pointwise in $\Cat$
    and are preserved $\Span \colon \AdTrip \to \Cat$.
    Thus we have cartesian squares
    \begin{equation}\label{eq:span-thing}\begin{tikzcd}[cramped]
        {\Span(\F[BG])} & {\Span(\F[BG])} & {\Span(\F)} \\
        {\Span(\F_G)} & {\Span_{\Forb}(\Fglo_{/G})} & {\Span_\Forb(\Fglo)}
        \arrow[Rightarrow, no head, from=1-1, to=1-2]
        \arrow["K"', hook, from=1-1, to=2-1]
        \arrow["\lrcorner"{anchor=center, pos=0.125}, draw=none, from=1-1, to=2-2]
        \arrow["P", from=1-2, to=1-3]
        \arrow["I"', hook, from=1-2, to=2-2]
        \arrow["\lrcorner"{anchor=center, pos=0.125}, draw=none, from=1-2, to=2-3]
        \arrow["J", hook', from=1-3, to=2-3]
        \arrow["L"', hook, from=2-1, to=2-2]
        \arrow["{\Pi_G}"', from=2-2, to=2-3]
    \end{tikzcd}\end{equation}
    Here $I,J,K$ denote the other extensive Borel inclusions from \cref{ex:borel-inc}.
    Since Beck-Chevalley transformations compose,
    it suffices to show that both squares induce an invertible Beck--Chevalley transformation.
    For the left square, we may equivalently show that the commuting square induced by the right Kan extensions
    is horizontally left adjointable. The Beck--Chevalley map is given by
    $L^*I_* \simeq L^*L_*K_* \xRightarrow{\eps^L K_*} K_*$, which is an equivalence
    since $L$ and hence $L_*$ is fully faithful.
    For the right square, we consider the commutative cube
    \[\begin{tikzcd}[cramped]
        & {\Mack_{BG}(\Cat)} \\
        {\Cat_{BG}} & {\Mack_{G\dgl}(\Cat)} & {\Mack(\Cat)} \\
        {\Cat_{G\dgl}} & \Cat & {\Mack_\gl(\Cat)} \\
        & {\Cat_\gl}
        \arrow[from=1-2, to=2-1]
        \arrow[from=2-2, to=1-2]
        \arrow[from=2-2, to=3-1]
        \arrow[from=2-3, to=1-2]
        \arrow[from=2-3, to=3-2]
        \arrow[from=3-1, to=2-1]
        \arrow[from=3-2, to=2-1]
        \arrow[from=3-3, to=2-2]
        \arrow[from=3-3, to=2-3]
        \arrow[from=3-3, to=4-2]
        \arrow[from=4-2, to=3-1]
        \arrow[from=4-2, to=3-2]
        \arrow[from=4-2, to=3-2]
    \end{tikzcd}\]
    Since the forgetful map $\Mack_{G\dgl}(\Cat) \to \Cat_{G\dgl}$
    is conservative and Beck--Chevalley maps compose,
    it suffices to check that the composite rectangle
    is vertically right adjointable.
    Since the front right face is vertically right adjointable
    by \cref{prop:borel-ran-img}, it remains to see that the front left face
    is vertically right adjointable.
    This is an instance of smooth/proper basechange (see e.g.~\cite[Theorem 6.4.13]{Cisinski}) since $(\Fglo_{/G})^\op \to \Fglo^\op$ is a left fibration
    and $(\Fglo_{/G})^\op \times_{\Fglo^\op} \F \simeq \F[BG]^\op$.
\end{proof}

\begin{remark}\label{rem:borel-glo-eq}
    Given a symmetric monoidal category $\cC$,
    we thus obtain a canonical equivalence
    $(\infl_G \cC)^{\flat}_G \simeq i^*(\cC_\gl^{\flat})$,
    where $i = \Pi_GL \colon \Span(G) \to \Span_{\Forb}(\Fglo)$
    is induced by $\F_G \simeq \Forb_{/G} \to \Forb \subset \Fglo$.
    The latter is how the Borel $G$-symmetric monoidal category
    of a symmetric monoidal category was defined in \cite[3.4.17]{Elmanto-Haugseng}, so our definition coincides with theirs.
\end{remark}

\begin{corollary}\label{cor:calg-comm}
    Let $\cC$ be a symmetric monoidal category. Then the following square commutes
    \[\begin{tikzcd}[cramped]
        {\CAlg_\gl(\cC^\flat_\gl)} & {\CAlg_G((\infl_G\cC)^\flat_G)} \\
        {\CAlg(\cC)} & {\CAlg(\cC^{BG})}
        \arrow["{i^*}", from=1-1, to=1-2]
        \arrow["\simeq"', from=1-1, to=2-1]
        \arrow["{J^*}", draw=none, from=1-1, to=2-1]
        \arrow["\simeq", from=1-2, to=2-2]
        \arrow["{K^*}"', draw=none, from=1-2, to=2-2]
        \arrow["{\infl_G}"', from=2-1, to=2-2]
    \end{tikzcd}\]
    where $i = \Pi_GL$ is as in the above remark and $K,J$ are as in (\ref{eq:span-thing}).
\end{corollary}
\begin{proof}
    The vertical maps are equivalences by \cref{prop:borel-sections},
    and commutativity follows because the maps of span categories
    form the commutative rectangle (\ref{eq:span-thing}).
    In more detail, the above square is identified with the following
    left square, which depicts the effect of the right commutative
    square of 2-categories on hom-categories:
    \[\hspace{-2em}\begin{tikzcd}[cramped,column sep=scriptsize]
        {\Fun_\gl^\tensor(\Env((*)^{\Forb\damalg}), \cC^\flat_\gl)} & {\Fun_G^\tensor(\Env((*)^{G\damalg}), (\infl_G\cC)^\flat_G)} & {\Mack_\gl(\Cat)} & {\Mack_G(\Cat)} \\
        {\Fun^\tensor(\F^\sqcup,\cC)} & {\Fun^\tensor_{BG}(\infl_G\F^{\sqcup}, \infl_G \cC)} & {\Mack(\Cat)} & {\Mack(\Cat)^{BG}}
        \arrow["{i^*}", from=1-1, to=1-2]
        \arrow["\simeq"', from=1-1, to=2-1]
        \arrow["{J^*}", draw=none, from=1-1, to=2-1]
        \arrow["\simeq", from=1-2, to=2-2]
        \arrow["{K^*}"', draw=none, from=1-2, to=2-2]
        \arrow["{i^*}", from=1-3, to=1-4]
        \arrow["{J^*}"', from=1-3, to=2-3]
        \arrow["{K^*}", from=1-4, to=2-4]
        \arrow["{P^*}"', from=2-1, to=2-2]
        \arrow["{P^*}"', from=2-3, to=2-4]
    \end{tikzcd}\]
\end{proof}

\begin{remark}\label{rem:borel-glo-functoriality}
    We can now argue for the claims in \cref{ex:borel-global}
    about the functoriality of the global monoidal Borelification.
    Namely, note that any span in $\Span_\Forb(\Fglo)$
    lies in the image of $\Span(G) \to \Span_{\Forb}(\Fglo)$
    for some large enough finite group $G$.
    The equivalence $(\cC^\tensor)^\flat_\gl|_{\Span(G)} \simeq (\infl_{G}\cC^\tensor)^\flat_{G}$ from \cref{prop:borel-global-eq}
    then allows us to deduce the claims in \cref{ex:borel-global}
    from the description of the functoriality
    of the $G$-equivariant Borelification from \cref{rem:borel-eqv-functoriality},
    applied to the case of a trivial $G$-action.
    As mentioned in \cref{rem:unique-norm} for the equivariant case,
    in view of \cref{prop:unique-norm} it is reasonable to refer to the functors
    $\cC^\flat(H = H \to G) \colon \cC^{BH} \to \cC^{BG}$
    as `the' symmetric monoidal norms.
\end{remark}

\begin{remark}
    Everything in this section which does not involve
    normed algebras, i.e.~\cref{prop:borel-ran-img,prop:bg-normed,prop:borel-global-eq},
    also works for an arbitrary 2-category $\mathbf{C}$ with enough limits in place of $\Cat$, and in fact naturally so,
    e.g.~$(-)^\flat_G \colon \Mack_{BG}(\mathbf{C}) \subseteq \Mack_G(\mathbf{C})$ can be upgraded to a natural transformation of functors $\Cat_2 \to \Cat_2$,
    and $\Mack_{BG}(\mathbf{C}) \simeq \Mack(\mathbf{C})^{BG}$
    can be upgraded to a natural equivalence of 2-categories.

    Moreover, we can even make this natural in $BG$
    using the functoriality of $\Span(\F[BG])$ in $G \in \Glo$,
    which gives a global 2-category $\Mack_{B(-)}(\mathbf{C})$
    that turns out to be Borel by the analogue of \cref{prop:bg-normed},
    i.e.~equivalent to $\Mack(\mathbf{C})^\flat$.

    In particular, using the global functor
    $\CAlg^\flat \colon \Mack(\Cat)^\flat \to \Cat^\flat$,
    the equivalence of \cref{thm:borel-eqv},
    and the above functoriality in $G \in \Glo$,
    we can assemble the functors $\und{\CAlg}_{BG}$ into a global functor
    making the following diagram commute
    \[\begin{tikzcd}[cramped]
        {\Mack_{B(-)}(\Cat)} && {\Mack(\Cat)^\flat} \\
        & {\Cat^\flat}
        \arrow["\simeq", from=1-1, to=1-3]
        \arrow["{\und{\CAlg}_{B(-)}}"', from=1-1, to=2-2]
        \arrow["{\CAlg^\flat}", from=1-3, to=2-2]
    \end{tikzcd}\]
\end{remark}

\section{Construction of Main Examples}\label{sec:ex}

The goal of this section is to construct the main examples
for us, the normed global categories $\und{\Sp}^\tensor, \und{\Sp}_\gl^\tensor$
of equivariant and global spectra, respectively.
Moreover, we build comparison functors between these
as well as between the categories of ultra-commutative (global) ring spectra
to the respective categories of normed algebras.
We approach this by suitably deriving Borel normed global categories
constructed from 1-categorical models.
We thus begin with a quick detour to set up a convenient framework
for deriving normed categories, based on the left derivable
cocartesian fibrations of \cite{Nikolaus-Scholze} (see also \cite{Hinich-Derived}).

\subsection{Dwyer--Kan Localization in Families}\label{subsec:loc}

Denote by $\Cat^\dagger$ the category of marked categories (also called relative categories), i.e.~categories $\cC$ equipped with a wide subcategory $W \subset \cC$ (the `weak equivalences'),
and functors which preserve weak equivalences (also called `homotopical functors').
Recall that Dwyer--Kan localization can be made functorial in relative categories:

\begin{proposition}[{{\cite[4.1.7.2]{HA}}}]\label{prop:functorial-dk}
    The inclusion $\Cat \hookrightarrow \Cat^\dagger, \cC \mapsto (\cC,\cC^\simeq)$
    admits a left adjoint
    \[
        \DK \colon \Cat^\dagger \to \Cat,\ (\cC,W) \mapsto \cC[W^{-1}]
    \]
    which preserves finite products.
    For any marked category $(\cC,W)$, the unit transformation
    $\gamma_\cC \colon (\cC,W) \to (\DK(\cC,W), \DK(\cC,W)^\simeq)$
    exhibits $\DK(\cC,W)$ as Dwyer--Kan localization of $\cC$ at $W$.
\end{proposition}

Note that a left Quillen functor $F\colon\cM \to \cN$ need not be homotopical,
yet may still be left derived to a left adjoint on the underlying
$\infty$-categories $\bfL F \colon \cM_\infty \to \cN_\infty$, see \cite[Prop.~1.5.1]{Hinich}.
In this case, $\bfL F$ is given by an (absolute) right Kan extension of $\gamma_\cN\circ F$ along $\gamma_\cM$ (cf. \cite[Ex.~A.10]{Nikolaus-Scholze}).

More generally, if a functor $\cC \colon S \to\Cat$ assigns a marked category $(\cC(A),W_A)$ to each $A \in S$
and sends morphisms to left derivable functors (see below for precise conditions),
then the cocartesian unstraightening $\int \cC \to S$ is a left derivable cocartesian fibration in the sense of \cite[Appendix A]{Nikolaus-Scholze}; in particular,
Dwyer--Kan localizing at the fiberwise weak equivalences
then produces a cocartesian fibration $(\int \cC)[W_\bullet^{-1}] \to S$ classifying a functor
\[
    S \to\Cat,\quad (f\colon A\to B)
    \mapsto \Bigl(\bfL\cC(f) \colon \cC(A)[W_A^{-1}]\to \cC(B)[W_B^{-1}]\Bigr).
\]
Below, we introduce a variant of these left derivable cocartesian fibrations
which are especially easy to construct from model-categorical examples;
it requires one to provide a lot more data, namely a left deformation
on each category $\cC(A)$, but the conditions to be left derivable become trivial to verify.

\begin{definition}
    A \emph{left deformation} on a marked category $\cC$ is an
    endofunctor $Q \colon \cC \to \cC$ together with a pointwise marked
    natural transformation $q \colon Q \Rightarrow \id_\cC$.
    If $\cD$ is another such category with left deformation
    $(P,p)$, then a functor $F \colon \cC \to \cD$ is
    \emph{compatible with the left deformations}
    if it restricts to a marked functor $\Img(Q) \to \Img(P)$
    between the essential images of the left deformations.
\end{definition}

\begin{example}[{{\cite[Example A.10]{Nikolaus-Scholze}}}]\label{lem:deform-deriv}
    Let $F \colon (\cC,Q,q) \to (\cD,P,p)$ be a functor compatible with the deformations.
    Then the transformation
    $
        \DK(FQ)\gamma_\cC \simeq \gamma_\cD FQ \xRightarrow{\gamma_\cD Fq} \gamma_\cD F
    $
    exhibits $\DK(FQ) \colon \cC[W_\cC^{-1}] \to \cD[W_\cD^{-1}]$\footnote{The weak equivalences in $\cD$ may not satisfy 2-out-of-3, so $FQ$ might not be homotopical, however equivalences do, so $\gamma_\cD FQ$ uniquely factors through $\gamma_\cC$,
        and under slight abuse of notation we denote the unique induced functor
        on localizations by $\DK(FQ)$.}
    as the left derivation $\bfL F$ of $F$,
    i.e.~as an absolute right Kan extension of $\gamma_\cD F$ along $\gamma_\cC$.
\end{example}

\begin{definition}
    Let $p \colon \cX \to S$ be a cocartesian fibration
    and suppose that each fiber $\cX_A$ is equipped
    with a marking $W_A$ and a left deformation $(Q_A,q_A)$.
    We say $p$ is \emph{left deformable} if for every morphism $f \colon A \to B$ in $S$,
    the pushforward $f_! \colon \cX_A \to \cX_B$ is compatible with the left deformations.
    We equip $\cX$ with the fiberwise weak equivalences $(\cX,W_\bullet)$.
\end{definition}

Note that if $K \to S$ is any functor then $\cX \times_{S} K \to K$
canonically inherits the structure of a left deformable cocartesian fibration.

\begin{theorem}\label{thm:ldeform-cfib}
    Let $p \colon \cX \to S$ be a left deformable cocartesian fibration.
    \begin{enumerate}
        \item By the universal property of Dwyer--Kan localization,
            we obtain a commuting triangle
            \[\begin{tikzcd}[cramped]
                \cX && {\cX[W_\bullet^{-1}]} \\
                & S
                \arrow["{\gamma_\cX}", from=1-1, to=1-3]
                \arrow["p"', from=1-1, to=2-2]
                \arrow["q", from=1-3, to=2-2]
            \end{tikzcd}\]
            Then $q$ is also a cocartesian fibration
            and $\gamma_\cX$ preserves cocartesian lifts
            of morphisms $f \colon A \to B$ whose pushforward
            $f_! \colon \cX_A \to \cX_B$ is homotopical.

        \item Specifically, for any such functor,
            the morphism $\cX \times_S K \to \cX[W_\bullet^{-1}]\times_S K$
            exhibits the latter as Dwyer--Kan localization
            of $\cX \times_S K$ at the edges mapped into $W_\bullet$.
            In particular, we can identify
            $(\gamma_\cX)_A \colon \cX_A \to \cX[W_\bullet^{-1}]_A$ with
            the Dwyer--Kan localization $\gamma_A \coloneqq \gamma_{\cX_A} \colon \cX_A \to \cX_A[W_A^{-1}]$.

        \item For every map $f \colon A \to B$ in $S$, the lax square
induced by the above triangle
            \[\begin{tikzcd}[cramped]
                {\cX_A} & {\cX_B} \\
                {\cX[W_\bullet^{-1}]_A} & {\cX[W_\bullet^{-1}]_B}
                \arrow["{f_!^p}", from=1-1, to=1-2]
                \arrow["{(\gamma_\cX)_A}"', from=1-1, to=2-1]
                \arrow["{(\gamma_\cX)_B}", from=1-2, to=2-2]
                \arrow[Rightarrow, from=2-1, to=1-2]
                \arrow["{f_!^q}"', from=2-1, to=2-2]
            \end{tikzcd}\]
            exhibits $f^q_!$ as the left derived functor
            $\bfL f_!^p = \DK(f_!^pQ_A)$ of $f_!^p$,
            i.e.~as the absolute right Kan extension of
            $\gamma_B f_!^p$ along $\gamma_A$.
            The natural transformation is invertible if $f^p_!$ is homotopical.
    \end{enumerate}
\end{theorem}
\begin{proof}
    By \cite[Theorem A.14, A.15]{Nikolaus-Scholze},
    it suffices to verify that a left deformable cocartesian
    fibration is left derivable in the sense of
    \cite[Definition A.8]{Nikolaus-Scholze},
    for which it suffices to consider $S = \Delta^2$,
    with $p$ classified by a diagram
    \begin{equation}\label{diag:ldeform-cfib}\begin{tikzcd}
        & \cN \\
        \cM && \cO
        \arrow["F", from=2-1, to=1-2]
        \arrow["G", from=1-2, to=2-3]
        \arrow["H"', from=2-1, to=2-3]
    \end{tikzcd}\end{equation}
    Let $(P,p)$ and $(Q,q)$ be the given left deformations on $\cM$ and $\cN$.
    By assumption and \cref{lem:deform-deriv} $F,G,H$ are left derivable
    with $\bfL F \simeq \DK(FP),\ \bfL G\simeq \DK(GQ)$ and $\bfL H \simeq \DK(HP)$.
    It remains to check compatibility of these left derivations
    with respect to composition,
    but the canonical morphism $\bfL G \circ \bfL F \Rightarrow \bfL H$
    is identified with the following equivalence:
    \[
        \DK(GQ) \circ \DK(FP)
        \simeq \DK(GQFP)
        \xRightarrow[\simeq]{\DK(GqFP)} \DK(GFP)
        \xRightarrow{\simeq} \DK(HP).
    \]
    This uses that $FP$ has image in $\Img(Q)$ and thus $GqFP$ is pointwise marked.
\end{proof}

The last sentence of the above proof
is one reason we consider left deformable cocartesian fibrations
instead of the more general left derivable ones of \cite{Nikolaus-Scholze}.
In general, both homotopical or left Quillen functors could occur
for either $F$ or $G$ in (\ref{diag:ldeform-cfib}).
While it is clear that the composite $GF$ is left derivable when both
$F$ and $G$ are homotopical, both are left Quillen, or $F$ is left Quillen
and $G$ is homotopical, it is not clear to the author whether
the remaining case is even true in full generality;
the above proof really rests on the assumption that $F$ would
also preserve cofibrant objects, i.e.~restrict to the images
of the left deformations.
The argument given in \cite[Example A.13]{Nikolaus-Scholze} misses this,
although it works in their specific example of a symmetric monoidal model category
since there also the cocartesian pushforwards along the inert maps preserve cofibrant objects.

We are now interested in conditions under which a homotopical map between two left deformable cocartesian fibrations
localizes to a map which preserves cocartesian edges. To this end, we will need the following general lemma.

\begin{lemma}\label{lem:cocart-morphisms}
    Let $p_i \colon \cX^i \to S,\ i=0,1$ be two cocartesian fibrations and suppose we have a functor
    $F \colon \cX^0 \to \cX^1$ over $S$.
    Under the unstraightening equivalence $\Fun(\Delta^1,\Cat_{/S}) \simeq \Cocart(\Delta^1)_{/S \times \Delta^1}$,
    we can view this as a morphism of cocartesian fibrations
    \[\begin{tikzcd}[cramped]
        {\int F} && {\Delta^1 \times S} \\
        & {\Delta^1}
        \arrow["p", from=1-1, to=1-3]
        \arrow["q"', from=1-1, to=2-2]
        \arrow[from=1-3, to=2-2]
    \end{tikzcd}\]
    Then $F$ preserves cocartesian edges if and only if $p$ is a cocartesian fibration.
\end{lemma}
\begin{proof}
    It follows from \cite[2.4.2.11]{HTT} that $p$ is a locally cocartesian fibration,
    and an edge $x \to z$ in $\int F$ is locally $p$-cocartesian if and only if
    it factors as $x \to y \to z$ where $x \to y$ is $q$-cocartesian
    and $y \to z$ is $p_{q(y)}$-cocartesian.
    Recall that $p$ is a cocartesian fibration if and only if these locally $p$-cocartesian edges compose.
    By the description just given,
    the only non-trivial case left to check is that the composite of a $p_{0}$-cocartesian map $x \to x'$
    followed by a $q$-cocartesian map $x' \to Fx'$ factors as a $q$-cocartesian map
    followed by a $p_1$-cocartesian one.
    We have the following left commutative square in $\int F$ lying over the right one in $S \times \Delta^1$:
    \[\begin{tikzcd}[cramped]
        x & {Fx} && {(s,0)} & {(s,1)} \\
        {x'} & {Fx'} && {(s',0)} & {(s',1)}
        \arrow[from=1-1, to=1-2]
        \arrow[from=1-1, to=2-1]
        \arrow[from=1-2, to=2-2]
        \arrow[from=1-4, to=1-5]
        \arrow[from=1-4, to=2-4]
        \arrow[from=1-5, to=2-5]
        \arrow[from=2-1, to=2-2]
        \arrow[from=2-4, to=2-5]
    \end{tikzcd}\]
    The map $x \to Fx'$ is $q$-cocartesian hence locally $p$-cocartesian.
    Now if $F$ preserves cocartesian edges, then $Fx \to Fx'$ is $p_1$-cocartesian,
    and we have thus verified that locally $p$-cocartesian edges compose, so that $p$ is a cocartesian fibration.
    Conversely, if $p$ is a cocartesian fibration, then by right-cancellability of cocartesian edges
    the map $Fx \to Fx'$ is $p$-cocartesian, hence $p_1$-cocartesian, proving that $F$ preserves cocartesian edges.
\end{proof}

\begin{proposition}\label{prop:deriv-cocart-mor}
    Let $p \colon \cX \to S$ and $q \colon \cY \to S$ be left deformable cocartesian fibrations and $f \colon A \to B$ a morphism in $S$.
    Consider a homotopical functor $\alpha \colon \cX \to \cY$ which preserves cocartesian lifts of $f$,
    and suppose that one of the following two conditions hold:
    \begin{enumerate}
        \item The pushforwards $f_! \colon \cX_A \to \cX_B$ and $f_! \colon \cY_A \to \cY_B$ are homotopical.
        \item $\alpha|_{\cX_A} \colon \cX_A \to \cY_A$ and $\alpha|_{\cX_B} \colon \cX_B \to \cY_B$
            are compatible with the deformations.
    \end{enumerate}
    Then $\DK(\alpha)$ preserves cocartesian lifts of $f$.
    In particular, if $\alpha$ preserves all cocartesian edges and is fiberwise compatible with the deformations,
    then $\DK(\alpha)$ is a map of cocartesian fibrations, i.e.~preserves all cocartesian edges.
\end{proposition}
\begin{proof}
    Since $\alpha$ is homotopical, we have a commutative diagram
    \[\begin{tikzcd}[cramped,sep=scriptsize]
        \cX && \cY \\
        & S \\
        {\DK(\cX)} && {\DK(Y)}
        \arrow["\alpha", from=1-1, to=1-3]
        \arrow["p"{description}, from=1-1, to=2-2]
        \arrow["{\gamma_\cX}"', from=1-1, to=3-1]
        \arrow["q"{description}, from=1-3, to=2-2]
        \arrow["{\gamma_\cY}", from=1-3, to=3-3]
        \arrow["{p'}"{description}, from=3-1, to=2-2]
        \arrow["{\DK(\alpha)}"', from=3-1, to=3-3]
        \arrow["{q'}"{description}, from=3-3, to=2-2]
    \end{tikzcd}\]
    Under the first assumption, we get from \cref{thm:ldeform-cfib}(1) that
    $\gamma_\cX$ and $\gamma_\cY$ preserve cocartesian lifts of $f$.
    Since they are also essentially surjective, uniqueness of cocartesian lifts with a given source
    implies that every $p'$-cocartesian lift of $f$
    is the image under $\gamma_\cX$ of a $p$-cocartesian lift of $f$. The claim follows by commutativity of the diagram.

    To derive the same conclusion using the second possible assumption, note that we may reduce
    to the case where $\alpha$ preserves all cocartesian edges and is fiberwise compatible with the
    left deformations by pulling back along $f \colon \Delta^1 \to \cC$ and using \cref{thm:ldeform-cfib}(2).
    We then have a cocartesian fibration $\int \alpha \to \Delta^1 \times S$ by \cref{lem:cocart-morphisms}.
    Equip $\int \alpha$ with the fiberwise weak equivalences and left deformations of $\cX$ respectively $\cY$.
    Because $p,q$ are left deformable and $\alpha$ is fiberwise compatible with the left deformations,
    also $\int \alpha \to \Delta^1 \times S$ is left deformable, and we can take the left derivation
    \[\begin{tikzcd}[cramped,sep=scriptsize]
        {\int \alpha} && {\DK(\int \alpha)} \\
        & {\Delta^1 \times S}
        \arrow[from=1-1, to=1-3]
        \arrow[from=1-1, to=2-2]
        \arrow["r", from=1-3, to=2-2]
    \end{tikzcd}\]
    By \cref{thm:ldeform-cfib} we get that $r$ is a cocartesian fibration.
    Since $\alpha$ is homotopical, we have $\DK(\int \alpha) \simeq \int \DK(\alpha)$
    as cocartesian fibrations over $\Delta^1$.
    But then another application of \cref{lem:cocart-morphisms} shows that $\DK(\alpha)$ preserves all cocartesian edges.
\end{proof}

Consider now the case of a functor $\cC \colon S \to \Cat^\dagger$.
Then the above constructions simply recover the functorial Dwyer--Kan localization:

\begin{corollary}\label{cor:compare-dk}
    The cocartesian unstraightening
    of $\DK \circ \cC$ agrees with
    the left derivation of the canonical left deformable
    fibration $\int \cC$ where all left deformations are taken to be the identity.
\end{corollary}
\begin{proof}
    The unit of the adjunction $\DK \colon \Cat^\dagger \rightleftarrows \Cat$
    induces a map of cocartesian fibrations $\alpha \colon \int \cC \to \int \DK \circ \cC$ over $S$,
    which clearly inverts the fiberwise weak equivalences,
    and thus uniquely factors through
    $\DK(\alpha) \colon \DK(\int \cC) \to \int \DK \circ \cC$ over $S$.
    Since $\DK(\alpha)$ is clearly a fiberwise equivalence,
    it remains to see that it preserves cocartesian edges,
    which follows from the above proposition.
\end{proof}

\begin{corollary}\label{cor:subfib}
    Let $p \colon \cX \to S$ be a left deformable cocartesian fibration
    and $i \colon \cX^0 \subseteq \cX$ be a fully faithful map of cocartesian fibrations over $S$
    with $\Img(Q_A) \subseteq \cX^0_A$ for all $A \in S$.
    Then we have an equivalence $\cX^0[W_\bullet^{-1}] \xrightarrow{\simeq} \cX[W_\bullet^{-1}]$ of cocartesian fibrations over $S$.
    Note that we can take $\cX^0$ to be the full subcategory generated by the images of the left deformations.
\end{corollary}
\begin{proof}
    Any such $\cX^0$ clearly inherits the structure of a left deformable cocartesian fibration from $\cX$.
    In particular, this makes $i$ automatically homotopical and fiberwise compatible with the left deformations,
    so that by \cref{prop:deriv-cocart-mor}(2) the induced map $\DK(i) \colon \DK(\cX^0) \to \DK(\cX)$
    is a morphism of cocartesian fibrations over $S$.
    It is therefore enough to check that $\DK(i)$ is a fiberwise equivalence.
    But by \cref{thm:ldeform-cfib}(2) we know that over $A \in S$
    this is the horizontal map in the following commutative triangle:
    \[\begin{tikzcd}[cramped]
        & {\Img(Q_A)[W_{A}^{-1}]} \\
        {\cX^0_A[W_A^{-1}]} && {\cX_A[W_A^{-1}]}
        \arrow["\simeq"', from=1-2, to=2-1]
        \arrow["\simeq", from=1-2, to=2-3]
        \arrow["{i_A[W_{A}^{-1}]}"', from=2-1, to=2-3]
    \end{tikzcd}\]
    Here the diagonal maps are equivalences since the left deformation $Q_A$ induces an inverse for both,
    and thus $\DK(i)$ is an equivalence.
    For the last point, it remains to verify that if we define $\cX^0 \subseteq \cX$ to be the full subcategory
    spanned by the images of the left deformations, then it inherits cocartesian lifts from $\cX$.
    This follows from the fact that the cocartesian pushforwards $\cX_A \to \cX_B$ restrict to $\Img(Q_A) \to \Img(Q_B)$.
\end{proof}

We will need a version of the above corollary for maps of left deformable cocartesian fibrations that may not be homotopical.

Let $p \colon \cX \to S$ be a left deformable cocartesian fibration
and $\cX^0 \subseteq \cX$ be the full subcategory on the essential images of the left deformations,
so that $p^0 \colon \cX^0 \to S$ is again left deformable by the above.
Let $q \colon \cY \to S$ and $q^0 \colon \cY^0 \to S$ be another such pair.
Suppose now that we have a map of cocartesian fibrations $\alpha \colon \cX \to \cY$
which is fiberwise compatible with left the deformations,
hence restricts to a homotopical functor $\alpha^0$ as in the following square of cocartesian fibrations over $S$:
\[\begin{tikzcd}[cramped]
    {\cX^0} && \cX \\
    & S \\
    {\cY^0} && \cY
    \arrow[hook', from=1-1, to=1-3]
    \arrow["{p^0}"{description}, from=1-1, to=2-2]
    \arrow["{\alpha^0}"', from=1-1, to=3-1]
    \arrow["p"{description}, from=1-3, to=2-2]
    \arrow["\alpha", from=1-3, to=3-3]
    \arrow["{q^0}"{description}, from=3-1, to=2-2]
    \arrow[hook, from=3-1, to=3-3]
    \arrow["q"{description}, from=3-3, to=2-2]
\end{tikzcd}\]
Note that by unstraightening, we can think of the above square as a fully faithful
inclusion $\int \alpha^0 \subseteq \int \alpha$ over $\Delta^1 \times S$,
where both source and target are cocartesian fibrations by \cref{lem:cocart-morphisms}.
Equipping $\int \alpha \to \Delta^1 \times S$ with the fiberwise weak equivalences and left deformations
of $\cX$ respectively $\cY$, we then get that $\int \alpha \to \Delta^1 \times S$ is a left deformable
fibration, and $\int \alpha_0 \subseteq \int \alpha$ is spanned by the images of the left deformations.
In particular, we find us in the situation of the previous corollary, and may conclude that
also $\int \alpha_0 \to \Delta^1 \times S$ is a left deformable cocartesian fibration,
and that the inclusion induces an equivalence $\DK(\int \alpha_0) \simeq \DK(\int \alpha)$
of cocartesian fibrations over $\Delta^1 \times S$.
Since $\alpha_0$ is homotopical, we have $\DK(\int \alpha_0) \simeq \int \DK(\alpha_0)$
as cocartesian fibrations over $\Delta^1$, and hence $\DK(\alpha_0)$ preserves cocartesian edges
by another application of \cref{lem:cocart-morphisms}.
Similarly, the lax map of cocartesian fibrations $\int \alpha \to \DK(\int \alpha)$ over $\Delta^1$
classifies the left derived functor of $\alpha$
(essentially by definition, c.f.~\cite[Definition 2.3.3, Corollary 4.2]{Hinich-Derived}).
Overall, this proves the following proposition:

\begin{proposition}\label{prop:idk}
    In the above situation, the left derived functor of $\alpha$ exists and agrees with $\DK(\alpha^0)$,
    in that we have a coherent cube where all faces
    except the right one commute, and the right exhibits $\bfL \alpha$ as the left derived functor of $\alpha$:
    \[\begin{tikzcd}[cramped,row sep=scriptsize]
        & {\cX_0} && \cX \\
        {\cY_0} && \cY \\
        & {\DK(\cX_0)} && {\DK(\cX)} \\
        {\DK(\cY_0)} && {\DK(\cY)}
        \arrow[hook, from=1-2, to=1-4]
        \arrow["{\alpha_0}"{description}, from=1-2, to=2-1]
        \arrow[from=1-2, to=3-2]
        \arrow["\alpha"{description}, from=1-4, to=2-3]
        \arrow[from=1-4, to=3-4]
        \arrow[hook, from=2-1, to=2-3]
        \arrow[from=2-1, to=4-1]
        \arrow[from=2-3, to=4-3]
        \arrow["\simeq"{description, pos=0.2}, from=3-2, to=3-4]
        \arrow["{\DK(\alpha_0)}"', pos=0.3, from=3-2, to=4-1]
        \arrow[Rightarrow, from=3-4, to=2-3]
        \arrow["{\bfL\alpha}"{description}, from=3-4, to=4-3]
        \arrow["\simeq"{description, pos=0.3}, from=4-1, to=4-3]
    \end{tikzcd}\]
    Moreover, all morphisms except possibly the vertical ones are morphisms of cocartesian fibrations over $S$.
\end{proposition}

\subsection{Models for $G$-equivariant and $G$-global Spectra}\label{subsec:ex}

Using the tools of the previous subsection,
we can now construct the normed global categories of equivariant and global spectra
$\und{\Sp}^\tensor$ and $\und{\Sp}_\gl^\tensor$.
Moreover, we build comparison functors which allow us to view ultra-commutative (global) ring spectra
as certain normed algebras in the above global categories.

We will follow the articles \cite{CLL-Global,CLL-Equivariant}
and use the closed symmetric monoidal 1-category $\Sp^{\Sigma,\tensor}$ of symmetric spectra (based on simplicial sets) as a base model for the 1-category of spectra.\footnote{Originally, ultra-commutative global ring
spectra were defined as strictly commutative
algebras in the category of \emph{orthogonal} spectra, considered up
to underlying global weak equivalence tested on all compact Lie groups, c.f.~\cite{Global}.
Restricting our attention to finite groups, this agrees
with strictly commutative algebras in symmetric spectra considered
up to the corresponding notion of global weak equivalence.
The details of such a comparison were spelled out in \cite[Appendix A]{LLP}.}
The category of $G$-equivariant objects $G\Sp^{\Sigma,\tensor} \coloneqq \Fun(BG,\Sp^{\Sigma,\tensor})$
equipped with the pointwise symmetric monoidal structure
is also known as the category of $G$-symmetric spectra \cite{Hausmann-GSym},
and supports various model structures modelling $G$-equivariant and $G$-global homotopy theory \cite{Lenz-phd}.
We begin by summarizing the relevant model-categorical results.

\begin{theorem}\label{thm:model}
    For a finite group $G$, the category $G\Sp^\Sigma$
    admits a flat $G$-global as well as flat and projective $G$-equivariant (combinatorial, stable)
    model structures with the following properties:
    \begin{enumerate}
        \item Both equivariant model structures have as weak equivalences
            the $G$-stable weak equivalences of \cite[Definition 2.35]{Hausmann-GSym}.
            The $G$-global equivalences are those maps $f \colon X \to Y$ in $G\Sp^\Sigma$
            such that $\alpha^*f$ is a $K$-stable equivalence
            for every map $\alpha = (K \to G) \in \Glo_{/G}$.
            The equivariant model structures both model the (symmetric monoidal) ($\infty$-)category
            of genuine $G$-spectra $\Sp_G$,
            and the global one models the ($\infty$-)category of $G$-global spectra
            $\Sp_{G\dgl}$ by definition.

        \item The cofibrations in all 3 model structures are independent of the group action,
            i.e.~defined on underlying symmetric spectra,
            so the flat equivariant and global cofibrations agree.
            Every projective cofibration is flat.

        \item Restricting along general group homomorphism
            preserves flat/projective cofibrations,
            but is only left Quillen for the $G$-equivariant projective or $G$-global flat model structure.
            Restricting along injections is homotopical and hence left Quillen for the global
            and both equivariant model structures.

        \item All model structures are monoidal with cofibrant unit $\S_G$,
            so the tensor product restricts to a homotopical functor
            $\tensor \colon G\Sp^{\Sigma}_{\cflat} \times G\Sp^{\Sigma}_{\cflat} \to G\Sp^{\Sigma}_{\cflat}$
            for both the $G$-equivariant and $G$-global weak equivalences,
            and analogously in the $G$-equivariant projective case.

        \item For $H \leq G$ the symmetric monoidal norm restricts to a homotopical functor
            $N_H^G \colon H\Sp^{\Sigma}_{\cflat} \to G\Sp^{\Sigma}_\cflat$
            for both the $G$-equivariant and $G$-global weak equivalences,
            and likewise in the $G$-equivariant projective case.
    \end{enumerate}
\end{theorem}
\begin{proof}
    Existence of the equivariant respectively global model structures
    with the mentioned weak equivalences and cofibrations is shown in \cite[Theorem 4.7,4.8]{Hausmann-GSym}
    respectively \cite[Theorem 3.1.40]{Lenz-phd}.
    The statements about the monoidal structure and norm are shown in \cite[Section 6]{Hausmann-GSym}
    in the equivariant case, and in \cite[Proposition 3.1.62]{Lenz-phd},\cite[Corollary 5.19]{Lenz-Stahlhauer}.
    It follows from \cite[Remark 2.20]{Hausmann-GSym} that both the flat and projective cofibrations
    do not depend on the $G$-action.
    Every projective cofibration is also a flat cofibration
    by definition, cf.~\cite[2.18]{Hausmann-GSym}.
    That restriction along injections is compatible
    with equivariant weak equivalences and cofibrations
    also follows directly from the definitions, compare \cite[Section 5.2]{Hausmann-GSym}.
    Since both types of cofibrations do not depend on the group,
    restriction along arbitrary group homomorphisms preserves them.
    That such restrictions are left Quillen
    for the projective model structure follows directly from this and
    \cite[Section 5.1]{Hausmann-GSym}, compare \cite[Lemma 9.3]{CLL-Equivariant}.
    It is clear from the definition of $G$-global equivalences that arbitrary restrictions are homotopical for them.
\end{proof}

In view of \cref{ex:borel-global} the monoidal global Borelification
\[
    (\Sp^{\Sigma,\tensor})^\flat : \Span_{\Forb}(\Fglo) \to \CAT
\]
has the following functoriality
\begin{itemize}
    \item A backwards map $K \xleftarrow{\alpha} G = G$ is sent to the restriction
        $\alpha^* : K\Sp^{\Sigma} \to G\Sp^{\Sigma}$.
    \item Restriction along $\Span(\F) \to \Span_{\Forb}(\Fglo)$ induced by $* \mapsto G$
        classifies the usual symmetric monoidal category $G\Sp^{\Sigma,\tensor}$ of $G$-symmetric spectra.
    \item A forwards inclusion $H = H \hookrightarrow G$ is sent to the multiplicative
        Hill--Hopkins--Ravenel norm $N_H^G : H\Sp^{\Sigma} \to G\Sp^{\Sigma}$.
\end{itemize}
Let $p : \int (\Sp^{\Sigma,\tensor})^\flat \to \Span_{\Forb}(\Fglo)$
denote the cocartesian unstraightening.

\begin{construction}[Construction of $\und{\Sp}^\tensor$]\label{con:sp-eq}
    We equip each fiber $p^{-1}(\coprod_{i=1}^n G_i) = \prod_{i=1}^n G_i\Sp^\Sigma$
    with weak equivalences $\prod_{i=1}^n W_{G_i\text{-stable}}$
    and left deformations induced by the functorial cofibrant replacements of the projective $G_i$-equivariant
    model structures.
    By \cref{thm:model} this yields a left deformable cocartesian fibration $p_\eqv$.
    In view of \cref{cor:subfib},
    we can take fiberwise full subcategories on the flat and projectively
    cofibrant spectra and obtain obtain a commutative diagram
    of cocartesian fibrations over $\Span_{\Forb}(\Fglo)$
    where the labels denote which cocartesian arrows are preserved ($\cc$ means all):
    \[\begin{tikzcd}[cramped]
        {\int (\Sp^{\Sigma,\tensor}_{\projcof})^\flat} & {\int(\Sp^{\Sigma,\tensor}_\cflat)^\flat} & {\int (\Sp^{\Sigma,\tensor})^\flat} \\
        & {\int\und{\Sp}^\tensor}
        \arrow["{\cc}", hook, from=1-1, to=1-2]
        \arrow["{\cc}"', from=1-1, to=2-2]
        \arrow["{\cc}", hook, from=1-2, to=1-3]
        \arrow["{\Span(\Forb)\dcc}"{description}, from=1-2, to=2-2]
        \arrow["{\Forb^\op\dcc}", from=1-3, to=2-2]
    \end{tikzcd}\]
    Here all downwards arrows exhibit $\int \und{\Sp}^\tensor$
    as the left derived cocartesian fibration of their source,
    i.e.~as the Dwyer-Kan localization of its source at the
    $W_{\bullet\text{-stable}}$ equivalences.
    In particular, by \cref{cor:compare-dk}, we can equivalently
    define $\und{\Sp}^\tensor$ as the functorial Dwyer--Kan localization
    of $(\Sp^{\Sigma,\tensor}_{\projcof})^\flat$.
    The horizontal arrows clearly preserve all cocartesian morphisms,
    and by \cref{thm:ldeform-cfib}
    the vertical ones preserve cocartesian lifts of those
    morphisms which induce homotopical pushforward functors.
    For example, it follows from \cref{thm:model}
    that when restricting to flat spectra,
    all relevant functors except inflations are homotopical,
    and hence the middle vertical localization preserves
    cocartesian lifts of morphisms in the subcategory
    $\Span(\Forb) \subset \Span_{\Forb}(\Fglo)$.
\end{construction}

\begin{construction}[Construction of $\und{\Sp}_\gl^\tensor$]\label{con:sp-gl}
    Analogously to the above, by \cref{thm:model}
    we obtain a left deformable cocartesian fibration $p_\gl$
    by equipping each fiber of $p$ with the $G$-global equivalences
    and left deformations induced by the functorial cofibrant
    replacement in the flat $G$-global model structure.
    We can also restrict to flat
    spectra in the global case and obtain a commutative diagram of cocartesian fibrations
    over $\Span_{\Forb}(\Fglo)$
    \[\begin{tikzcd}[cramped]
        {\int (\Sp^{\Sigma,\tensor}_\cflat)^\flat} && {\int (\Sp^{\Sigma,\tensor})^\flat} \\
        & {\int\und{\Sp}_\gl^\tensor}
        \arrow["{\cc}", hook, from=1-1, to=1-3]
        \arrow["{\cc}"', from=1-1, to=2-2]
        \arrow["{\Fglo^\op\dcc}", from=1-3, to=2-2]
    \end{tikzcd}\]
    with the downwards arrows exhibiting $\int \und{\Sp}_\gl^\tensor$
    as the left derived cocartesian fibration of their source,
    i.e.~as Dwyer-Kan localization at the $W_{\bullet\text{-global}}$ equivalences.
    In particular, by \cref{cor:compare-dk}, we can equivalently
    define $\und{\Sp}_\gl^\tensor$ as the functorial Dwyer--Kan localization
    of $(\Sp^{\Sigma,\tensor}_{\cflat})^\flat$.
\end{construction}

\begin{warning}
    The localization of $G\text{-}\Sp^{\Sigma}_{\projcof}$ at the $G$-\emph{global}
    weak equivalences does \emph{not} give the category of $G$-global spectra.
    Namely, by \cref{prop:lquil} any $G$-equivariant equivalence between
    projectively cofibrant symmetric spectra is already a $G$-global equivalence,
    and thus the localization actually agrees with $\Sp_G$.
    The functor $\Sp_G \to \Sp_{G\dgl}$ induced by inverting
    $G$-global equivalences for $G\Sp^{\Sigma}_{\projcof} \to G\Sp^{\Sigma}_{\cflat}$
    is the fully faithful left adjoint $\Sp_G \subseteq \Sp_{G\dgl}$,
    whose image consists of those $G$-global spectra whose underlying
    global spectrum is `left induced' i.e.~in the image of the left adjoint
    $\Sp \subseteq \Sp_\gl$, i.e.~has constant geometric fixed points.
\end{warning}

\begin{proposition}\label{prop:sp-gl-univ-prop}
    Let $\und{\Sp}^\tensor, \und{\Sp}_\gl^\tensor$ be as above.
    \begin{enumerate}
        \item The restriction $\und{\Sp} \coloneqq \und{\Sp}^\tensor|_{\Glo^\op}$
            is the global category of equivariant spectra
            constructed in \cite{CLL-Equivariant}.
            Thus $\und{\Sp}$ is the free equivariantly
            presentable equivariantly stable global category on one generator.

        \item Analogously, $\und{\Sp}_\gl \coloneqq \und{\Sp}_\gl^\tensor|_{\Glo^\op}$ is the global
            category of global spectra from \cite{CLL-Global},
            and hence the free globally presentable equivariantly stable global category
            on one generator.

        \item For every injection $p \colon H \hookrightarrow G$,
            the derived multiplicative norms
            $\Sp_H \to \Sp_G$ of $\und{\Sp}^\tensor$
            and $\Sp_{H\dgl} \to \Sp_{G\dgl}$
            of $\und{\Sp}_\gl^\tensor$
            preserve sifted colimits.
    \end{enumerate}
\end{proposition}
\begin{proof}
    For point (1), note that we have already remarked in \cref{con:sp-eq} that
    \[
        \und{\Sp}^\tensor \simeq \DK \circ\ ((\Sp^{\Sigma,\tensor}_{\projcof})^\flat, W_{\bullet\text{-stable}}).
    \]
    Restricting along $\Glo^\op \to \Span_{\Forb}(\Fglo)$, this yields
    precisely the definition of $\und{\mathscr{S}p}$ used in \cite[Section 9.1]{CLL-Equivariant},
    and so the claims follow from Theorems 9.4 and 9.5 of op.~cit.

    For point (2), we consider the morphism
    $\gamma_\gl \colon \int (\Sp^{\Sigma,\tensor})^\flat \to \int \und{\Sp}_\gl^\tensor$
    of cocartesian fibrations over $\Span_{\Forb}(\Fglo)$ from \cref{con:sp-gl},
    which exhibits its target $\int \und{\Sp}_\gl^\tensor \to \Span_{\Forb}(\Fglo)$
    as the left derivation of its source.
    By \cref{thm:ldeform-cfib}(2) this is stable under pullback along
    $\Glo^\op \to \Span_{\Forb}(\Fglo)$,
    where the domain is now (the straightening of) a homotopical functor,
    so \cref{cor:compare-dk} we get an equivalence of global categories
    \[
        \und{\Sp}_\gl \simeq \DK \circ ((\Sp^\Sigma)^\flat, W_{\bullet\text{-global}}).
    \]
    This is precisely the definition of $\und{\mathscr{S}p^\gl}$ from \cite[Section 7.1]{CLL-Global},
    and so the claims follow from Theorem 7.3.2 and Corollary 7.3.3 of op.~cit.

    For the equivariant case of (3), note
    that by \cref{thm:model} the norm $H\Sp_{\cflat}^{\Sigma} \to G\Sp_{\cflat}^\Sigma$ is homotopical and yields the norm $\Sp_H \to \Sp_G$ on localizations.
    One then checks that geometric realizations
    and filtered colimits in $\Sp^\Sigma$ are homotopical and preserve
    flat symmetric spectra, and thus reduces to showing
    that the pointset level norm $H\Sp^\Sigma \to G\Sp^\Sigma$
    commutes with filtered colimits and geometric realizations
    up to isomorphism.
    For the details, we refer the reader to \cite[Proposition 5.5.7]{Tambara}.

    The global case is similar;
    by \cref{thm:model} the norm $H\Sp_{\cflat}^\Sigma \to G\Sp_{\cflat}^\Sigma$ is homotopical for the $H$- respectively $G$-global weak equivalences, and one checks that geometric realizations and filtered colimits
    in $G\Sp^\Sigma$ are homotopical for the $G$-global weak equivalences.
    Hence the claim again follows from the pointset level statement.
    For the details, we refer the reader to \cite[Proof of Theorem 5.10]{LLP}.
\end{proof}

\begin{definition}\label{def:gsym-sp}
    For a finite group $G$,
    we define the $G$-symmetric monoidal $G$-category of $G$-spectra as the restriction
    $\und{\Sp}_G^\tensor \coloneqq \und{\Sp}^\tensor|_{\Span(G)} \in \Mack_G(\CAT)$
    and $\und{\Sp}_G \coloneqq \und{\Sp}_G^\tensor|_{\Orb_G^\op}$.
    In view of \cref{prop:borel-global-eq},
    we may equivalently define $\und{\Sp}_G^\tensor$ as the functorial
    Dwyer--Kan localization
    \[
        \und{\Sp}_G^\tensor \simeq \DK \circ ((\infl_G\Sp^{\Sigma,\tensor}_{\projcof})^\flat_G, W_{\bullet\text{-stable}}).
    \]
\end{definition}

\begin{remark}\label{rem:sp-g-univ-prop}
    It was shown in \cite[Theorem 9.13]{CLL-Equivariant}
    that $\und{\Sp}_G$ is the free $G$-presentable
    $G$-stable $G$-category on one generator.
    Moreover, the underlying `fiberwise symmetric monoidal'
    $\Orb_G^\op \to \Mack(\CAT),\ G/H \mapsto \Sp_H$
    is the initial presentably symmetric monoidal $G$-stable $G$-category
    of \cite[Theorem 4.10]{Cnossen-Ambidexterity}.
\end{remark}

\begin{corollary}\label{cor:sp-compat-sift}
    $\und{\Sp}_G^\tensor, \und{\Sp}^\tensor$ and $\und{\Sp}_\gl^\tensor$
    are compatible with sifted colimits,
    i.e.~factor through the subcategory $\CAT(\sift)$ on (large) categories admitting
    sifted colimits and functors preserving them.
\end{corollary}

To compare $\und{\Sp}^\tensor$ and $\und{\Sp}^\tensor_\gl$, we will need the following Proposition.

\begin{proposition}\label{prop:lquil}
    The identity on $G\Sp^\Sigma$ induces Quillen adjunctions
    \[
        G\Sp^\Sigma_{G\text{-eqv.~proj.}} \rightleftarrows G\Sp^\Sigma_{G\text{-global flat}}
        \qquad\text{and}\qquad
        G\Sp^\Sigma_{G\text{-global flat}} \rightleftarrows G\Sp^\Sigma_{G\text{-eqv.~flat}}.
    \]
    These derive to exhibit $\Sp_G$ as left and right Bousfield localization of $\Sp_{G\dgl}$:
    \[\begin{tikzcd}
        {\Sp_G} && {\Sp_{G\dgl}}
        \arrow["{i_!}"{description}, bend left=20pt, hook, from=1-1, to=1-3]
        \arrow["{i^*}"{description}, from=1-3, to=1-1]
        \arrow["{i_*}"{description}, bend right=20pt, hook', from=1-1, to=1-3]
    \end{tikzcd}\]
\end{proposition}
\begin{proof}
    This is \cite[Proposition 3.3.1 and Corollary 3.3.3]{Lenz-phd}.
\end{proof}

\begin{construction}[Comparing $\und{\Sp}^\tensor$ and $\und{\Sp}_\gl^\tensor$]\label{con:comp-eq-gl}
    Consider the following diagram
    \begin{equation}\label[diag]{diag:eq-gl-comp}\begin{tikzcd}[cramped]
        {\int (\Sp^{\Sigma,\tensor})^\flat} & {\int (\Sp^{\Sigma,\tensor})^\flat} & {\int (\Sp^{\Sigma,\tensor})^\flat} \\
        {\int \und{\Sp}^\tensor} & {\int \und{\Sp}^\tensor_\gl} & {\int \und{\Sp}^\tensor}
        \arrow["\id", from=1-1, to=1-2]
        \arrow["{\gamma_\eqv}"', from=1-1, to=2-1]
        \arrow["\id", from=1-2, to=1-3]
        \arrow["{\gamma_\gl}", from=1-2, to=2-2]
        \arrow["{\gamma_\eqv}", from=1-3, to=2-3]
        \arrow[Rightarrow, from=2-1, to=1-2]
        \arrow["{i_! \coloneqq \bfL\id}"', hook, from=2-1, to=2-2]
        \arrow["{i^* \coloneqq \DK(\id)}"', from=2-2, to=2-3]
    \end{tikzcd}\end{equation}
    where $\gamma_\eqv$ respectively $\gamma_\gl$ denote the localization at the fiberwise equivariant respectively global weak equivalences.
    The right square commutes since $G$-global equivalences are in particular $G$-equivariant weak equivalences,
    so the top right horizontal map is homotopical and localizes to $i^*$.
    The left square is obtained via \cref{prop:idk}
    by equipping the left respectively right $\int (\Sp^{\Sigma,\tensor})^\flat$ with the structure of left deformable
    cocartesian fibrations via the $G$-equivariant projective respectively $G$-global flat model structures
    (e.g.~on a product $\prod_{i=1}^n G_i$ we use the pointwise $G_i$-equivariant resp.~$G_i$-global equivalences).
    The fact that all cocartesian pushforward functors are compatible with the resulting deformations
    easily follows from \cref{thm:model}.
    We then note that the identity map is fiberwise compatible with the deformations
    since $G$-stable equivalences between projectively cofibrant spectra are already $G$-global equivalences
    by \cref{prop:lquil}.
    We get that $i_!$ preserves cocartesian edges,
    and is equivalently expressed as the localization of the homotopical map where we restrict to the images of the deformations on both sides:
    \[
        (\smallint (\Sp^{\Sigma,\tensor}_{\projcof})^\flat,W_{\bullet\text{-stable}})
        \subseteq (\smallint (\Sp^{\Sigma,\tensor}_{\cflat})^\flat,W_{\bullet\text{-global}}).
    \]
    Since $i^*$ is both the left and right derived functor of $\id$ (c.f.~\cite[Proposition 4.3]{Hinich-Derived})
    and clearly $\id \dashv \id$, we get that also $i_! \dashv i^*$ since deriving preserves adjunctions (c.f.~\cite[Proposition 2.3.5]{Hinich-Derived}).
    Fiberwise, this is given by the right Bousfield localization $\Sp_G \rightleftarrows \Sp_{G\dgl}$
    of the above proposition.
    Finally, we note that by the dual of \cite[Proposition 7.3.2.6]{HA} $i_!$ admits a relative
    right adjoint, so by uniqueness of adjoints $i_! \dashv i^*$ is a relative adjunction over $\Span_\Forb(\Fglo)$.
\end{construction}

We summarize the above discussion into the following proposition.

\begin{proposition}\label{prop:e-lax-right-adjoint}
    The fully faithful normed global functor $i_! \colon \und{\Sp}^\tensor \subseteq \und{\Sp}^\tensor_\gl$
    admits a lax normed right adjoint $i^*$ which is strong on $\Span(\Forb)$.
    Restricting along $\Span(\F) \to \Span_\Forb(\Fglo), * \mapsto G$,
    this recovers the right Bousfield localization $i_! \colon \Sp_{G} \rightleftarrows \Sp_{G\dgl} \noloc i^*$ and shows that the natural lax symmetric
    monoidal structure on $i^*$ is strong.
\end{proposition}
\begin{proof}
    It only remains to see that $i^*$ is strong on $\Span(\Forb)$.
    Clearly $i^*$ is equivalently given by the localization of the homotopical map
    \[
        (\smallint (\Sp^{\Sigma,\tensor}_\cflat)^\flat, W_{\bullet\text{-global}}) \xto{\id} (\smallint (\Sp^{\Sigma,\tensor}_\cflat)^\flat, W_{\bullet\text{-stable}}).
    \]
    By \cref{thm:model} all cocartesian pushforwards on the left hand side are homotopical,
    and the cocartesian pushforwards along maps in $\Span(\Forb)$
    are homotopical for the right side. The claim then follows from \cref{prop:deriv-cocart-mor}.
\end{proof}

\begin{remark}
    By work of Lenz--Stahlhauer \cite{Lenz-Stahlhauer} there is a robust theory of so-called
    global model categories, which essentially consist of a collection of (two Quillen equivalent) model structures
    on the category of $G$-equivariant objects in $\cC$ for varying finite groups $G$,
    and these must satisfy various compatibility conditions with respect to restriction
    along group homomorphisms.
    For example, $(\Sp^{\Sigma})^\flat$ equipped with the levelwise $G$-global flat model structures
    can be endowed with the structure of a global model category, see \cite[Example 2.3, Example 2.14]{Lenz-Stahlhauer}.
    Similarly, there is a notion of global Quillen adjunction.
    The methods developed in \cref{subsec:loc} (and its dual version for cartesian fibrations
    which are right-derivable) can be used in a fashion analogous to the above
    to derive global model categories and global Quillen adjunctions
    to global ($\infty-$)categories and adjunctions between them.
\end{remark}

We will now build the comparison functors from categories of ultra-commutative $G$(-global) ring spectra
to our categories of normed algebras.

\begin{construction}\label{con:ucomm-comp}
    We have a commutative diagram of global functors
    \[\begin{tikzcd}[cramped]
        {\CAlg(\Sp^{\Sigma,\tensor}_\cflat)^\flat} & {\und{\CAlg}_\gl((\Sp^{\Sigma,\tensor}_\cflat)^\flat)} && {\und{\CAlg}_\gl(\und{\Sp}_\gl^\tensor)} \\
        & {(\Sp^{\Sigma}_\cflat)^\flat} && {\und{\Sp}_\gl}
        \arrow["{\bbU^\flat}"', from=1-1, to=2-2]
        \arrow["\simeq"', from=1-2, to=1-1]
        \arrow["{\und{\CAlg}_\gl(\gamma_\gl)}", from=1-2, to=1-4]
        \arrow["\bbU", from=1-2, to=2-2]
        \arrow["\bbU", from=1-4, to=2-4]
        \arrow["{\gamma_\gl}"', from=2-2, to=2-4]
    \end{tikzcd}\]
    where $\gamma_\gl$ denotes the localization from \cref{con:sp-gl},
    and the equivalence compatible with the forgetful functors comes from \cref{prop:gl-borel}.
    In particular, since the forgetful functors are conservative,
    it is now clear that in level $G \in \Glo^\op$, the entire top horizontal composite inverts $G$-global equivalences.
    Hence, if we define the global category $\und{\UComm}_\gl$ as the functorial Dwyer--Kan localization
    $\DK \circ (\CAlg(\Sp^{\Sigma,\tensor}_\cflat)^\flat, W_{\bullet\text{-global}})$,
    then by \cref{cor:compare-dk},
    the universal property of DK-localization yields a unique factorization of global functors
    \begin{equation}\label[diag]{diag:ucomm-gl}\begin{tikzcd}[cramped]
        {\und{\CAlg}_\gl((\Sp^{\Sigma,\tensor}_\cflat)^\flat)} && {\und{\CAlg}_\gl(\und{\Sp}_\gl^\tensor)} \\
        {\CAlg(\Sp^{\Sigma,\tensor}_\cflat)^\flat} && {\und{\UComm}_\gl}
        \arrow["{\und{\CAlg}_\gl(\gamma_\gl)}", from=1-1, to=1-3]
        \arrow["\simeq"', from=1-1, to=2-1]
        \arrow["{\gamma_\gl}"', from=2-1, to=2-3]
        \arrow["{\Phi_\gl}"', dashed, from=2-3, to=1-3]
    \end{tikzcd}\end{equation}
    Note that by construction $\Phi_\gl$ is compatible with the forgetful functors to $\und{\Sp}_\gl$.
\end{construction}

\begin{construction}\label{con:ucomm-G}
    Analogously to the above,
    for a fixed $G$ one obtains a commutative diagram of $G$-categories
    \[\begin{tikzcd}[cramped]
        {\und{\CAlg}_G((\infl_G\Sp^{\Sigma,\tensor}_\cflat)^\flat)} && {\und{\CAlg}_G(\und{\Sp}_G^\tensor)} \\
        {(\infl_G\CAlg(\Sp^{\Sigma,\tensor}_\cflat))^\flat} && {\und{\UComm}_G}
        \arrow["{\und{\CAlg}_G(\gamma_G)}", from=1-1, to=1-3]
        \arrow["\simeq"', from=1-1, to=2-1]
        \arrow["{\gamma_G}"', from=2-1, to=2-3]
        \arrow["{\Phi_G}"', dashed, from=2-3, to=1-3]
    \end{tikzcd}\]
    where the left vertical equivalence comes from \cref{thm:borel-eqv},
    and $\gamma_G$ exhibits $\und{\UCom}_G$ as functorial Dwyer--Kan localization
    of $(\infl_G \CAlg(\Sp^{\Sigma,\tensor}_\cflat))^\flat$
    at the  underlying $H$-stable weak equivalences in level $G/H$.
    Again, by construction $\Phi_G$ is compatible with the forgetful functors to $\und{\Sp}_G$.
\end{construction}

\begin{construction}\label{con:ucomm-eq}
    Recall the functor $i^* \colon \int \und{\Sp}^\tensor_\gl \to \int \und{\Sp}^\tensor$ from \cref{prop:e-lax-right-adjoint}. Since it preserves cocartesian edges over $\Span(\Forb)$, postcomposition with $i^*$ induces a global functor
    \[
        \und{\CAlg}_\gl(\und{\Sp}_\gl^\tensor)
        \to \und{\CAlg}_\eq(\und{\Sp}^\tensor) \coloneqq \Fun_{/\Span_\Forb(\Fglo)}^{\Forb^\op\dcc}(\Span_{\Forb}(\Fglo_{/\bullet}), \smallint \und{\Sp}^\tensor)
    \]
    and thus a global comparison functor $\Phi_\eq \coloneqq (i^*)_* \Phi_\gl \colon \und{\UCom}_\gl \to \und{\CAlg}_\eq(\und{\Sp}^\tensor)$.
    Note that $\und{\CAlg}_\eq(\und{\Sp}^\tensor)$ admits a forgetful functor $\bbU$ to
    \[
        \Fun_{/\Fglo^\op}^{\Forb^\op\dcc}((\FGlo_{/\bullet})^\op,\smallint \und{\Sp})
        \eqqcolon \plaxlim_{K \in (\Fglo_{/\bullet})^\Op} \Sp_K
        \simeq \und{\Sp}_\gl
    \]
    where the equivalence was shown in \cite{Globalization}.
    This yields a commutative diagram
    \[\begin{tikzcd}[cramped,column sep=scriptsize]
        & {\und{\CAlg}_\gl(\und{\Sp}_\gl^\tensor)} & {\und{\CAlg}_\eq(\und{\Sp}_\gl^\tensor)} & {\und{\CAlg}_\eq(\und{\Sp}^\tensor)} \\
        {\und{\Sp}_\gl} & {\lim_{K \in (\Fglo_{/\bullet})^\op}\Sp_{K\dgl}} & {\laxlim^\dagger_{K \in (\Fglo_{/\bullet})^\op}\Sp_{K\dgl}} & {\laxlim^\dagger_{K \in (\Fglo_{/\bullet})^\op}\Sp_{K}} & {\und{\Sp}_\gl}
        \arrow[hook, from=1-2, to=1-3]
        \arrow[from=1-2, to=2-2]
        \arrow["{i^*}", from=1-3, to=1-4]
        \arrow[from=1-3, to=2-3]
        \arrow[from=1-4, to=2-4]
        \arrow[Rightarrow, no head, from=2-1, to=2-2]
        \arrow[hook, from=2-2, to=2-3]
        \arrow["{i^*}", from=2-3, to=2-4]
        \arrow["\simeq", from=2-4, to=2-5]
    \end{tikzcd}\]
    The bottom horizontal composite is equivalent to the identity,
    by the triangle identity for the Globalization adjunction constructed in \cite{Globalization},
    and the fact that $i_! \colon \und{\Sp} \to \und{\Sp}_\gl$ exhibits the target as the globalization of its source by \cite[Theorem B]{Globalization}.
    Thus also $\Phi_\eq$ is compatible with the forgetful functors to $\und{\Sp}_\gl$.

    One can equivalently construct $\Phi_\eq$ exactly as $\Phi_\gl$,
    by replacing the top horizontal map in \cref{diag:ucomm-gl} with postcomposition
    by $\gamma_\eq \colon \int (\Sp^{\Sigma,\tensor}_\cflat)^\flat \to \int \und{\Sp}^\tensor$
    which by \cref{con:sp-eq} then lands in $\und{\CAlg}_\eq(\und{\Sp}^\tensor)$.
    That this agrees with the previous definition of $\Phi_\eq$
    follows from the commutativity of the right square in \cref{diag:eq-gl-comp}.
\end{construction}

\begin{remark}
    We will not need this here, but let us remark
    that in joint work with Tobias Lenz and Sil Linskens \cite{LLP}
    we show that all the above comparison functors $\Phi_\gl,\Phi_\eq,\Phi_G$ are equivalences.
    In particular, we also have an equivalence of global categories $\und{\CAlg}_\gl(\und{\Sp}^\tensor_\gl) \simeq \und{\CAlg}_\eq(\und{\Sp}^\tensor)$.
    Moreover, we show that the comparison functors $\Phi_G$ for varying $G \in \Glo$
    assemble into an equivalence of global categories.
    Note that it is a priori not obvious at all how to even make the association $G \mapsto \CAlg_G(\und{\Sp}_G^\tensor)$
    into a global category, as there does not seem to be any natural map
    $\CAlg(\Sp) \to \CAlg_G(\und{\Sp}_G^\tensor)$ one could write down.
\end{remark}

There is another model for $\und{\Sp}_G,\und{\Sp}$ and $\und{\Sp}_\gl$ which will be
vital for us in \cref{sec:pic}. Namely, recall that there is an equivalence
$
    \Mack_G(\Sp) \simeq \Sp_G
$
between the category of spectral $G$-Mackey functors and the category
of genuine $G$-spectra, as first observed by Guillou--May \cite{Guillou-May}
(see also \cite[Appendix A]{CMNN}).
To turn this into an equivalence of parametrized categories and also
give a global version, we consider the following general definition.

Let $(\cF,\cN)$ be an extensive span pair and $A \in \cF$.
Since coproducts and pullbacks in $\cF_{/A}$ are computed in $\cF$,
one easily checks that also $(\cF_{/A},\cF_{/A} \times_\cF \cN)$ is extensive,
and that postcomposition in $A$ defines a functor $\Span_\cN(\cF_{/\bullet})\colon \cF \to \Cat^\oplus_{/\Span_\cN(\cF)}$.
Moreover, by \cref{ex:n-a} $\Span_\cN(\cF_{/A}) \to \Span_\cN(\cF)$
is simply $\und{A}^\Ncoprod$. It follows that $\colim_{\cF} \Span_\cN(\cF_{/\bullet}) \simeq \Span_\cN(\cF)$ since $(-)^\Ncoprod$ is a left adjoint, the colimit of the Yoneda embedding $\und{(-)} \colon \cF \to \PSh(\cF)$ is terminal,
and $(*)^\Ncoprod = \Span_\cN(\cF)$.

\begin{definition}\label{def:param-mack}
    The $\cF$-category
    of $\cN$-normed $\cF$-monoids in $\cE \in \Cat^\times$ is
    \[
        \und{\Mack}_\cF^\cN(\cE) \coloneqq \Fun^\times(\Span_{\cN}(\cF_{/\bullet}), \cE).
    \]
    It has underlying category $\lim_{\cF^\op}\und{\Mack}_\cF^\cN(\cE) \simeq \Mack_\cF^\cN(\cE)$.
\end{definition}

\begin{notation}
    We will also write $\und{\Mack}_G \coloneqq \und{\Mack}_{\F_G}^{\F_G}$ and $\und{\Mack}_\gl \coloneqq \und{\Mack}_\Fglo^\Forb$.
\end{notation}

\begin{theorem}\label{thm:spectral-mack}
    There are equivalences of $G$- respectively global categories
    \[
        \und{\Mack}_G(\Sp) \simeq \und{\Sp}_G
        \quad\text{and}\quad
        \und{\Mack}_\gl(\Sp) \simeq \und{\Sp}_\gl
    \]
    which levelwise identify evaluation with genuine fixed points, in that
    that following two triangles commute for all $K \leq H \leq G$ respectively $G \in \Fglo$ and $\phi \in \Fglo_{/G}$:
    \[\begin{tikzcd}[cramped]
        {\Mack_H(\Sp)} && {\Sp_H} & {\Mack_{G\dgl}(\Sp)} && {\Sp_{G\dgl}} \\
        & \Sp &&& \Sp
        \arrow["\simeq", from=1-1, to=1-3]
        \arrow["{\ev_{H/K}}"', from=1-1, to=2-2]
        \arrow["{(-)^K}", from=1-3, to=2-2]
        \arrow["\simeq", from=1-4, to=1-6]
        \arrow["{\ev_\phi}"', from=1-4, to=2-5]
        \arrow["{(-)^\phi}", from=1-6, to=2-5]
    \end{tikzcd}\]
\end{theorem}
\begin{proof}
    This follows from \cref{prop:sp-gl-univ-prop} and \cite[Corollary 9.14]{CLL-Semiadd}.
    An independent proof using the language of algebraic patterns \cite{BHS-Env}
    was also given in the author's master's thesis
    \cite[Corollary 2.56, Remark 2.57]{puetzstueck}.
\end{proof}

\begin{remark}
    In fact, one can even define natural normed structures on $\und{\Mack}_G(\Sp)$
    and $\und{\Mack}_\gl(\Sp)$, and it is shown in \cite[Section 8]{LLP}
    that the above equivalences then upgrade to equivalences
    of normed categories with $\und{\Sp}_G^\tensor$ and $\und{\Sp}_\gl^\tensor$.
    Note that the definition of $\und{\Mack}_\gl$ in \cite{LLP}
    agrees with ours by \cite[Remark 4.1.10]{Tambara}.
\end{remark}

\section{Modules}\label{sec:modules}

Consider an $\cN$-normed $\cF$-category $\cC$ for some extensive span pair $(\cF,\cN)$,
and let $R \in \CAlg_\cF^\cN(\cC)$ be a normed algebra in $\cC$.
The goal of this section is to construct (under suitable hypotheses) an $\cN$-normed $\cF$-category
$\und{\Mod}_R(\cC)$ of $R$-modules in $\cC$.
Moreover, we will upgrade the association $R \mapsto \und{\Mod}_R(\cC)$,
into a parametrized functor, and also analyze the naturality with respect to normed functors $\cC \to \cD$.

Recall that if $\cC$ is some symmetric monoidal category and $R \in \CAlg(\cC)$,
then we need $\cC$ to admit geometric realizations and its tensor product to preserve
them in both variables separately to also obtain an induced symmetric monoidal structure on $\Mod_R(\cC)$.
The resulting tensor product is then given by the geometric realization of the Bar construction
$M \tensor_R N \simeq |M \tensor R^{\tensor \bullet} \tensor N|$, see e.g.~\cite[Section 4.4]{HA}.
The analogue of this condition in the parametrized setting is that our normed categories
are compatible with geometric realizations, i.e.~factor through the category
$\Cat(\Delta^\op) \subset \Cat$ of categories admitting geometric realizations
and functors preserving them. We have already verified this for the examples
we are most interested in c.f.~\cref{cor:sp-compat-sift}.

For this section, it will be convenient to work in slightly larger generality than usual.
We introduce some notation which suggests the relation to $\cN$-normed $\cF$-categories
considered in the main text.

\begin{notation}
    Consider a map of semiadditive categories $\pi \colon T \to S$,
    and a class $\cK$ of small sifted categories.
    \begin{itemize}
        \item We refer to $\Fun^\times(S,\Cat(\cK))$
            as the category of $S$-normed categories that are
            compatible with $\cK$-colimits.
            Here $\Cat(\cK) \subset \Cat$ is the subcategory of categories
            admitting $\cK$-colimits and functors preserving them.
        \item Given an $S$-normed category $\cC$,
            we define the category of normed algebras in context $T$ to be
            the category
            \[
                \und{\CAlg}_S(\cC)(T)
                \coloneqq \Fun_{/S}^{\pr\dcc}(T,\smallint \cC)
                = \Fun_{/S}^\times(T,\smallint \cC)
            \]
            of maps $T \to \int \cC$ over $S$ which preserve cocartesian lifts of projection maps in $S$,
            or equivalently (by \cref{lem:obvious}) of product preserving functors over $S$.
            Note that this is completely functorial in $T$, and we denote the resulting functor by
            $\und{\CAlg}_{S}(\cC) \colon (\Cat^\sadd_{/S})^\op \to \Cat$.

        \item We refer to functors $(\Cat^\sadd_{/S})^\op \to \Cat$ as $S$-based categories,
            and natural transformations of them as $S$-based functors.
    \end{itemize}
\end{notation}

\begin{remark}\label{rem:mod-functoriality}
    We give a list of desiderata our parametrized module construction
    should satisfy. First, let us describe what it should do to a normed algebra.
    \begin{enumerate}
        \item[(i)] Given an $S$-normed category $\cC \in \Fun^\times(S,\Cat(\cK))$
            and a normed algebra $R \in \und{\CAlg}_S(\cC)(T)$,
            we want a $T$-normed category $\und{\Mod}_R(\cC) \in \Fun^\times(T,\Cat(\cK))$
            sending a map $f \colon A \to B$ in $T$ to the composite
            \[
                \hspace{3em}\Mod_{RA}(\cC(\pi A))
                \xto{f_\tensor} \Mod_{f_\tensor RA}(\cC(\pi B))
                \xto{RB \tensor_{(f_\tensor RA)} (-)} \Mod_{RB}(\cC(\pi B)),
            \]
            where $f_\tensor = \cC(\pi f)$.
    \end{enumerate}
    Next, we will describe the desired functoriality and naturality
    properties this construction should satisfy.
    \begin{enumerate}
        \item[(ii)] A map of algebras $R \to S$ in $\und{\CAlg}_S(\cC)(T)$ should induce a normed functor
            \[
                S \tensor_R - \colon \und{\Mod}_R(\cC) \to \und{\Mod}_S(\cC)
            \]
            given by the usual basechange
            $SA \tensor_{RA} - \colon \Mod_{RA}(\cC(\pi A)) \to \Mod_{SA}(\cC(\pi A))$
            at each $A \in T$.

        \item[(iii)] In fact, the above should be entirely functorial in $R \in \und{\CAlg}_S(\cC)(T)$
            and natural in $T$, ultimately giving rise to an $S$-based functor
            \[
                \und{\Mod}_{(-)}(\cC) \colon \und{\CAlg}_S(\cC) \to \Fun^\times(-,\Cat(\cK))
            \]

        \item[(iv)] Finally, the entire construction should also be functorial in the $S$-normed category $\cC$;
            In particular, given a normed functor $F \colon \cC \to \cD$, we want a natural transformation of $S$-based functors
            \[
                \und{\Mod}_{(-)}(\cC) \Rightarrow \und{\Mod}_{F(-)}(\cD)
            \]
            which in context $T \in (\Cat_{/S}^\sadd)^\op$ is given by the lax commuting square
            \begin{equation}\label{diag:mod-lax-nat-square}\begin{tikzcd}[cramped,column sep = scriptsize]
                {\und{\CAlg}_S(\cC)(T)} && {\Fun^\times(T,\Cat(\cK))} & R & {\und{\Mod}_R(\cC)} \\
                {\und{\CAlg}_S(\cD)(T)} && {\Fun^\times(T,\Cat(\cK))} & FR & {\und{\Mod}_{FR}(\cD)}
                \arrow["{\und{\Mod}_{(-)}(\cC)}", from=1-1, to=1-3]
                \arrow["{\und{\CAlg}_S(F)(T)}"', shorten <=1pt, from=1-1, to=2-1]
                \arrow["F"{description}, Rightarrow, from=1-3, to=2-1]
                \arrow[Rightarrow, no head, from=1-3, to=2-3]
                \arrow[maps to, from=1-4, to=1-5]
                \arrow[maps to, from=1-4, to=2-4]
                \arrow["F", Rightarrow, from=1-5, to=2-5]
                \arrow["{\und{\Mod}_{(-)}(\cD)}"', from=2-1, to=2-3]
                \arrow[maps to, from=2-4, to=2-5]
            \end{tikzcd}\end{equation}
            Here $F \colon \und{\Mod}_R(\cC) \to \und{\Mod}_{FR}(\cD)$ is a $T$-normed functor
            which at $A \in T$ is given by $\Mod_{RA}(\cC(\pi A)) \to \Mod_{FRA}(\cD(\pi A))$
            induced by the symmetric monoidal functor $F(A)$.
    \end{enumerate}
\end{remark}

One convenient way to package all of the above data succinctly is as follows.
Recall the notion of lax natural transformations from \cref{appendix:lax}.

\begin{theorem}\label{thm:mod}
    There exists a lax natural transformation
    \[\begin{tikzcd}[cramped]
        {\Fun^\times(S,\Cat)} &&& {\Fun^\times((\Cat_{/S}^\sadd)^\op, \what{\Cat})}
        \arrow[""{name=0, anchor=center, inner sep=0}, "{\und{\CAlg}_S(-)}", curve={height=-18pt}, from=1-1, to=1-4]
        \arrow[""{name=1, anchor=center, inner sep=0}, "{\const \Fun^\times(-,\Cat(\cK))}"', curve={height=18pt}, from=1-1, to=1-4]
        \arrow["\lax"', shorten <=5pt, shorten >=5pt, Rightarrow, from=0, to=1]
        \arrow["{\und{\Mod}}", draw=none, from=0, to=1]
    \end{tikzcd}\]
    satisfying all of the above points, e.g.~having lax naturality squares as in (\ref{diag:mod-lax-nat-square}).
\end{theorem}

\begin{remark}
    Let $(\cF,\cN)$ be an extensive span pair.
    We specialize the above results to $\cN$-normed $\cF$-categories by picking
    $S = \Span_\cN(\cF)$. Recall from \cref{def:param-mack}
    that the $\cF$-category structure on $\und{\Mack}_\cF^\cN(\Cat(\cK))$
    is induced by precomposing with
    $\Span_{\cN}(\cF_{/\bullet}) \colon \cF \to \Cat^\sadd_{/\Span_\cN(\cF)}$.
    Given $\cC \in \Mack_\cF^\cN(\Cat(\cK))$ we have an inclusion
    \[
        \und{\CAlg}_\cF^\cN(\cC)(A)
        \simeq \Fun_{/S}^{\cF^\op\dcc}(\Span_{\cN}(\cF_{/A}),\smallint \cC)
        \subseteq \Fun_{/S}^{\pr\dcc}(\Span_\cN(\cF_{/A}),\smallint \cC)
    \]
    of a full subcategory, since projection maps in $\Span_\cN(\cF)$ are backwards maps by \cref{prop:span-semiadd}.
    This is clearly natural in $A \in \cF^\op$.
    Precomposing the above lax natural transformation with this, we obtain a lax natural transformation
    \begin{equation}\label[diag]{diag:mod-lax-nat}\begin{tikzcd}[cramped]
        {\Mack_\cF^\cN(\Cat(\cK))} && {\Fun^\times(\cF^\op,\what{\Cat})}
        \arrow[""{name=0, anchor=center, inner sep=0}, "{\und{\CAlg}_\cF^\cN}", curve={height=-18pt}, from=1-1, to=1-3]
        \arrow[""{name=1, anchor=center, inner sep=0}, "{\const \und{\Mack}_\cF^\cN(\Cat(\cK))}"', curve={height=18pt}, from=1-1, to=1-3]
        \arrow["\lax"', shorten <=5pt, shorten >=5pt, Rightarrow, from=0, to=1]
        \arrow["{\und{\Mod}}", draw=none, from=0, to=1]
    \end{tikzcd}\end{equation}
    satisfying all the analogous properties listed above.
\end{remark}

\begin{remark}
    Our construction improves and provides further details on the construction of $\und{\Mod}_R(\cC)$
    for a fixed $R$ and $\cC$ given in \cite[Proposition 7.6(4)]{Bachmann-Hoyois}.
    This, together with the polynomial functoriality of connective algebraic K-theory \cite{BGMN},
    has been used in in \cite[Section 4.3]{Elmanto-Haugseng} to show that the connective
    equivariant algebraic K-theory of a $G$-$\E_\infty$ ring $R \in \CAlg_G(\und{\Sp}_G^\tensor)$ enhances
    to a (grouplike) space-valued Tambara functor $\Omega^\infty K(R) \colon \text{Bispan}(\F_G) \to \Spc$.
    By \cite[Theorem B]{Tambara}, this data is equivalent to a connective normed $G$-ring spectrum,
    so as noted in \cite[Corollary C]{Tambara} we obtain $\Omega^\infty K(R) \in \CAlg_G(\und{\Sp}_G^{\tensor,\geq 0})$,
    where $\und{\Sp}_G^{\tensor,\geq 0} \subseteq \und{\Sp}^\tensor_G$
    is the full normed $G$-subcategory consisting levelwise of the connective $H$-spectra.
    Our added functoriality in $R$ and $H \leq G$
    can then be used to show that this extends connective equivariant algebraic $K$-theory to a $G$-functor
    \[
        \und{K} \colon \und{\CAlg}_G(\und{\Sp}_G^\tensor) \to \und{\CAlg}_G(\und{\Sp}_G^{\tensor,\geq 0}).
    \]
\end{remark}

We now begin with our construction.
Luckily for us, Lurie already showed in \cite[Sections 4.8.3-4.8.5]{HA} that for a monoidal category $\cC$ compatible with $\cK$-colimits and an algebra $R \in \Alg(\cC)$,
the category $\RMod_R(\cC)$ of right $R$-modules with its left $\cC$-action
is functorial in both $\cC$ and $R$ in the expected sense.
He does this by building a functor $\Theta \colon \Cat^{\Alg}(\cK) \to \Cat^\Mod(\cK)$
\cite[Construction 4.8.3.24]{HA}
from a certain category of pairs $(\cC,R)$ of monoidal categories $\cC$ compatible with $\cK$-colimits
and algebras $R \in \Alg(\cC)$ to a category of pairs $(\cC,\cM)$ of monoidal categories
$\cC$ compatible with $\cK$-colimits and categories $\cM$ that admit a left $\cC$-action.
We will not need this module structure, and will thus postcompose with a forgetful functor $\Cat^\Mod(\cK) \to \Cat(\cK),\ (\cC,\cM) \to \cM$.

The idea of our construction is now that given an $S$-normed category $\cC$
and a normed algebra $R$ in it, we should obtain a functor $(\cC,R) \colon S \to \Cat^{\Alg}(\cK)$
which we can then postcompose with Lurie's $\Theta$ to obtain $\und{\Mod}_R(\cC) \colon S \to \Cat(\cK)$.
To this end, it will be convenient to give an equivalent description of $\Cat^{\Alg}(\cK)$,
which we do now.

Let $p \colon \Cat_{*\sslash} \to \Cat$ denote the universal cocartesian fibration,
i.e.~the cocartesian unstraightening of the identity on $\Cat$.
Moreover, let $p^\cK \colon \Cat_{*\sslash}(\cK) \to \Cat(\cK)$ denote the pullback
along the subcategory inclusion $\Cat(\cK) \subset \Cat$.
For a category with finite products $\cC$,
denote by $\Mon(\cC) \subseteq \Fun(\Delta^\op,\cC)$ the category of monoids in $\cC$
in the sense of \cite[Definition 4.1.2.5]{HA}.
We will need the following facts shown in \cref{appendix:cocart}, c.f.~\cref{prop:univ-cocart-prod,prop:mack-univ-calg}:
\begin{itemize}
    \item $\Cat_{*\sslash}(\cK)$ admits and $p^\cK$ preserves finite products.

    \item The functor $\Mon(p^\cK) \colon \Mon(\Cat_{*\sslash}(\cK)) \to \Mon(\Cat(\cK))$
        is a cocartesian fibration classifying $\Alg(-) \colon \Mon(\Cat(\cK)) \to \Cat$.

    \item The functor $\Mack(p^\cK) \colon \Mack(\Cat_{*\sslash}(\cK)) \to \Mack(\Cat(\cK))$
        is a cocartesian fibration classifying $\CAlg(-) \colon \Mack(\Cat(\cK)) \to \Cat$.
\end{itemize}

\begin{proposition}
    The cocartesian fibration $\Cat^{\Alg}(\cK) \to \Mon(\Cat(\cK))$
    defined by Lurie in \cite[Construction 4.8.3.24, Proposition 4.8.5.1]{HA}
    is equivalent to $\Mon(p^\cK)$.
\end{proposition}
\begin{proof}
    We begin by recalling some details of Lurie's construction.
    Lurie uses the model $\Mon_{\Assoc}(-)$ defined in \cite[2.4.2.1]{HA}
    for monoids in a monoidal category,
    where $\Assoc^\tensor$ is the associative operad defined in \cite[Definition 4.1.1.3]{HA}.
    By \cite[Proposition 4.1.2.10]{HA} there is a natural equivalence $\Mon_{\Assoc}(-) \simeq \Mon(-)$
    induced by precomposing with an approximation $\text{Cut} \colon \Delta^\op \to \Assoc^\tensor$
    to the operad $\Assoc^\tensor$.
    This also allows us to identify the category of algebras in a monoidal category $\cC \in \Mon(\Cat)$
    with $\Fun_{/\Delta^\op}^{\inert\dcc}(\Delta^\op,\int \cC)$, the category of sections
    that are cocartesian on inert maps.
    We will implicitly use these identifications in the following.

    In \cite[4.8.3.7]{HA} Lurie defines the category $\Cat^{\Alg}$ as follows.
    Consider the evaluation $\ev \colon \Delta^\op \times \Mon(\Cat) \to \Cat$
    and let $\wtilde{\Mon}(\Cat)$ be its cocartesian unstraightening.
    We have pullback squares
    \[\begin{tikzcd}[cramped]
        {\Fun_{/\Delta^\op}(\Delta^\op,\int \cC)} & {\Cat^{\Alg}_\all} & {\Fun_{/\Delta^\op}(\Delta^\op,\wtilde{\Mon}(\Cat))} \\
        {*} & {\Mon(\Cat)} & {\Fun_{/\Delta^\op}(\Delta^\op, \Delta^\op \times \Mon(\Cat))}
        \arrow[from=1-1, to=1-2]
        \arrow[from=1-1, to=2-1]
        \arrow[from=1-2, to=1-3]
        \arrow[from=1-2, to=2-2]
        \arrow[from=1-3, to=2-3]
        \arrow[""{name=0, anchor=center, inner sep=0}, "\cC", from=2-1, to=2-2]
        \arrow[""{name=1, anchor=center, inner sep=0}, from=2-2, to=2-3]
        \arrow["\lrcorner"{anchor=center, pos=0.125}, draw=none, from=1-1, to=0]
        \arrow["\lrcorner"{anchor=center, pos=0.125}, draw=none, from=1-2, to=1]
    \end{tikzcd}\]
    Here the bottom right functor sends $\cC$ to the functor $[n] \mapsto ([n],\cC)$,
    i.e.~is curried from the identity on $\Delta^\op \times \Mon(\Cat)$.
    To compute the fiber, note that under the equivalence
    $* \simeq \Fun_{/\Delta^\op}(\Delta^\op,\Delta^\op)$ the bottom horizontal
    composite is postcomposing with $[n] \mapsto ([n],\cC)$.
    Since $\Fun_{/\Delta^\op}(\Delta^\op,-)$ preserves pullbacks,
    the description of the fiber follows from $\ev \circ\,([n] \mapsto ([n],\cC)) = \cC$
    and pullback pasting.

    Lurie defines $\Cat^{\Alg} \subseteq \Cat^{\Alg}_\all$ as the full subcategory
    spanned by the pairs $(\cC,A)$ where $\cC$ is a monoidal category and $A$ an algebra in $\cC$, i.e.~so that the fibers are
    $\Alg(\cC) = \Fun_{/\Delta^\op}^{\inert\dcc}(\Delta^\op,\int\cC)$.
    Observe that we have pullback squares
    \[\hspace{-1em}\begin{tikzcd}[cramped,column sep =tiny]
	{\Cat^{\Alg}_\all} & {\Fun_{/\Delta^\op}(\Delta^\op,\wtilde{\Mon}(\Cat))} & {\Fun(\Delta^\op,\wtilde{\Mon}(\Cat))} & {\Fun(\Delta^\op,\Cat_{*\sslash})} \\
	{\Mon(\Cat)} & {\Fun_{/\Delta^\op}(\Delta^\op, \Delta^\op \times \Mon(\Cat))} & {\Fun(\Delta^\op,\Delta^\op \times \Mon(\Cat))} & {\Fun(\Delta^\op,\Cat)} \\
	& {*} & {\Fun(\Delta^\op,\Delta^\op)}
	\arrow[from=1-1, to=1-2]
	\arrow[from=1-1, to=2-1]
	\arrow[from=1-2, to=1-3]
	\arrow[from=1-2, to=2-2]
	\arrow[from=1-3, to=1-4]
	\arrow[from=1-3, to=2-3]
	\arrow["{p_*}", from=1-4, to=2-4]
	\arrow[""{name=0, anchor=center, inner sep=0}, from=2-1, to=2-2]
	\arrow[""{name=1, anchor=center, inner sep=0}, from=2-2, to=2-3]
	\arrow[from=2-2, to=3-2]
	\arrow[""{name=2, anchor=center, inner sep=0}, from=2-3, to=2-4]
	\arrow[from=2-3, to=3-3]
	\arrow[""{name=3, anchor=center, inner sep=0}, "{\id_{\Delta^\op}}", from=3-2, to=3-3]
	\arrow["\lrcorner"{anchor=center, pos=0.125}, draw=none, from=1-1, to=0]
	\arrow["\lrcorner"{anchor=center, pos=0.125}, draw=none, from=1-2, to=1]
	\arrow["\lrcorner"{anchor=center, pos=0.125}, draw=none, from=1-3, to=2]
	\arrow["\lrcorner"{anchor=center, pos=0.125}, draw=none, from=2-2, to=3]
\end{tikzcd}\]
    The composite $\Mon(\Cat) \to \Fun(\Delta^\op,\Cat)$ is the fully faithful inclusion,
    which therefore shows that $\Cat^{\Alg}_\all$ is the pullback of $p_*$ along $\Mon(\Cat) \subseteq \Fun(\Delta^\op,\Cat)$.
    Since $p$ preserves finite products, we obtain a fully faithful inclusion $i$:
    \[\begin{tikzcd}[cramped]
        {\Mon(\Cat_{*\sslash})} & {\Cat^{\Alg}_\all} & {\Fun(\Delta^\op,\Cat_{*\sslash})} \\
        & {\Mon(\Cat)} & {\Fun(\Delta^\op,\Cat)}
        \arrow["i", hook, from=1-1, to=1-2]
        \arrow["{p_*}"', from=1-1, to=2-2]
        \arrow[hook, from=1-2, to=1-3]
        \arrow[from=1-2, to=2-2]
        \arrow["\lrcorner"{anchor=center, pos=0.125}, draw=none, from=1-2, to=2-3]
        \arrow["{p_*}", from=1-3, to=2-3]
        \arrow[hook, from=2-2, to=2-3]
    \end{tikzcd}\]
    But now it follows from the definition of $\Cat^{\Alg}$
    and the fact that $\Mon(p)$ classifies $\Alg(-)$ that $i$ is a fiberwise equivalence, hence an equivalence.
    Finally, since $\Mon(-)$ preserves pullbacks, we see that pulling back along the forgetful
    $\Mon(\Cat(\cK)) \to \Mon(\Cat)$ yields the desired equivalence of cocartesian fibrations.
\end{proof}

\begin{definition}
    We will denote $\Cat^{\CAlg}(\cK) \coloneqq \Mack(\Cat_{*\sslash}(\cK))$.
    By \cref{prop:mack-univ-calg} we have a canonical forgetful functor
    $\fgt \colon \Cat^{\CAlg}(\cK) \to \Cat^{\Alg}(\cK)$
    sending a pair $(\cC,R)$ of a symmetric monoidal category $\cC$ and a commutative algebra $R \in \CAlg(\cC)$ to the same pair $(\cC,R)$, now viewed as monoidal category and ($\E_1$-)algebra.
\end{definition}

We can now compose this with Lurie's functor
$\Theta \colon \Cat^{\Alg}(\cK) \to \Cat^{\Mod}(\cK) \to \Cat(\cK)$ discussed above.

\begin{corollary}\label{cor:theta}
    Let $\cK$ be a class of small sifted categories containing $\Delta^\op$.
    The functor
    \[
        \Theta' \colon \Cat^{\CAlg}(\cK) \xto{\fgt} \Cat^{\Alg}(\cK) \xto{\Theta} \Cat(\cK)
    \]
    preserves finite products, and sends an edge $F \colon (\cC,R) \to (\cD,S)$ in $\Cat^{\CAlg}(\cK)$
    consisting of a symmetric monoidal
    $F \colon \cC \to \cD$ and a map of algebras $FR \to S$ to the composite functor
    \[
        \Mod_R(\cC) \xto{F} \Mod_{FR}(\cD) \xto{S \tensor_{FR} -} \Mod_S(\cD).
    \]
\end{corollary}
\begin{proof}
    Lurie constructs the functor $\Theta$ in \cite[Construction 4.8.3.24]{HA}.
    By \cite[Remark 4.8.3.25, Theorem 4.8.5.1, Theorem 4.5.3.1]{HA} this has the claimed functoriality.
    Since $\fgt$ is identified with the forgetful map $\Mack(\Cat_{*\sslash}) \to \Mon(\Cat_{*\sslash})$
    by definition and the above proposition, this clearly preserves products.
    It follows from \cite[Notation 4.8.5.14, Theorem 4.8.5.16]{HA} that also $\Theta$ and the forgetful functor $\Cat^{\Mod}(\cK) \to \Cat(\cK)$
    preserve finite products, hence $\Theta'$ does as well.
\end{proof}

We are now equipped to perform the main construction of this section.

\begin{proof}[Proof of \cref{thm:mod}]
    Note that it suffices to build a lax natural transformation
    \begin{equation}\label{eq:lax-nat}
        \Fun^\times_{/S}(\bullet,-) \xRightarrow{\lax} \const \Fun^\times(\bullet, \Cat_{*\sslash}(\cK)),
    \end{equation}
    of functors $\Fun^\times(S,\Cat(\cK)) \to \Fun^\times((\Cat^\sadd_{/S})^\op,\Cat)$
    which in context $\bullet = T \in (\Cat^\sadd_{/S})^\op$ has lax naturality squares
    \begin{equation}\label[diag]{diag:lax-nat-fws-inc}\begin{tikzcd}[cramped]
        {\Fun^\times_{/S}(T,\cC)} & {\Fun^\times(T,\Cat_{*\sslash}(\cK))} & R & {(s \mapsto (\cC,Rs))} \\
        {\Fun^\times_{/S}(T,\cD)} & {\Fun^\times(T,\Cat_{*\sslash}(\cK))} & FR & {(s \mapsto (\cD,FRs))}
        \arrow["{(\cC,-)}", from=1-1, to=1-2]
        \arrow["{F_*}"', from=1-1, to=2-1]
        \arrow["F"{description}, Rightarrow, from=1-2, to=2-1]
        \arrow[Rightarrow, no head, from=1-2, to=2-2]
        \arrow[maps to, from=1-3, to=1-4]
        \arrow[maps to, from=1-3, to=2-3]
        \arrow["F", Rightarrow, from=1-4, to=2-4]
        \arrow["{(\cD,-)}"', from=2-1, to=2-2]
        \arrow[maps to, from=2-3, to=2-4]
    \end{tikzcd}\end{equation}
    Once we have this, we can postcompose the natural transformation
    of constant functors induced by the composite of $S$-based functors
    \[
        \Fun^\times(\bullet,\Cat_{*\sslash}(\cK))
        \xleftarrow{\simeq} \Fun^\times(\bullet,\Cat^{\CAlg}(\cK))
        \xto{\Theta'_*} \Fun^\times(\bullet,\Cat(\cK)).
    \]
    The first map postcomposes with
    $\fgt \colon \Cat^{\CAlg}(\cK) = \Mack(\Cat_{*\sslash}(\cK)) \to \Cat_{*\sslash}(\cK)$,
    which is an equivalence since $\bullet$ is semiadditive, see \cref{prop:mack}.
    Using the description of the functoriality of $\Theta'$ from \cref{cor:theta}
    and the description of the lax naturality above,
    we see that the composite lax natural transformation
    $\und{\CAlg}_S(-) \Rightarrow \const \Fun^\times(-,\Cat(\cK))$
    constructed in this way satisfies all the properties listed in \cref{rem:mod-functoriality}.
    For example, evaluating at an $S$-normed category $\cC \in \Fun^\times(S,\Cat(\cK))$,
    we obtain an $S$-based functor which in context $T \in (\Cat_{/S}^\sadd)^\op$ is given by
    \begin{align*}
        \und{\CAlg}_S(\cC)(T)
        = \Fun_{/S}^{\times}(T,\smallint \cC)
        \xto{(\cC,-)} \Fun^\times(T,\Cat_{*\sslash}(\cK))
        &\simeq \Fun^\times(T,\Cat^{\CAlg}(\cK))\\
        &\xto{\Theta'_*} \Fun^\times(T,\Cat(\cK)).
    \end{align*}
    The description of the functoriality of $\Theta'$ then easily yields the first three points of \cref{rem:mod-functoriality}.

    We will first construct the lax natural transformation
    with lax naturality squares (\ref{diag:lax-nat-fws-inc}) for a fixed $T$,
    and then consider the functoriality in $T$.
    To this end, consider the following lax map of cocartesian fibrations:
    \[\hspace{-2em}\begin{tikzcd}[cramped,column sep=tiny]
        {\Fun^\times(T,\Cat_{*\sslash}(\cK)) \times_{\Fun^\times(T,\Cat(\cK))}\Fun^\times(S,\Cat(\cK))} & {\Fun^\times(T,\Cat_{*\sslash}(\cK)) \times \Fun^\times(S,\Cat(\cK))} \\
        {\Fun^\times(S,\Cat(\cK))}
        \arrow[from=1-1, to=1-2]
        \arrow["q"', from=1-1, to=2-1]
        \arrow[from=1-2, to=2-1]
    \end{tikzcd}\]
    We know from \cref{prop:univ-cocart-prod}
    that $p^\cK_* \colon \Fun^\times(T,\Cat_{*\sslash}(\cK)) \to \Fun^\times(T,\Cat(\cK))$
    is the cocartesian fibration classifying $\Fun^\times_{/T}(T,\int -)$.
    In view of the natural\footnote{Using \cite[Theorem 3.3]{HHLN-Two} one easily checks that this is even natural in $T \in (\Cat^\times_{/S})^\op$.} equivalence $\Fun^\times_{/T}(T,\pi^*-) \simeq \Fun^\times_{/S}(T,-)$,
    we see that $q$ classifies the functor $\Fun^\times_{/S}(T,\int -)$.
    It then follows from \cref{lem:cocart-lax-square} that the above lax map of cocartesian fibrations
    precisely encodes the lax natural transformation given by inclusion of the fibers of $q$
    into the total space of $q$, followed by (the natural transformation of constant functors induced by)
    the projection from the total space to $\Fun^\times(T,\Cat_{*\sslash})$.
    This is precisely the lax naturality square from (\ref{diag:lax-nat-fws-inc}).

    Note that the above diagram is functorial in $T \in (\Cat^\sadd_{/S})^\op$,
    in that each $T' \to T$ over $S$ induces a map of cocartesian fibrations for the both legs;
    this is obvious for the right one, and for $q$ follows by functoriality of pullbacks.
    Hence we have constructed a functor
    \[
        (\Cat_{/S}^\sadd)^\op \times \Delta^1 \to \Cocart^\lax(\Fun^\times(S,\Cat(\cK)))
    \]
    whose restriction to $(\Cat_{/S}^\sadd)^\op \times \{0,1\}$ factors through $\Cocart(\Fun^\times(S,\Cat(\cK)))$.
    By the natural straightening equivalence $\Fun^\lax(-,\Cat) \simeq \Cocart^\lax(-)$
    recalled in \cref{thm:lax-str}
    and \cref{prop:lax-curry}, this is equivalent to the data of a lax natural transformation of the form (\ref{eq:lax-nat}), as desired.
\end{proof}

\section{Global Picard Spectra}\label{sec:pic}

In this section we will construct
the equivariant respectively global Picard spectrum of an ultracommutative ring spectrum
using the strategy outlined in the introduction.
We begin in \cref{subsec:pic} with a brief recollection
on Picard spectra and some useful lemmas,
and then move on to the parametrized global and equivariant
picture in \cref{subsec:param-pic} where we prove \cref{introthm:pic}
from the introduction.
Along the way, we will see increasingly stronger versions of the statement
that for questions of invertibility, it makes no difference
whether one works in $G$-equivariant or $G$-global stable homotopy theory.

\subsection{Picard Groups and Spectra}\label{subsec:pic}

Given a symmetric monoidal category $(\cC,\tensor,\one)$,
an object $x \in \cC$ is invertible if there exists $x^{-1} \in \cC$ such that
$x \tensor x^{-1} \simeq \one$. Note that such a $x^{-1}$, if it exists, is unique;
the space of inverses for $x$ is $(-1)$-truncated.
The Picard group of $\cC$ is then defined as the group of isomorphism classes of invertible elements
in $\cC$, with group operation induced by the tensor product.
It is often more convenient to consider the whole commutative group
of units in the underlying commutative monoid $\cC^\simeq$ of $\cC$,
as it is categorically more well behaved.
For example, because the following composite then preserves limits:
\begin{equation}\label{def:pic}
    \pic \colon \CMon(\Cat)
    \xrightarrow{(-)^\simeq_*} \CMon(\Spc)
    \xrightarrow{(-)^\times} \CGrp(\Spc)
    \simeq \Sp_{\geq 0}.
\end{equation}
If $\cC$ is a (large) accessible category, this definition still makes sense,
since if $\one$ is $\kappa$-compact for some $\kappa$ then so is any invertible object.
In other words, if $\CAT_\acc \subseteq \what{\Cat}$ denotes the full symmetric monoidal
subcategory on accessible categories, then we have an analogously defined
functor $\pic \colon \CMon(\CAT_\acc) \to \Sp_{\geq 0}$
and by further restriction $\pic \colon \CAlg(\Pr^L) \to \Sp_{\geq 0}$.

For a fixed $\cC \in \CAlg(\Pr^L)$ the association $R \mapsto \Mod_R(\cC)$ is functorial
in $R \in \CAlg(\cC)$ as we saw in \cref{sec:modules}, so by composition one obtains
\[
    \pic_\cC \colon \CAlg(\cC) \to \Sp,\ R \mapsto \pic_\cC(R) \coloneqq \pic(\Mod_R(\cC)).
\]
When it is clear from context,
we will leave out the subscript and simply write
$\pic R \coloneqq \pic_\cC R$.
For $\cC = \Sp$ and $R \in \CAlg(\Sp)$ one easily checks that $\pi_n\pic(R)$
is given by $(\pi_0R)^\times$ for $n = 1$ and by $\pi_{n-1}(R)$ for $n \geq 2$.
So all the interesting information lies in the Picard group $\pi_0 \pic R$,
and one thinks of $\pic R$ as an interesting delooping of $\GL_1(R) = (\Omega^\infty R) \times_{\pi_0 R} (\pi_0 R)^\times$.
An easy calculation shows $\pi_0\pic\S = \Z$,
but for the classical examples $\pi_0\pic\KU \cong \Z/2$ and $\pi_0\pic \KO \cong \Z/8$
some more work is required, see e.g.~\cite[Section 7]{Mathew-Stojanoska}.

In equivariant stable homotopy theory, the situation gets even more complicated.
Consider a finite group $G$.
Since the geometric fixed point functors are all symmetric monoidal and hence preserve invertible
objects, it follows from $\pi_0\pic(\S)= \Z$ that if $X \in \Sp_G$ is an invertible
genuine $G$-spectrum, then $\Phi^HX$ is a shifted sphere for every closed subgroup $H \leq G$.
Despite this, the Picard group of the category of genuine $G$-spectra $\Sp_G$ is generally quite complicated.
One way of producing elements in $\pi_0\pic \Sp_G$ is to consider representation spheres;
given a finite dimensional orthogonal $G$-representation $V$ its representation sphere $S^V$
becomes invertible in $\Sp_G$.\footnote{In fact,
    it was shown in \cite[Corollary C.7]{Gepner-Meier} that this is essentially the defining property of genuine $G$-spectra;
one can define $\Sp_G$ as the symmetric monoidal localization $\An_*^G[\{S^V\}^{-1}]$
of pointed $G$-spaces at all representation spheres,
and $\Sigma^\infty \colon \An_*^G \to \Sp_G$ is the unique symmetric monoidal left adjoint
which inverts representation spheres.}
In this way one obtains a map $\RO(G) \to \pi_0\pic\Sp_G, V \mapsto S^V$,
but this is in general neither injective nor surjective.\footnote{For example, Stefan Schwede presents two examples disproving injectivity respectively
surjectivity in his lecture \cite{Schwede-Video} starting at minute 38.}
Because of this map, invertible $G$-spectra are often called (stable) homotopy representations.
There is much more to be said about invertible genuine $G$-spectra, but this will take us too far afield.
We refer the interested reader to \cite{GPic} and \cite{Krause}.
Surprisingly, the global picture is much simpler.

\begin{example}[{{\cite[Theorem 4.5.5]{Global}}}]\label{ex:pic-gl}
    $\pi_0\pic(\S_\gl) = \{\Sigma^n\S_\gl \mid n \in \Z\}$.
\end{example}

This is a direct consequence of the following general lemma
applied to the counit $\eps \colon i_!i^* \Rightarrow \id$
for $i_! \colon \Sp \rightleftarrows \Sp_\gl \noloc i^*$ the symmetric monoidal
adjunction from \cref{prop:e-lax-right-adjoint} for $G=1$.

\begin{lemma}[{{\cite[4.5.4]{Global}}}]\label{lem:pic-trick}
    Let $(\cC,\tensor,\one)$ be a symmetric monoidal category,
    $P \colon \cC \to \cC$ a symmetric monoidal functor,
    and $\eps \colon P \Rightarrow \id_\cC$ a symmetric monoidal natural transformation.
    Then $\eps$ induces a homotopy of the map of Picard spectra
    \[
        \pic(P) \colon \pic\cC \to \pic\cC
    \]
    to the identity. The dual statement with a symmetric monoidal
    transformation $\eta \colon \id_\cC \Rightarrow P$
    in place of $\eps$ holds as well.
\end{lemma}
\begin{proof}
    Because $\eps$ is a symmetric monoidal natural transformation, we have commutative squares
    \[\begin{tikzcd}[cramped]
        {P(X) \tensor P(Y)} & {X \tensor Y} \\
        {P(X \tensor Y)} & {X \tensor Y}
        \arrow["{\eps_X \tensor \eps_Y}", from=1-1, to=1-2]
        \arrow["\simeq"', from=1-1, to=2-1]
        \arrow[Rightarrow, no head, from=1-2, to=2-2]
        \arrow["{\eps_{X \tensor Y}}"', from=2-1, to=2-2]
    \end{tikzcd}\]
    where the vertical maps come from the symmetric monoidal structures on $P$ respectively $\id_\cC$.
    Note in particular that $\eps_\one$ is an equivalence.
    Hence if $X \in \cC$ is invertible with inverse $X^{-1}$, then the map
    \[
        \eps_{X} \tensor \eps_{X^{-1}}
        \simeq (X \circ \eps_{X^{-1}}) \circ (\eps_X \tensor PX^{-1})
        \simeq (\eps_X \tensor X^{-1}) \circ (PX \tensor \eps_{X^{-1}})
    \]
    is an equivalence. Since $P(X^{-1})$ is also an inverse for $PX$,
    it follows that $\eps_X$ admits both a left and a right inverse, hence is an equivalence.
    This proves that if $X$ is invertible, then the component $\eps_X$ is an equivalence.
    In particular, $\eps$ restricts to a natural equivalence $\pic(P) = P|_{\pic(\cC)} \simeq \id_{\pic(\cC)}$,
    as desired.
\end{proof}

\begin{corollary}\label{cor:adj-pic}
    Let $L \colon \cC \rightleftarrows \cD \noloc R$ be a symmetric monoidal adjunction
    and suppose the induced lax symmetric monoidal
    structure on $R$ is strong\footnote{
        Generally, (oplax) symmetric monoidal structures on left adjoints
        canonically give rise to lax symmetric monoidal structures on the corresponding
        right adjoints and vice versa, by \cite[Theorem A]{HHLN-Lax}.
    }.
    Then $L$ and $R$ induce mutually inverse equivalences
    $\pic(\cC) \simeq \pic(\cD)$.
\end{corollary}

\begin{corollary}
    Let $L \colon \cC \to \cD$ be a functor in $\CAlg(\Pr^L)$ with right adjoint $R$.
    For $A \in \CAlg(\cD)$, the functor $R' \colon \Mod_A(\cD) \to \Mod_{RA}(\cC)$
    induced by $R$ admits a symmetric monoidal left adjoint $L'$ given by the composite
    \[
        \Mod_{RA}(\cC)
        \xto{L} \Mod_{LRA}(\cD)
        \xto{A \tensor_{LRA} -} \Mod_{A}(\cD).
    \]
    If the natural lax symmetric monoidal structure on $R$ is strong
    and $R$ preserves geometric realizations, then the natural lax symmetric
    monoidal structure on $R'$ is also strong.
    In particular, in this case we obtain an equivalence $\pic(\Mod_{RA}(\cC)) \simeq \pic(\Mod_A(\cD))$.
\end{corollary}
\begin{proof}
    Note that we have natural equivalences
    \[
        \Mod_{RA}(\cC)(RA \tensor c, R'M)
        \simeq \cC(c, RUM)
        \simeq \cD(Lc,UM)
        \simeq \Mod_A(\cD)(A \tensor Lc, M)
    \]
    where $U \colon \Mod_A(\cD) \to \cD$ is the forgetful functor.
    Since $L'(RA \tensor c) \simeq A \tensor Lc$,
    and the free modules $RA \tensor c$ generate $\Mod_{RA}(\cC)$ under colimits,
    this shows that $L' \dashv R'$.

    Note that $L'$ is symmetric monoidal since by \cref{thm:mod}
    (in the case $(\cF,\cN) = (\F,\F)$),
    $L$ induces a natural transformation
    $\Mod_{(-)}(\cC) \Rightarrow \Mod_{L(-)}(\cD)$
    of functors $\CAlg(\cC) \to \Mack(\Cat)$,
    and the first functor in $L'$ is the value of this natural transformation
    at $RA$, whereas the second functor is the image of $LRA \to A$
    under the functor $\Mod_{(-)}(\cD) \colon \CAlg(\cD) \to \Mack(\Cat)$.

    The symmetric monoidal structure on $L'$ automatically induces a lax symmetric monoidal
    structure on the right adjoint $R'$ by \cite[Theorem A]{HHLN-Lax}.
    By the dual of \cite[Proposition 3.2.7]{HHLN-Lax}, we can describe the structure maps as the composite
    \[
        RM \tensor_{RA} RN \xto{\eta} RL(RM \tensor_{RA} RN) \simeq R(LR \tensor_{LRA} LRN) \xto{R(\eps \tensor_{\eps} \eps)} R(M \tensor_A N).
    \]
    Recall (e.g.~from \cite[]{HA})
    that the relative tensor product is defined as the geometric realization of the Bar construction,
    which we will simply write as $M \tensor_A N \coloneqq |M \tensor A^{\tensor \bullet} N|$.
    A consequence of the naturality of $\eta$ and the lax symmetric monoidal structure on $LR$,
    we have a commutative diagram
    \[\hspace{-6.5em}\begin{tikzcd}[cramped,column sep=small]
        {RM \tensor_{RA} RN} & {RL(R M \tensor_{RA} RN)} & {R(LRM \tensor_{LRA} LRN)} & {R(M \tensor_A N)} \\
        {|RM \tensor (RA)^{\tensor \bullet}\tensor RN|} & {RL|RM \tensor (RA)^{\tensor \bullet}\tensor RN|} & {R|LRM \tensor(LRA)^{\tensor \bullet}\tensor LRN|} & {R|M \tensor A^{\tensor \bullet} \tensor N|} \\
        && {|RLRM \tensor (RLRA)^{\tensor \bullet}\tensor RLRN|} & {|RM \tensor (RA)^{\tensor \bullet} \tensor RN|}
        \arrow["\eta", from=1-1, to=1-2]
        \arrow["\simeq"', from=1-3, to=1-2]
        \arrow["{R(\eps \tensor_{\eps} \eps)}", from=1-3, to=1-4]
        \arrow[Rightarrow, no head, from=2-1, to=1-1]
        \arrow["\eta", from=2-1, to=2-2]
        \arrow["{|\eta_{RM} \tensor (\eta_{RA})^{\tensor \bullet}\tensor \eta_{RN}|}"', from=2-1, to=3-3]
        \arrow[Rightarrow, no head, from=2-2, to=1-2]
        \arrow[Rightarrow, no head, from=2-3, to=1-3]
        \arrow["\simeq"', from=2-3, to=2-2]
        \arrow["{R|\eps_M \tensor (\eps_A)^{\tensor \bullet}\tensor \eps_N|}"', shift right=2, draw=none, from=2-3, to=2-4]
        \arrow[from=2-3, to=2-4]
        \arrow[Rightarrow, no head, from=2-4, to=1-4]
        \arrow[from=3-3, to=2-3]
        \arrow["{|R\eps_M \tensor (R\eps_A)^{\tensor \bullet}\tensor R\eps_N|}"', shift right=2, draw=none, from=3-3, to=3-4]
        \arrow[from=3-3, to=3-4]
        \arrow[from=3-4, to=2-4]
    \end{tikzcd}\]
    where the unlabeled horizontal respectively vertical arrows are induced by the strong respectively
    lax symmetric monoidal structure maps followed by the colimit interchange with the geometric realization
    of $L$ respectively $R$.
    The additional hypotheses on $R$ thus guarantee that the vertical arrows are equivalences,
    and by the triangle identities it then follows that the lax symmetric monoidal structure map for $R'$ is an equivalence.
\end{proof}

\begin{example}\label{ex:pic-eq-gl}
    Let $G$ be a finite group, $R \in \CAlg(\Sp_G)$ and $S \in \CAlg(\Sp_{G\dgl})$.
    The adjunction $i_! \dashv i^*$ from \cref{prop:lquil} induces equivalences of Picard spectra
    \[
        (S \tensor_{i_!i^*S}-) \circ i_! \colon \pic(\Mod_{i^*S}(\Sp_G)) \simeq \pic(\Mod_S(\Sp_{G\dgl})) \noloc i^*.
    \]
    Picking $S = i_!R$ and identifying $R \simeq i^*i_!R$ we also get
    \[
        i_! \colon \pic(\Mod_{R}(\Sp_G)) \simeq \pic(\Mod_{i_!R}(\Sp_{G\dgl})) \noloc i^*
    \]
    In particular, $\pic(\Sp_G) \simeq \pic(\Sp_{G\dgl})$, generalizing \cref{ex:pic-gl}.
\end{example}
\begin{proof}
    This follows from \cref{prop:e-lax-right-adjoint} and the previous corollary.
\end{proof}

\subsection{Parametrized Picard Spectra}\label{subsec:param-pic}

In this subsection we finally construct the equivariant and global Picard spectra mentioned in the introduction.

\begin{definition}
    Let $(\cF,\cN)$ be an extensive span pair.
    Since $\pic \colon \Mack(\Cat) \to \Sp_{\geq 0}$
    preserves limits and $\Sp_{\geq 0} \subseteq \Sp$
    finite products, we can define the parametrized Picard spectrum functor
    by postcomposing with $\pic$:
    \[
        \und{\pic} \colon \und{\Mack}_\cF^\cN(\Cat)
        \simeq \und{\Mack}_\cF^\cN(\Mack(\Cat))
        \to \und{\Mack}_\cF^\cN(\Sp).
    \]
    Analogous to the non-parametrized version,
    given a fixed $\cN$-normed $\cF$-category $\cC$ compatible with geometric realizations,
    we also want to consider the Picard spectrum of module categories,
    and thus define
    \[
        \und{\pic}_\cC \colon \und{\CAlg}_\cF^\cN(\cC)
        \xrightarrow{\und{\Mod}_{(-)}(\cC)}
        \und{\Mack}_\cF^\cN(\Cat)
        \xrightarrow{\und{\pic}} \und{\Mack}_\cF^\cN(\Sp).
    \]
    In fact, we can even makes this lax natural in $\cC$ by composing the lax natural transformation
    from \cref{diag:mod-lax-nat} with the natural transformation of constant functors induced by $\und{\pic}$.
\end{definition}

\begin{example}
    Using the equivalences from \cref{thm:spectral-mack},
    we see that given a normed $G$-category $\cC$,
    the above definition yields a $G$-functor
    \[
        \und{\pic}_{\cC} \colon \und{\CAlg}_G(\cC) \to \und{\Sp}_G
    \]
    such that for $R \in \CAlg_G(\cC)$ we have $\pic_\cC(R)^H = \pic(\Mod_{R(G/H)}(\cC(G/H)))$.
    Similarly, if $\cD$ is a normed global category, we obtain a global functor
    \[
        \und{\pic}_{\cD} \colon \und{\CAlg}_\gl(\cD) \to \und{\Sp}_\gl.
    \]
\end{example}

It remains to use the comparison functors from Constructions \ref{con:ucomm-comp}, \ref{con:ucomm-G}
and \ref{con:ucomm-eq} to obtain equivariant and global Picard spectra of ultra-commutative ring spectra.

\begin{construction}\label{con:pic-g}
    Consider the normed $G$-category of $G$-spectra
    $\und{\Sp}_G^\tensor$ from \cref{def:gsym-sp}, which is compatible with
    sifted colimits by \cref{cor:sp-compat-sift}.
    We obtain a composite $G$-functor
    \[
        \und{\pic}_G \colon \und{\UComm}_G
        \xrightarrow{\Phi_G} \und{\CAlg}_G(\und{\Sp}_G^\tensor)
        \xrightarrow{\und{\pic}_{\und{\Sp}_G^\tensor}} \und{\Sp}_G.
    \]
    In level $G/H$, this is a functor
    \[
        \pic_H \colon \UCom_H \to \Sp_H
    \]
    that sends an ultra-commutative $H$-ring spectrum
    to the $H$-spectrum $\pic_H(R) \in \Sp_H$ with $\pic_H(R)^K \simeq \pic(\Mod_{\res^H_K R}(\Sp_K))$.
    Since $\und{\pic}_G$ is a $G$-functor, it then automatically follows
    that $\res^G_H \pic_G \simeq \pic_H \circ \res^G_H$.
\end{construction}

\begin{remark}
    A $G$-commutative Picard group (i.e.~an object in $\Mack_G(\CGrp(\An))$)
    of a $G$-symmetric monoidal category has previously
    been considered in \cite[Section 5.1]{HKZ},
    where it was used to construct a $G$-symmetric monoidal
    Thom spectrum functor
    $(\und{\An}_G)_{/\und{\pic}_G(\und{\Sp}^\tensor_G)} \to \und{\Sp}_G$.
    The work is done in the setting of \cite{Nardin-Shah},
    and their Picard $G$-space of a $G$-symmetric monoidal $G$-category $\und{\cC}^\tensor$
    is defined as the maximal $G$-subgroupoid on the invertible objects,
    with $G$-commutative monoid structure
    inherited from the $G$-symmetric monoidal structure.
    Here an object in $\und{\cC}$, i.e.~a $G$-functor $\const * \to \und{\cC}$
    adjoint to $* \to \Gamma\und{\cC} = \und{\cC}(G/e)$ is equivalently
    a section $x \colon \Orb_G^\op \to \int \und{\cC}$, and is called invertible
    if it is so fiberwise; this means for each $H \leq G$, the object $x(G/H) \in \und{\cC}(G/H)$
    is invertible in the symmetric monoidal structure on $\und{\cC}(G/H)$ defined by $\und{\cC}^\tensor$.
    Under the identification $\Mack(-) \simeq \CMon(-)$
    from \cref{rem:mack-vs-cmon}, this precisely agrees with our definition
    of postcomposing $\cC^\tensor \in \Mack_G(\CMon(\Cat))$ with
    $\pic \colon \CMon(\Cat) \xrightarrow{(-)^\simeq_*}
    \CMon(\Spc) \xrightarrow{(-)^\times} \CGrp(\Spc) \simeq \Sp_{\geq 0}$.
\end{remark}

\begin{construction}\label{con:pic-eqv}
    Recall that the parametrized module functor is actually defined on the larger source category
    \[
        \und{\Mod}_{(-)}(\und{\Sp}^\tensor)
        \colon \Fun_{/\Span_\Forb(\Fglo)}^{\pr\dcc}(\Span_\Forb(\Fglo_{/\bullet}),\smallint \und{\Sp}^\tensor)
        \to \und{\Mack}_{\gl}(\Cat).
    \]
    Moreover, by \cref{con:ucomm-eq} we also have a comparison functor
    \[
        \Phi_\eq \colon \und{\UCom}_\gl \to \und{\CAlg}_\eq(\und{\Sp}^\tensor)
        = \Fun_{/\Span_\Forb(\Fglo)}^{\Forb^\op\dcc}(\Span_\Forb(\Fglo_{/\bullet}), \smallint \und{\Sp}^\tensor).
    \]
    We thus obtain a globally functor assembling the Picard spectra
    of the equivariant ultra-commutative rings underlying a global ultra-commutative
    ring via:
    \[
        \und{\pic}_{\eq}
        \colon \und{\UCom}_\gl
        \xto{\Phi_\eq} \und{\CAlg}_\eq(\und{\Sp}^\tensor)
        \xto{\und{\Mod}_{(-)}(\und{\Sp}^\tensor)} \und{\Mack}_\gl(\Cat)
        \xto{\und{\pic}} \und{\Sp}_\gl.
    \]
    The underlying functor $\pic_\eq \colon \UComm_\gl \to \Sp_\gl$
    sends an ultra-commutative global ring spectrum $R$
    to a global spectrum $\pic_\eq(R)$ with genuine fixed points $\pic_\eq(R)^G \simeq \pic(\Mod_{\res_G R}(\Sp_G))$
    for $G \in \Glo$.
\end{construction}

The previous construction is precisely the description of the global Picard
spectrum of an ultra-commutative global ring spectrum which was conjectured
to exist in \cite[Remark 5.1.18]{Global}.
We can also consider $\und{\Sp}_\gl^\tensor$ in place of $\und{\Sp}^\tensor$.

\begin{construction}\label{con:pic-gl}
    We consider the normed global category of global spectra
    $\und{\Sp}_\gl^\tensor$ defined in \cref{con:sp-gl}.
    It is compatible with sifted colimits by \cref{cor:sp-compat-sift},
    so we may consider the following composite of global functors
    \[
        \und{\pic}_\gl \colon \und{\UComm}_\gl
        \xrightarrow{\Phi_\gl} \und{\CAlg}_\gl(\und{\Sp}_\gl^\tensor)
        \xrightarrow{\und{\pic}_{\und{\Sp}_\gl^\tensor}} \und{\Sp}_\gl,
    \]
    where the first functor was defined in \cref{con:ucomm-comp}.
    In level $G$, this is a functor
    \[
        \pic_{\gl} \colon \UComm_{G\dgl} \to \Sp_{G\dgl}
    \]
    which sends an ultracommutative $G$-global ring spectrum $R$
    to a $G$-global spectrum $\pic_{\gl}(R)$
    with genuine fixed points
    $\pic_{\gl}(R)^{\phi} \simeq \pic(\Mod_{\phi^*R}(\Sp_{K\dgl}))$
    for any group homomorphism $\phi \colon K \to G$.
\end{construction}

\begin{remark}
    Let us describe the genuine fixed points of the $G$-global Picard spectrum
    $\pic_\eq(R)$ for an ultracommutative $G$-global ring spectrum $R \in \UCom_{G\dgl}$.
    Recall that $\Phi_\eq \simeq i^* \circ \Phi_\gl$ where $i^*$ is the lax $\Span(\Forb)$-strong
    normed global right Bousfield localization $\und{\Sp}_\gl^\tensor \to \und{\Sp}^\tensor$
    from \cref{prop:e-lax-right-adjoint}.
    Now let $\phi \colon K \to G$ in $\Glo_{/G}$.
    By construction we have $\pic_\eq(R)^\phi = \pic(\Mod_{i^*((\Phi_\gl R)(\phi))}(\und{\Sp}_K))$,
    and by definition $(\Phi_\gl R)(\phi) = \phi^*R$, which we already used above.

    In more detail, since $\Phi_{\gl}R \colon \Span_\Forb(\Fglo_{/G}) \to \int \und{\Sp}_\gl^\tensor$
    is cocartesian on backwards maps, we note that it sends the span
    \[\begin{tikzcd}[cramped]
        G & K & K \\
        & G
        \arrow[Rightarrow, no head, from=1-1, to=2-2]
        \arrow["\phi"', from=1-2, to=1-1]
        \arrow[Rightarrow, no head, from=1-2, to=1-3]
        \arrow["\phi", from=1-2, to=2-2]
        \arrow["\phi", from=1-3, to=2-2]
    \end{tikzcd}\]
    to a cocartesian map. We therefore have $(\Phi_\gl R)(\phi) \simeq \phi^*R$
    where under minimal abuse of notation the right $R$
    also denotes the underlying commutative $G$-global ring spectrum
    $(\Phi_\gl R)(\id_G) \in \CAlg(\Sp_{G\dgl})$, and $\phi^* \colon \Sp_{G\dgl} \to \Sp_{K\dgl}$ is the restriction along $\phi$, i.e.~$\und{\Sp}_\gl^\tensor$ evaluated on the above span.
    Note that in the special case of $G=1$, this shows that
    \[
        \res_K R \coloneqq \Phi_\eq(R)(K) \simeq i^*\Phi_\gl(R)(K) \eqqcolon i^*\infl_K R
    \]
    for finite groups $K$ and $R \in \UCom_\gl$. Overall, this gives the formula
    \[
        \pic_\eq(R)^\phi \simeq \pic(\Mod_{i^*\phi^* R}(\Sp_K))
    \]
    Note that if $\phi$ lies in $\Forb$,
    we can commute it past $i^*$ since the latter is strong on $\Forb^\op$.
\end{remark}

In view of \cref{ex:pic-eq-gl}, this already hints at the following proposition,
which is our strongest form of the statement that $G$-equivariant and $G$-global stable homotopy
theory have the same invertible objects.

\begin{proposition}
    There is a natural equivalence of global functors
    \[
        i^* \colon \und{\pic}_\gl \simeq \und{\pic}_\eq.
    \]
    In level $G$, evaluated at some $R \in \UComm_{G\dgl}$ and on $\phi$-fixed points, this is
    \[
        i^* \colon \pic(\Mod_{\phi^*R}(\Sp_{K\dgl})) \xto{\simeq} \pic(\Mod_{i^*\phi^*R}(\Sp_K))
    \]
    from \cref{ex:pic-eq-gl}.
\end{proposition}
\begin{proof}
    We have a diagram where the left triangle commutes and the right one
    is obtained from \cref{rem:mod-functoriality}(iv):
    \[\begin{tikzcd}[cramped]
        & {\und{\CAlg}_\gl(\und{\Sp}_\gl^\tensor)} && {\und{\Mack}_\gl(\Cat)} \\
        {\und{\UComm}_\gl} & {\und{\CAlg}_\eq(\und{\Sp}^\tensor)} && {\und{\Mack}_\gl(\Cat)} & {\und{\Sp}_\gl}
        \arrow["{\und{\Mod}_{(-)}(\und{\Sp}_\gl^\tensor)}", from=1-2, to=1-4]
        \arrow["{i^*}"', from=1-2, to=2-2]
        \arrow["{i^*}"{description}, shorten <=5pt, shorten >=5pt, Rightarrow, from=1-4, to=2-2]
        \arrow[Rightarrow, no head, from=1-4, to=2-4]
        \arrow["{\und{\pic}}", from=1-4, to=2-5]
        \arrow["{\Phi_\gl}", from=2-1, to=1-2]
        \arrow["{\Phi_\eq}"', from=2-1, to=2-2]
        \arrow["{\und{\Mod}_{(-)}(\und{\Sp}^\tensor)}"', from=2-2, to=2-4]
        \arrow["{\und{\pic}}"', from=2-4, to=2-5]
    \end{tikzcd}\]
    By \cref{rem:mod-functoriality} it is evident that the composite natural transformation evaluated on $G,R,\phi$ is precisely as claimed, and thus an equivalence.
\end{proof}

Finally, we also have the following compatibility between $\pic_G$ and $\pic_\eq$.

\begin{lemma}\label{lem:pic-g-vs-eq}
    For a finite group $G$, there is an equivalence natural in $R \in \UCom_\gl$
    \[
        \pic_G\res_G R \simeq \res_G \pic_\eq(R).
    \]
\end{lemma}
\begin{proof}
    Let $i \colon \Span(G) \to \Span_\Forb(\Fglo)$ denote the map
    \mbox{induced by $\F_G \simeq \FOrb_{/G} \to \Forb \subset \Fglo$.}
    Since $i$ is a map of semiadditive categories,
    we can use the naturality from \cref{rem:mod-functoriality}(iii)
    to observe that restriction along $i$ yields a commutative diagram
    \[\begin{tikzcd}[cramped,sep=scriptsize]
        {\CAlg_\eq(\und{\Sp}^\tensor)} & {\Fun^{\Forb^\op\dcc}_{/\Span_\Forb(\Fglo)}(\Span_\Forb(\Fglo), \int\und{\Sp}^\tensor)} &&& {\Mack_\gl(\Cat)} & {\Sp_\gl} \\
        {\CAlg_G(\und{\Sp}_G)} & {\Fun^{\Forb^\op\dcc}_{/\Span_\Forb(\Fglo)}(\Span(G), \int\und{\Sp}^\tensor)} &&& {\Mack_G(\Cat)} & {\Sp_G}
        \arrow[Rightarrow, no head, from=1-1, to=1-2]
        \arrow["{i^*}"', from=1-1, to=2-1]
        \arrow["{\und{\Mod}_{(-)}(\und{\Sp}^\tensor)}", from=1-2, to=1-5]
        \arrow["{i^*}"', from=1-2, to=2-2]
        \arrow["\pic", from=1-5, to=1-6]
        \arrow["{i^*}", from=1-5, to=2-5]
        \arrow["{\res_G}", from=1-6, to=2-6]
        \arrow["\simeq", from=2-1, to=2-2]
        \arrow["{\und{\Mod}_{(-)}(\und{\Sp}_G^\tensor)}"', curve={height=18pt}, from=2-1, to=2-5]
        \arrow["{\und{\Mod}_{(-)}(\und{\Sp}^\tensor)}", from=2-2, to=2-5]
        \arrow["\pic", from=2-5, to=2-6]
    \end{tikzcd}\]
    where the bottom left equivalences follows by adjunction
    since $\und{\Sp}_G^\tensor = \und{\Sp}^\tensor|_{\Span(G)}$.
    For compatibility of $i^*$ with the comparison functors $\Phi_G$ and $\Phi_\eq$,
    we consider the following diagram
    \[\begin{tikzcd}[cramped,sep=scriptsize]
        {\CAlg(\Sp^{\Sigma,\tensor}_{\cflat})} & {\UCom_\gl} & {\UCom_G} & {\CAlg(G\Sp^{\Sigma,\tensor}_\cflat)} \\
        {\CAlg_\gl((\Sp^{\Sigma,\tensor}_\cflat)^\flat)} & {\CAlg_\eq(\und{\Sp}^\tensor)} & {\CAlg_G(\und{\Sp}_G^\tensor)} & {\CAlg_G((\infl_G\Sp^{\Sigma,\tensor}_\cflat)^\flat)}
        \arrow["{\gamma_\gl}"', from=1-1, to=1-2]
        \arrow["{(\infl_G)_*}", curve={height=-18pt}, from=1-1, to=1-4]
        \arrow["{\res_G}"', dashed, from=1-2, to=1-3]
        \arrow["{\Phi_\eq}"', dashed, from=1-2, to=2-2]
        \arrow["{\Phi_G}", dashed, from=1-3, to=2-3]
        \arrow["{\gamma_G}", from=1-4, to=1-3]
        \arrow["\simeq", from=2-1, to=1-1]
        \arrow["{(\gamma_\eq)_*}", from=2-1, to=2-2]
        \arrow["{i^*}"', curve={height=18pt}, from=2-1, to=2-4]
        \arrow["{i^*}", from=2-2, to=2-3]
        \arrow["\simeq"', from=2-4, to=1-4]
        \arrow["{(\gamma_G)_*}"', from=2-4, to=2-3]
    \end{tikzcd}\]
    Here all dashed maps are induced by the universal property of Dwyer--Kan localization,
c.f.~Constructions \ref{con:ucomm-G} and \ref{con:ucomm-eq}.
    The left and right squares commute by definition.
    For commutativity of the middle square, which is the one we are interested in, it suffices by the universal property of Dwyer--Kan localization
    to check that the outer rectangle commutes,
    which we did in \cref{cor:calg-comm}.
\end{proof}

\appendix
\section{Monoids in the Universal Cocartesian Fibration}\label{appendix:cocart}

In this appendix, we show that applying $\Mon(-)$ respectively $\Mack(-)$
to the universal cocartesian fibration $\Cat_{*\sslash} \to \Cat$
gives models for the cocartesian fibrations classifying $\Alg(-) \colon \Mon(\Cat) \to \Cat$
respectively $\CAlg(-) \colon \Mack(\Cat)\to \Cat$ (c.f.~\cref{rem:mack-vs-cmon}).
This is used in \cref{sec:modules} to identify Lurie's $\Cat^{\Alg}$ with $\Mon(\Cat_{*\sslash})$, which allows us to build our parametrized module categories.

\begin{definition}
    We say a cocartesian fibration $p \colon \cE \to \cC$ is compatible with finite products
    if $\cE,\cC$ admit finite products, $p$ preserves finite products, and projections in $\cE$ are $p$-cocartesian.
\end{definition}

\begin{proposition}[{{\cite[Lemma 2.48, Corollary 2.49]{LLP}}}]\label{prop:cocart-prod}
    Let $\cC$ be a category with finite products, and consider a functor $F \colon \cC \to \Cat$.
    Then $F$ preserves finite products if and only if its cocartesian unstraightening is compatible
    with finite products.
    In this case, a morphism $X \to Y \times Z$ in $\int \cC$ is cocartesian if and only if both $X \to Y$ and $X \to Z$ are.
\end{proposition}

The following observations are immediate.

\begin{lemma}\label{lem:obvious}
    Let $p \colon \cE \to \cC$ be a cocartesian fibration compatible with finite products.
    \begin{enumerate}
        \item Any basechange of $p$ along a finite product preserving functor $\cC' \to \cC$
            is again a cocartesian fibration compatible with finite products.

        \item A section $s \colon \cC \to \cE$ of $p$ is cocartesian on projections if and only if it preserves
            finite products.
    \end{enumerate}
\end{lemma}

\begin{proposition}\label{prop:fun-univ-cocart}
    Let $p \colon \Cat_{*\sslash} \to \Cat$ denote the universal cocartesian fibration,
    i.e.~the cocartesian unstraightening of the identity functor.
    If $T$ is a category, then the cocartesian fibration
    $
        p_* \colon \Fun(T,\Cat_{*\sslash}) \to \Fun(T,\Cat)
    $
    given by postcomposition with $p$ classifies the lax limit functor
    \[
        \Fun_{/T}(T,\smallint -) \colon \Fun(T,\Cat) \to \Cat
    \]
\end{proposition}
\begin{proof}
    Recall that if $p'$ is any cocartesian fibration,
    then $p'^\op$ is the cartesian fibration classifying
    $(-)^\op \circ \St^\cocart(p')$.
    Hence $q \coloneqq p^\op$ is the cartesian fibration classifying
    the equivalence $(-)^\op \colon \Cat \to \Cat$.
    It follows from \cite[Proposition 7.1, Proposition 8.1]{GHN}
    that $q_* \colon \Fun(T^\op,(\Cat_{*\sslash})^\op) \to \Fun(T^\op,\Cat^\op) \simeq \Fun(T,\Cat)^\op$
    is a cartesian fibration classifying the functor
    \[
        \Fun(T,\Cat) \to \Cat,\
        \phi \mapsto \Fun_{/T^\op}(T^\op, \Un^\cart(\phi)).
    \]
    Postcomposing this with $(-)^\op$,
    we get the claimed description of the cocartesian unstraightening
    of $p_* = (q_*)^\op$.
\end{proof}

Another more obvious description of the fiber of $p_*$ over $F \colon T \to \Cat$
is $\Fun_{/\Cat}(T,\Cat_{*\sslash})$. This is equivalent to $\Fun_{/T}(T,\int F)$
via the $\Cat$-linear adjunction of 2-categories
$F_! \colon \what{\Cat}_{/T} \rightleftarrows \what{\Cat}_{/\Cat} \noloc F^*$,
and we can think of the above proposition as making this equivalence natural in $F$.

\begin{proposition}\label{prop:univ-cocart-prod}
    The cocartesian fibration $p \colon \Cat_{*\sslash} \to \Cat$
    is compatible with finite products.
    Moreover, if $T$ is a category admitting finite products, then
    \[\begin{tikzcd}[cramped]
        {\Fun^\times(T,\Cat_{*\sslash})} & {\Fun(T,\Cat_{*\sslash})} \\
        {\Fun^\times(T,\Cat)} & {\Fun(T,\Cat)}
        \arrow[hook, from=1-1, to=1-2]
        \arrow["{p_*}"', from=1-1, to=2-1]
        \arrow["{p_*}", from=1-2, to=2-2]
        \arrow[hook, from=2-1, to=2-2]
    \end{tikzcd}\]
    is a map of cocartesian fibrations classifying $\Fun^\times_{/T}(T,\int -) \subseteq \Fun_{/T}(T,\int -)$.
\end{proposition}
\begin{proof}
    Since $\id_{\Cat}$ clearly preserves finite products,
    $p$ is compatible with finite products by \cref{prop:cocart-prod}.
    Moreover, the right $p_*$ is a cocartesian fibration by the previous proposition.
    Taking fibers over $F \in \Fun^\times(T,\Cat)$ in the above square
    gives the fully faithful inclusion $\Fun_{/T}^\times(T,\int F) \subseteq \Fun_{/T}(T,\int F)$, since $p$ and $F$ in the following pullback square preserve finite products,
    so a section $T \to \int F$ preserves finite products
    if and only if $T \to \int F \to \Cat_{*\sslash}$ does:
    \[\begin{tikzcd}[cramped]
        {\int F} & {\Cat_{*\sslash}} \\
        T & \Cat
        \arrow[from=1-1, to=1-2]
        \arrow[from=1-1, to=2-1]
        \arrow["\lrcorner"{anchor=center, pos=0.125}, draw=none, from=1-1, to=2-2]
        \arrow["p", from=1-2, to=2-2]
        \arrow[curve={height=-12pt}, dashed, from=2-1, to=1-1]
        \arrow[dashed, from=2-1, to=1-2]
        \arrow["F"', from=2-1, to=2-2]
    \end{tikzcd}\]
    It remains to see that these full subcategories
    are closed under the cocartesian pushforwards of $p_* \colon \Fun(T,\Cat_{*\sslash}) \to \Fun(T,\Cat)$.
    To this end, it suffices to check that if
    $F \colon T \to \Cat_{*\sslash}$ preserves finite products
    and $\alpha \colon pF \Rightarrow G$ is some map in $\Fun^\times(T,\Cat)$,
    then the cocartesian pushforward $\alpha_!F \colon T \to \Cat_{*\sslash}$ again preserves finite products.

    Since we already know that $p\alpha_!F \simeq G$ preserves finite products,
    the fact that $p$ is compatible with finite products tells us that it suffices
    to check that $\alpha_!F$ sends projections in $T$ to $p$-cocartesian morphisms.
    So let $s,t \in T$ and consider $\pr_s \colon s \times t \to s$.
    Since $F \Rightarrow \alpha_!F$ is pointwise cocartesian, we know that $\alpha_!F(\pr_s)$
    is the unique morphism making the following square commute:
    \[\begin{tikzcd}[cramped]
        {F(s \times t)} & {\alpha_!F(s \times t)} \\
        {F(s)} & {\alpha_!F(s)}
        \arrow["\cc", from=1-1, to=1-2]
        \arrow["{F(\pr_s)}"', from=1-1, to=2-1]
        \arrow["{\alpha_!F(\pr_s)}", from=1-2, to=2-2]
        \arrow["\cc"', from=2-1, to=2-2]
    \end{tikzcd}\]
    where the horizontal morphisms are cocartesian.
    Since $F$ preserves finite products, we see that the left vertical map is also cocartesian,
    and thus right-cancellability of cocartesian morphisms implies that also the right vertical map
    is cocartesian, as desired.
\end{proof}

For a category with finite products $\cE$, we denote by $\Mon(\cE) \subseteq \Fun(\Delta^\op,\cE)$
the full subcategory on monoids in $\cE$ in the sense of \cite[Definition 4.1.2.5]{HA}.
Concretely, a functor $X \colon \Delta^\op \to \cE$ is a monoid if it satisfies the Segal condition,
that for each $[n] \in \Delta^\op$ the face maps $d_i \colon [1] = \{i-1,i\} \to [n]$
induce an equivalence $X_n \simeq \prod_{i=1}^n X_1$ (in particular $X_0 = *$).
The composite functor
\[
    \Delta^\op \xto{\text{Cut}} \Fin_* \simeq \Span_{\inj,\all}(\F) \subset \Span(\F)
\]
induces via restriction a forgetful functor $\fgt \colon \Mack(\cE) \simeq \CMon(\cE) \xto{\textup{Cut}^*} \Mon(\cE)$;
here Cut and its induced forgetful functor are discussed in \cite[Definition 2.4.2.1, Proposition 4.1.2.10]{HA},
and the composite $\Fin_* \simeq \Span_{\inj,\all}(\F) \subset \Span(\F)$ and induced equivalence $\Mack(\cE) \simeq \CMon(\cE)$ is \cite[Proposition C.1]{Bachmann-Hoyois}.

\begin{proposition}\label{prop:mack-univ-calg}
    Let $p \colon \Cat_{*\sslash} \to \Cat$ denote the universal cocartesian fibration,
    and let $\cK$ be a class of small sifted categories.
    Denote by $p^\cK \colon \Cat_{*\sslash}(\cK) \to \Cat(\cK)$ the pullback of $p$ along the subcategory
    inclusion $\Cat(\cK) \subset \Cat$.
    Consider the commutative square
    \[\begin{tikzcd}[cramped]
        {\Mack(\Cat_{*\sslash}(\cK))} & {\Mon(\Cat(\cK))} \\
        {\Mack(\Cat(\cK))} & {\Mon(\Cat(\cK))}
        \arrow["\fgt", from=1-1, to=1-2]
        \arrow["{\Mack(p^\cK)}"', from=1-1, to=2-1]
        \arrow["{\Mon(p^\cK)}", from=1-2, to=2-2]
        \arrow["\fgt", from=2-1, to=2-2]
    \end{tikzcd}\]
    Then both vertical maps are cocartesian fibrations classifying
    \[
        \CAlg(-) \colon \Mack(\Cat(\cK)) \to \Cat
        \quad\text{respectively}\quad
        \Alg(-) \colon \Mon(\Cat(\cK)) \to \Cat
    \]
    and the square gives a map of cocartesian fibrations classifying the forgetful
    transformation $\CAlg(-) \Rightarrow \Alg(-)$ of functors $\Mack(\Cat(\cK)) \to \Cat$.
\end{proposition}
\begin{proof}
    Since $\Mack(-)$ and $\Mon(-)$ preserve pullbacks,
    it suffices to show this for $\cK = \empty$.
    It follows from the previous proposition that $\Mack(p) = \Fun^\times(\Span(\F),p)$
    is a cocartesian fibration classifying $\Fun^\times_{/\Span(\F)}(\Span(\F),\int -)$.
    We claim that this functor is equivalent to $\CAlg(-) \coloneqq \Fun_{/\Span(\F)}^{\F^\op\dcc}(\Span(\F),\int -)$.
    In view of \cref{lem:obvious}, it remains to check that
    for $\cC \in \Mack(\Cat)$ with unstraightening $q \colon \int \cC \to \Span(\F)$
    and a section $s$ of $q$, if $s$ is cocartesian on projections then it is already cocartesian on all of $\F^\op$.
    Recall from \cref{prop:span-semiadd}
    that for $n,m \in \F$ the backwards inclusions $n \sqcup m \leftarrow n = n$
    and $n \sqcup m \leftarrow m = m$ exhibit $n \sqcup m$ as product of $n$ and $m$ in $\Span(\F)$.
    By \cref{prop:cocart-prod} we can check that $s(m \leftarrow n)$ is cocartesian after postcomposing
    with a cocartesian lift of the projection $\pr_i \colon n \xleftarrow{i} 1$ for each $i=1,\dots,n$
    (note also that it is clear for $m = \empty$).
    Since $s$ is cocartesian on projections, we may pick $s(n \xleftarrow{i} 1)$ to be this cocartesian lift,
    and note that the resulting composite $s(m \leftarrow n \xleftarrow{i} 1)$ is cocartesian
    as it is again a projection.

    Similarly, it follows from \cref{prop:fun-univ-cocart} that $\Fun(\Delta^\op,p)$
    is a cocartesian fibration classifying $\Fun_{/\Delta^\op}(\Delta^\op, \int -)$.
    The same argument as in \cref{prop:univ-cocart-prod} using the Segal morphisms in $\Delta^\op$
    instead of arbitrary projections shows that this restricts to a cocartesian
    fibration $\Mon(p)$ classifying the subfunctor $\Fun^\textup{Seg}_{/\Delta^\op}(\Delta^\op,-)$
    on those sections which send the Segal morphisms to projections, i.e.~which are cocartesian on the Segal maps.
    Recall (e.g.~from \cite[Proposition 4.1.3.19]{HA}) that if $\cC \in \Mon(\Cat)$,
    then $\Alg(\cC) \simeq \Fun_{/\Delta^\op}^{\inert\dcc}(\Delta^\op,\int \cC)$
    where inert maps in $\Delta$ are inclusions of intervals.
    So as above, it remains to verify that for $\cC \in \Mon(\Cat)$ with unstraightening
    $q \colon \int \cC \to \Delta^\op$ and a section $s$ of $q$, if $s$ is cocartesian on Segal maps,
    then it is already cocartesian on all inert maps.
    But note that if $\{k,k+1,\dots,\ell\} \subseteq [n]$ is the inclusion of some interval,
    then for any Segal map $\{i,i+1\} \subseteq \{k,k+1,\dots,\ell\}$ also the composition
    to $[n]$ is a Segal map, so we conclude with the same argument as before.

    Finally, we need to show that the square induces a map of cocartesian fibrations,
    i.e.~that the top horizontal map sends $\Mack(p)$-cocartesian transformations
    to $\Mon(p)$-cocartesian ones.
    But this is clear, as in both cases they are given precisely by those transformations which are pointwise $p$-cocartesian.
\end{proof}

\section{2-Natural and Lax Natural Transformations}\label{appendix:lax}

The purpose of this appendix is twofold;
we begin with a short reminder on 2-natural transformations and the 2-Yoneda lemma
which we need for the uniqueness result \cref{prop:unique-norm}.
We then recall the notion of lax natural transformations,
and prove two lemmas needed in the proof of \cref{thm:mod}.

The category $\Cat_2$ of 2-categories is cartesian closed,
and admits an internal hom $\mathbf{Fun}_2(-,-)$,
the 2-category of 2-functors.

\begin{definition}
    Let $\mathbf{A},\mathbf{B}$ be 2-categories and $F,G \colon \mathbf{A} \to \mathbf{B}$ two 2-functors. We define the category $\Nat_2(F,G)$ of 2-natural transformations from $F$ to $G$ as the hom-category $\hom_{\mathbf{Fun}_2(\mathbf{A},\mathbf{B})}(F,G)$
    inside the 2-category of 2-functors $\mathbf{Fun}_2(\mathbf{A},\mathbf{B})$.
\end{definition}

In other words, a 2-natural transformation is simply a 1-morphism
in the category $\Fun_2(\mathbf{A},\mathbf{B})$ of 2-functors,
i.e.~a functor $\Delta^1 \to \Fun_2(\mathbf{A},\mathbf{B})$.
Now we have adjunctions
\[
    \Cat(\Delta^1,\Fun_2(\mathbf{A},\mathbf{B}))
    \simeq \Cat_2(\Delta^1,\mathbf{Fun}_2(\mathbf{A},\mathbf{B}))
    \simeq \Cat_2(\mathbf{A}, \mathbf{Ar}(\mathbf{B}))
\]
so we can equivalently think of a 2-natural transformation
as a 2-functor $\mathbf{A} \to \mathbf{Ar}(\mathbf{B})$.
In the main text, we will mostly care about $\mathbf{B} = \Cat$,
and we will simply write $\Ar(\Cat) = \Fun(\Delta^1,\Cat)$
and understand it as inheriting its 2-categorical structure from $\Cat$,
as in \cref{rem:2cat}.
In this special case of $\mathbf{B} = \Cat$, the data of a 2-natural transformation
$\alpha \colon F \Rightarrow G$ essentially consists of:
\begin{itemize}
    \item a functor $\alpha_a \colon Fa \to Ga$ for every object $a \in \mathbf{A}$;
    \item a natural equivalence $\sigma_f \colon (Gf)\alpha_a \simeq \alpha_bFf$
        for every map $f \colon a \to b$ in $\mathbf{A}$;
    \item for every 2-morphism $\phi \colon f \Rightarrow f'$ in $\mathbf{A}$
        a 3-cell filling the following cylinder
        \[\begin{tikzcd}[cramped]
            Fa && Ga \\
            \\
            Fb && Gb
            \arrow["{\alpha_a}", from=1-1, to=1-3]
            \arrow[""{name=0, anchor=center, inner sep=0}, "{Ff'}"{description}, curve={height=-18pt}, from=1-1, to=3-1]
            \arrow[""{name=1, anchor=center, inner sep=0}, "Ff"{description}, curve={height=18pt}, from=1-1, to=3-1]
            \arrow[""{name=2, anchor=center, inner sep=0}, "{Gf'}"{description}, curve={height=-18pt}, from=1-3, to=3-3]
            \arrow[""{name=3, anchor=center, inner sep=0}, "Gf"{description}, curve={height=18pt}, from=1-3, to=3-3]
            \arrow["{\alpha_b}", from=3-1, to=3-3]
            \arrow["{F\phi}", shorten <=7pt, shorten >=7pt, Rightarrow, from=1, to=0]
            \arrow["{G\phi}", shorten <=7pt, shorten >=7pt, Rightarrow, from=3, to=2]
        \end{tikzcd}\]
        where the front face is filled by $\sigma_f$ and the back face by $\sigma_{f'}$;

    \item further coherence data.
\end{itemize}

We will need the following version of the 2-categorical Yoneda lemma.

\begin{theorem}\label{thm:2-yoneda}
    Let $\mathbf{A}$ be a (large) 2-category, $X \in \mathbf{A}$ an object,
    and $F \colon \mathbf{A} \to \CAT$ a 2-functor.
    Then evaluation at $\id_X$ induces an equivalence
    \[
        \Nat_2(\hom_{\mathbf{A}}(X,-),F) \xto{\simeq} F(X).
    \]
\end{theorem}
\begin{proof}
    As briefly mentioned in \cref{rem:2cat},
    2-categories are by definition $\Cat$-enriched categories.
    The 2-categorical Yoneda lemma is then a special case of the general
    enriched Yoneda lemma proven in \cite[Theorem 6.2.7]{Hinich}.
    We have taken the above reformulation of it from the proof of \cite[Proposition 13]{Ben-Moshe}.
\end{proof}

We now give a recollection on lax natural transformations.
For an extensive treatment, we refer the reader to \cite{AGH-Lax},
although we will only need the case where our domain category is actually a category as opposed to a general 2-category.

So let $\cC$ be a category (which we can view as a 2-category with only invertible 2-morphisms)
and $\mathbf{D}$ a 2-category.
Given two functors $F,G \colon \cC \to \mathbf{D}$ of 2-categories,
a lax natural transformation $\alpha \colon F \Rightarrow G$ essentially consists of a collection of morphisms
$\alpha_c \colon F(c) \to G(c)$ for each $c \in \cC$, together with lax naturality squares
\[\begin{tikzcd}[cramped]
	{F(c)} & {G(c)} \\
	{F(d)} & {G(d)}
	\arrow["{\alpha_c}", from=1-1, to=1-2]
	\arrow["Ff"', from=1-1, to=2-1]
	\arrow["{\alpha_f}"{description}, Rightarrow, from=1-2, to=2-1]
	\arrow["Gf", from=1-2, to=2-2]
	\arrow["{\alpha_d}"', from=2-1, to=2-2]
\end{tikzcd}\]
where $\alpha_f \colon G(f)\alpha_c \Rightarrow \alpha_dF(f)$ is a 2-morphism in $\mathbf{D}$.
As usual, the data of a lax natural transformation also incorporates an infinite amount of further coherences.
Formally, one can define lax natural transformations using the so-called Gray tensor product of 2-categories $\boxtimes \colon \Cat_2 \times \Cat_2 \to \Cat_2$ (see \cite[Section 2.2]{AGH-Lax}).
Specifically, one defines the category $\Fun^\lax(\cC,\mathbf{D})$ with objects the functors $\cC \to \mathbf{D}$ and morphisms the lax natural transformations via a natural adjunction equivalence
\[
    \map_{\Cat}(K,\Fun^\lax(\cC,\mathbf{D})) \simeq \map_{\Cat_2}(K \boxtimes \cC, \mathbf{D}).
\]
In the special case where $\mathbf{D} = \Cat$ is the 2-category of 1-categories,
there is a very convenient fibrational model available.

\begin{definition}
    For a category $\cC$, we denote by $\Cocart^\text{lax}(\cC) \subseteq \Cat_{/\cC}$ the \emph{full} subcategory on cocartesian fibrations. We refer to morphisms in this category as lax maps of cocartesian fibrations.
\end{definition}

\begin{theorem}[{{\cite[Theorem E]{HHLN-Lax}}}]\label{thm:lax-str}
    There is a natural (via pullback respectively precomposition) equivalence
    \[
        \Cocart^\lax(\cC) \simeq \Fun^\lax(\cC,\Cat)
    \]
    given on objects by cocartesian straightening.
\end{theorem}

An immediate consequence of this theorem is the following construction.

\begin{construction}\label{lem:cocart-lax-square}
    Let $p \colon \cE \to \cC$ be a cocartesian fibration.
    Then the lax map of cocartesian fibrations
    \[\begin{tikzcd}[cramped]
        \cE && {\cE \times \cC} \\
        & \cC
        \arrow["{(\id,p)}", from=1-1, to=1-3]
        \arrow["p"', from=1-1, to=2-2]
        \arrow["\pr", from=1-3, to=2-2]
    \end{tikzcd}\]
    classifies a lax natural transformation $\Str^\cocart(p) \xRightarrow{\text{lax}} \const \cE$
    which is pointwise given by the inclusion of fibers, with lax naturality squares
    \[\begin{tikzcd}
        {\cE_c} & \cE \\
        {\cE_d} & \cE
        \arrow["\inc_c", from=1-1, to=1-2]
        \arrow["{f_!}"', from=1-1, to=2-1]
        \arrow["\cc"{description}, Rightarrow, from=1-2, to=2-1]
        \arrow[Rightarrow, no head, from=1-2, to=2-2]
        \arrow["\inc_d"', from=2-1, to=2-2]
    \end{tikzcd}\]
    where the natural transformation is pointwise cocartesian.
    Note that the natural transformation is the unique one which is pointwise cocartesian,
    i.e.~is a cocartesian lift of $\const f$ for the cocartesian fibration
    $\Fun(\cE_c,\cE) \to \Fun(\cE_c,\cC)$.
\end{construction}

Using this, we can more generally understand where the lax naturality squares
of the lax natural transformation associated to a lax map of cocartesian fibrations
come from.
Namely, if $\alpha \colon \cE \to \cF$ is a map in $\Cocart^\lax(\cC)$
and $f \colon c \to d$ a morphism in $\cC$, straightening produces functors $f_! \colon \cE_c \to \cE_d$
and $f_! \colon \cF_c \to \cF_d$. Moreover, since $\alpha$ is a functor over $\cC$,
it restricts to $\alpha_c \colon \cE_c \to \cF_c$ and $\alpha_d \colon \cE_d \to \cF_d$.
Given $X \in \cE_c$, we can factor the image the cocartesian map $X \to f_!X$
under $\alpha$ as $\alpha_cX \xto{\cc} f_!\alpha_c X \to \alpha_df_!X$.
We claim that the maps $f_!\alpha_c X \to \alpha_d f_!X$ for varying $X$
assemble into a natural transformation
\[\begin{tikzcd}[cramped]
	{\cE_c} & {\cF_c} \\
	{\cE_d} & {\cF_d}
	\arrow["{\alpha_c}", from=1-1, to=1-2]
	\arrow["{f_!}"', from=1-1, to=2-1]
	\arrow["{\alpha_f}"{description}, Rightarrow, from=1-2, to=2-1]
	\arrow["{f_!}", from=1-2, to=2-2]
	\arrow["{\alpha_d}"', from=2-1, to=2-2]
\end{tikzcd}\]
Indeed, by whiskering the pointwise cocartesian natural transformation $\inc_c \Rightarrow \inc_df_!$
from the above construction with $\alpha$ we obtain a morphism $\alpha\inc_c \Rightarrow \alpha\inc_d f_!$
in $\Fun(\cE_c,\cF)$ lying over $\const f$.
Factoring this into its cocartesian and fiberwise part thus produces
\[\begin{tikzcd}[cramped]
	{\cE_c} && \cF \\
	& {\cF_d} \\
	{\cE_d} && \cF
	\arrow["{\alpha\inc_c \simeq \inc_c\alpha_c}", from=1-1, to=1-3]
	\arrow["{f_!\inc_c\alpha_c}"{description}, from=1-1, to=2-2]
	\arrow["{f_!}"', from=1-1, to=3-1]
	\arrow["\cc"{description}, Rightarrow, from=1-3, to=2-2]
	\arrow[Rightarrow, no head, from=1-3, to=3-3]
	\arrow["{\alpha_f}"{description}, Rightarrow, from=2-2, to=3-1]
	\arrow["{\inc_d}"{description}, from=2-2, to=3-3]
	\arrow["{\alpha\inc_d \simeq \inc_d \alpha_d}"', from=3-1, to=3-3]
\end{tikzcd}\]
and since the bottom composite lands in $\cF_d \subset \cF$, this yields the desired natural transformation.
Finally, we need the following lemma in the proof of \cref{thm:mod}.

\begin{proposition}\label{prop:lax-curry}
    Let $A,B,C$ be categories. Then a functor
    \[
        F \colon A \to \Fun^\lax(B,\Fun(C,\Cat))
    \]
    is naturally equivalent (under ``currying'') to the datum of a functor
    \[
        F^\sharp \colon C \times A \to \Fun^\lax(B,\Cat)
    \]
    so that for each $a \in A$, the restriction $F^\sharp(-,a)$ factors through $\Fun(B,\Cat)$.
\end{proposition}
\begin{proof}
    By definition of the Gray tensor product, we have equivalences of mapping spaces
    \[
        \Cat(A,\Fun^\lax(B,\Fun(C,\Cat)))
        \simeq \Cat_2(A \boxtimes B, \Fun(C,\Cat))
        \simeq \Cat_2(C \times (A \boxtimes B),\Cat)
    \]
    and
    \[
        \Cat(C \times A, \Fun^\lax(B,\Cat))
        \simeq \Cat_2((C \times A) \boxtimes B, \Cat)
    \]
    Now by \cite[Corollary 2.9.2]{AGH-Lax} there is a natural localization at 2-morphisms
    $(C \times A) \boxtimes B \to C \times (A \boxtimes B)$
    and a 2-functor $\Phi \colon (C \times A) \boxtimes B \to \Cat$ (uniquely) descends to a 2-functor
    $C \times (A \boxtimes B) \to \Cat$ precisely if $\Phi(-,a,-)$ is strict for each $a \in A$,
    which is precisely the condition in the statement.
\end{proof}

\bibliography{ref}

\end{document}